\documentclass[10pt]{article}
 \usepackage{amsmath, amsfonts, amsthm, amssymb, amscd, enumerate, url, slashed, stmaryrd, upgreek, thmtools, mathdots}
 \usepackage{float}
\usepackage{tikz}
\usetikzlibrary{matrix,arrows,decorations.pathmorphing}
 \usepackage{tikz-cd}
\begingroup
 \makeatletter
 \@for\theoremstyle:=definition,remark,plain\do{%
 \expandafter\g@addto@macro\csname th@\theoremstyle\endcsname{%
 \addtolength\thm@preskip\parskip
 }%
 }
 \endgroup
 
 \usepackage[titletoc,toc,title]{appendix}
\usepackage[all]{xy}
\usepackage{tikz-cd}
 
 \usepackage[colorlinks = true,
 linkcolor = blue,
 urlcolor = cyan,
 citecolor = red,
 anchorcolor = blue,
 pagebackref]{hyperref}
 
\usepackage[alphabetic,backrefs]{amsrefs}
 
 \usepackage[parfill]{parskip}
 \usepackage{anysize}
 \marginsize{1in}{1in}{1in}{1in}

\usepackage{graphicx}

\DeclareMathOperator\Id{Id}

\newcommand{\vol}{\mathsf{vol}}

\newcommand{\cL}{\mathcal L}

\newcommand{\cS}{\mathcal S}

\newcommand{\cY}{\mathcal M}
\newcommand{\cX}{\mathcal X}

\newcommand{\cQ}{\mathcal{Q}}
\newcommand{\cP}{\mathcal{P}}

\newcommand{\jota}{\mathrm{j}}

\newcommand{\ssf}{\mathsf{f}}
\newcommand{\sg}{\mathbf{g}}

\renewcommand{\a}{\alpha}
\renewcommand{\b}{\beta}
\renewcommand{\d}{\delta}
\newcommand{\g}{\gamma}
\newcommand{\e}{\varepsilon}
\newcommand{\h}{\theta}
\renewcommand{\l}{\lambda}
\newcommand{\m}{\mu}
\newcommand{\n}{\nu}
\renewcommand{\o}{\omega}
\newcommand{\p}{\phi}

\newcommand{\vp}{\varphi}
\newcommand{\s}{\sigma}
\renewcommand{\t}{\tau}

\renewcommand{\O}{\Omega}

\renewcommand{\L}{\Lambda}
\renewcommand{\P}{\Phi}

\newcommand{\U}[1]{\mathrm{U}(#1)}
\newcommand{\SU}[1]{\mathrm{SU}(#1)}

\newcommand{\SO}[1]{\mathrm{SO}(#1)}

\newcommand{\be}{\bar \varepsilon}
\newcommand{\bz}{\bar 0}

\newcommand{\R}{\mathbb R}
\newcommand{\C}{\mathbb C}
\newcommand{\Q}{\mathbb Q}

\newcommand{\Z}{\mathbb Z}

\newcommand{\cV}{\mathcal{V}}
\newcommand{\rE}{\mathrm{e}}
\newcommand{\rQ}{\mathrm{q}}
                                                                                                                                                                                                                                                                                                                                               \newcommand{\rTh}{\mathrm{Th}}

\newcommand{\Gtwo}{\mathrm{G}_2}

\newcommand{\la}{\langle}
\newcommand{\ra}{\rangle}

\newcommand{\ZZ}{\mathbb{Z}}

\newcommand{\x}{\times}
\newcommand{\UU}{\mathrm{U}}
\newcommand{\Fix}{\operatorname{Fix}}

\newcommand{\CC}{\mathbb{C}}
\newcommand{\CP}{\mathbb{CP}}
\newcommand{\RP}{\mathbb{RP}}

\newcommand{\rId}{\mathrm{Id}}

\renewcommand{\l}{\lambda}

\renewcommand{\b}{\beta}
\renewcommand{\a}{\alpha}

\renewcommand{\P}{\Phi}
\renewcommand{\o}{\omega}

\renewcommand{\L}{\Lambda}
\newcommand{\bx}{\mathbf{x}}
\newcommand{\by}{\mathbf{y}}
\renewcommand{\t}{\tau}

\DeclareMathOperator{\Diff}{Diff}

\theoremstyle{plain}
\newtheorem{proposition}{Proposition}
\newtheorem{theorem}[proposition]{Theorem}
\newtheorem{lemma}[proposition]{Lemma}
\newtheorem{corollary}[proposition]{Corollary}
\theoremstyle{definition}
\newtheorem{definition}[proposition]{Definition}

\theoremstyle{remark}
\newtheorem{remark}{Remark}

\begin{document}

\title{Closed $\Gtwo$ structures on compact manifolds satisfying the known topological obstructions to holonomy $\Gtwo$ metrics}
\author{ Luc\'ia Mart\'in-Merch\'an, \emph{Humboldt-Universität zu Berlin} \\ \tt{lucia.martin.merchan@hu-berlin.de} }

\maketitle

\begin{abstract}
\noindent We construct new compact manifolds endowed with closed $\Gtwo$ structures that satisfy the topological properties found by Joyce and Baraglia for the existence of a torsion-free $\Gtwo$ structure in the same cohomology class. Those manifolds arise as resolutions of orbifolds of the form $M/\Z_2$, where $M$ itself does not admit any torsion-free $\Gtwo$ structure. We develop an equivariant resolution procedure for closed $\Gtwo$ orbifolds with $\Z_2$ isotropy, and analyze some topological properties of the resolved manifolds, including their first Pontryagin class and certain triple Massey products.
\end{abstract}
 
{\small \tableofcontents}

\section{Introduction} \label{sec:introduction}

A $\Gtwo$ structure on a $7$-dimensional manifold $M$ is a reduction of the structure group of its frame bundle from $\mathrm{Gl}(7,\R)$ to the exceptional Lie group $\Gtwo$. Such a 
reduction exists if and only if $M$ is orientable and spin. A $\Gtwo$ structure is
 determined by $\varphi \in \Omega^3(M)$ with pointwise stabilizer equal to $\Gtwo$, and it induces a Riemannian metric $g$ and an orientation. The structure is said to be torsion-free when $\nabla \varphi=0$, or equivalently, when $d\varphi=0$ and $d(\star \varphi)=0$ \cite{FG}. In this case, $g$ has holonomy contained in $\Gtwo$, and it is  Ricci flat. 

Metrics with holonomy $\Gtwo$ were predicted in Berger’s classification \cite{Berger}, but their construction in the compact setting presents analytical difficulties.
The first examples were obtained by Joyce via resolution of flat orbifolds \cites{J1,J2}. Later, Kovalev \cite{Kovalev} introduced the twisted connected sum method, and in joint work with Lee \cite{KovalevLee}, produced more examples. This method was later extended by Corti, Haskins, Nordström, and Pacini \cite{CHNP}, and further refined by Nordström \cite{No}. The most recent method, by Joyce and Karigiannis \cite{JK} consits on resolving torsion-free $\Gtwo$ orbifolds with isolated singularities of $\Z_2$ isotropy.
Additional examples using it appear in \cite{Reid}. 

The known topological constraints imposed by the existence of a holonomy $\Gtwo$ metric 
on a compact manifold seem to be mild. These include two results from Joyce’s foundational work \cite[Proposition 1.1.1, Lemma 1.1.2]{J2}, and a further condition established by Baraglia \cite[Proposition 3.1]{Baraglia}. Let $M$ be a  compact $7$-dimensional manifold that 
admits a holonomy $\Gtwo$ metric whose associated $3$-form lies in the cohomology class $[\phi]\in H^3(M)$ then:
\begin{enumerate}
\item[P1)] $\pi_1(M)$ is finite.
\item[P2)] If the orientation of $M$ is induced by $\phi$ then,
\begin{equation} \label{eqn:T2}
    \int_{M}{p_1(M)\wedge \phi}<0, \quad \mbox{and} \quad   \int_{M}{\alpha \wedge \alpha \wedge \phi}<0 \quad \mbox{ for all } [\alpha]\in H^2(M), \, [\alpha]\neq 0.
\end{equation}
\item[P3)] There is no fibration $\pi \colon M \to B$ onto a $3$-dimensional base that is locally trivial.
\end{enumerate}
For some time, it was conjectured that compact holonomy $\Gtwo$ manifolds; the same was expected for compact manifolds with special holonomy. For simply connected manifolds, formality implies that the rational homotopy type is determined by the rational cohomology algebra; in particular, the rational homotopy groups can be computed from it.
This expectation was motivated by the work of Deligne, Griffiths, Morgan, and Sullivan \cite{DGMS}, who proved that compact Kähler manifolds are formal, and further supported by
the work of Amann and Kapovitch \cite{AK}, which showed formality of compact quaternion Kähler manifolds with positive scalar curvature. In the $\Gtwo$ case, partial results included the formality of compact torsion-free $\Gtwo$ manifolds with $b_1=0$ and $b_2 < 4$ \cite{CN},  with $b_1 > 0$ \cite{FKLS}, and an example constructed by Joyce \cite{AT}. The conjecture for compact holonomy $\Gtwo$ manifolds was finally disproved by the author in \cite{LMM-2}.

Closed $\Gtwo$ structures, i.e., those satisfying $d\varphi=0$, play an important role in constructing compact holonomy $\Gtwo$ manifolds. All known methods, except for that of \cite{JK}, 
begin with a smooth manifold obtained by gluing two open pieces, each equipped with a torsion-free $\Gtwo$ structure and 
 with matching asymptotic behaviour. 
Interpolating these forms yields a closed $\Gtwo$ structure whose torsion $\star d (\star \varphi)$ is small in a precise sense. The construction in \cite{JK} includes a piece that does not come equipped with torsion-free $\Gtwo$ structures, so additional corrections are needed to obtain a closed $\Gtwo$ structure with small torsion.
In any case, Joyce's foundational results \cite{J1} guarantee that such structure can be deformed into a torsion-free one.

Prior to this paper, all known simply connected manifolds with a closed $\Gtwo$ structure were those that admit a holonomy $\Gtwo$ metric. In addition,
it remains an open problem whether there exist simply connected manifolds with closed $\Gtwo$ structures that do not admit torsion-free ones. For compact manifolds with $b_1 >0$,
the existence of a torsion-free $\Gtwo$ structure implies stronger topological properties that are not satisfied by some known manifolds admitting closed $\Gtwo$ structures. A classification of nilmanifolds with this property is provided in \cite{CF}, and further examples with $b_1=1$ were constructed in \cites{FFKM,LMM} by finding involutions on two of such nilmanifolds and resolving the singularities of the resulting quotient. Solvmanifolds admitting closed $\Gtwo$ structures appear in \cites{F-2,VM}.
Another open problem is to identify conditions, beyond small torsion, under which one can deform a closed $\Gtwo$ structure to reduce the torsion and potentially produce a torsion-free one. While the problem currently seems out of reach, geometric flows have been proposed as tools in this direction. In particular, the Laplacian flow introduced by Bryant \cite{Br06}, evolves a closed $\Gtwo$ structure within its cohomology class in the direction of its Hodge Laplacian. Motivated by these questions, we prove:

\begin{theorem} \label{theo:1}
There are compact manifolds $\widetilde{X}$
endowed with a one-parameter family of closed $\Gtwo$ structures, $\widetilde{\varphi}_t,$ such that the properties \emph{(P1)},\emph{(P2)},\emph{(P3)} hold for each pair $(\widetilde{X},\tilde{\varphi}_t)$. These are constructed by resolving the singularities of an orbifold $M/\Z_2$, where $M$ does not carry any torsion-free $\Gtwo$ structure.
\end{theorem}

We do not rule out the existence of holonomy $\Gtwo$ metrics on them. However, $\widetilde{\varphi}_t$ does not have small torsion, as it is constructed from a closed $\Gtwo$ structure on $M$, which itself admits no closed $\Gtwo$ structure satisfying the hypotheses of Joyce's theorem. 
Furthermore, we verify that, with the exception of one example, their topological properties distinguish them from all known compact holonomy $\Gtwo$ manifolds constructed to date.  This suggests that a deeper study of these examples could provide new insights, potentially leading to obstructions to the existence of torsion-free $\Gtwo$ structures or providing a testing ground for existing methods, such as geometric flows. We summarize the studied topological properties in a table:

\begin{table}[h] \label{table:1}
\begin{center}
\begin{tabular}{| r | c | c | c |}
Name & Fundamental group & $(b_2,b_3)$ & Formal \\ \hline
$\widetilde{X}_{1,3}^1$ & $\Z_3$ & $(3,11)$ & Yes \\\hline
$\widetilde{X}_{1,6}^1$ & $\{1\}$ & $(2,10)$ & Yes \\ \hline
$\widetilde{X}_{2,4}^1$ & $\Z_2$ & $(4,17)$ & Yes \\ \hline
$\widetilde{X}_{3,4}^1$ & $\{1\}$ & $(8,29)$ & Yes \\\hline
$\widetilde{Z}_{2,4}$ & $\Z_2$ &   $(4,9)$ & No \\\hline
$\widetilde{Z}_{3,4}$ & $\{1\}$ & $(8,13)$ & No  \\\hline
$\widetilde{X}_{1,3}^2$ & $\{1\}$ & $(2,25)$ & Yes \\\hline
$\widetilde{X}_{1,6}^2$ & $\{1\}$ & $(2,25)$ & Yes\\\hline
$\widetilde{X}_{2,4}^2$ & $\{1\}$ & $(7,46)$ & Yes \\ \hline
\end{tabular}
\label{tab:examples}
\end{center}
\end{table}

In our construction, the initial manifolds are non-formal and have $b_1=1$. Among these, all but three arise as mapping tori of diffeomorphisms of $T^6$ and they are finitely covered by nilmanifolds whose universal cover is the Lie group $G$ with $\mathfrak{g}^*$ spanned by a basis $(e^1, \dots, e^7)$ such that $(de^1,\dots, de^7)=(0,0,0,e^{12}, e^{13},0,0)$. The remaining are mapping tori of diffeomorphisms on the product of a $K3$ surface and $T^2$.  These are finitely covered by the resolution of the quotient of nilmanifolds with universal cover $G$ under a $\Z_2$ action. 
 The closed $\Gtwo$ structures we use come from a left-invariant closed $\Gtwo$ structure on $G$, found in \cite{CF}.  The Laplacian flow for this $\Gtwo$ structure was studied in \cite[Theorem 4.2]{Fernandez-Fino-Manero}. The solution exists for $t\in (-\frac{3}{10},\infty)$, and the associated metrics are nilsolitons. The curvature of such metrics tends to $0$ as $t\to \infty$.

Our examples are constructed using the resolution method developed in \cite{LMM}, which is analogous to the approach in \cite{JK}, and consists on blowing-up the normal bundle of the singular locus. This is a $3$-dimensional submanifold calibrated by $\varphi$. Hence, the exceptional divisor is a $\CP^1$ bundle over this locus. In some cases, we require an equivariant version of this method, which is established in this paper:  
\begin{theorem} \label{theo:2}
Let $(M,\vp,g)$ be a closed $\Gtwo$ structure on a compact manifold. Suppose that $\iota \colon M \to M$ is an involution such that $\iota^*\vp=\vp$ and consider the orbifold $X=M/\iota$.
Let $L=\Fix(\iota)$ be the singular locus of $X$ and suppose that there is a nowhere-vanishing closed $1$-form $\h \in \O^1(L)$.\\
Let $\kappa$ be a finite order diffeomorphism such that $\kappa^*\vp=\vp$, $\kappa \circ \iota = \iota \circ \kappa$ and $\kappa^*\h=\s \h$, where $\s \colon L \to \{\pm 1\}$ is a continuous function.
Then, there is a one-parameter family of closed
$\Gtwo$ resolutions $\rho \colon (\widetilde X, \widetilde \vp_t, \widetilde g) \to (X,\varphi,g)$,
and a finite order diffeomorphism $\tilde{\kappa} \colon \widetilde{X} \to \widetilde{X}$  such that $\tilde{\kappa}\circ  \rho= \rho \circ \kappa$ and
$\tilde{\kappa}^*(\widetilde \vp_t)=\widetilde \vp_t$.
\end{theorem}

The topological study of closed $\Gtwo$ resolutions began in \cite{LMM}, where we computed the cohomology algebra of such spaces. In this paper, we go further by computing the first Pontryagin class and analyzing formality. To describe the results, we let $L=\sqcup_j L_j$ where $L_j$ is connected, and $Q_j=\rho^{-1}(L_j)$. 
For the first Pontryagin class of $\widetilde{X}$, we prove:
 $$
p_1(\widetilde{X})=\rho^*(p_1(M)) +  \sum_j {-3\, \rho^*(\mathrm{Th}[L_j]) + 2\, c_1(\ker(\theta|_{L_j}))\wedge \mathrm{Th}[Q_j]},
$$
where $\mathrm{Th}[Q_j]$ denotes the Thom class of  $Q_j$. The complex structure on $\ker(\theta|_{L_j})$ is determined by $\frac{1}{\|\theta\|}i(\theta^\sharp)\varphi|_{L_j}$, and the
class $c_1(\ker(\theta|_{L_j}))$ is computed for a specific choice of $\theta\in \Omega^1(L)$.
As for formality, we study triple Massey products, which are higher-order cohomological operations whose non-vanishing provides an obstruction to formality. These are partially defined: a triple Massey product $\langle [\alpha], [\beta], [\gamma] \rangle$ is only defined when $[\alpha \wedge \beta]=[\beta \wedge \gamma]=0$. When this condition holds, the result is a cohomology class defined modulo an indeterminacy ideal.  In \cite{LMM-2}, such a non-vanishing triple Massey product was found in a torsion-free $\Gtwo$ resolution, and in this paper, we analyze the mechanism behind this phenomenon, which depends on the configuration of the singular locus. More precisely, we prove:

\begin{proposition}
Let $L_1,L_2,L_3,L_4$ be different connected components of $L$ diffeomorphic to a mapping torus with fiber $T^2$, and such that $L_1\sqcup -L_2$ and  $L_3\sqcup -L_4$ are nullhomologous. If their
linking number is non-zero, 
 then the Triple Massey product
$
\langle\rTh[Q_1]+  \rTh[Q_2], \rTh[Q_3]+ \rTh[Q_4], \rTh[Q_3]- \rTh[Q_4]\rangle $
does not vanish.
\end{proposition}
We also provide a self-contained discussion on intersection theory, which facilitates the computation of linking numbers. Combining these results, we show that the manifolds $\widetilde{Z}_{2,4}$ and $\widetilde{Z}_{3,4}$ are not formal. Furthermore, we show that the remaining examples are formal using the Bianchi-Massey tensor introduced in \cite{CN}. 
 This tensor refines triple Massey products and for $7$-manifolds with $b_1=0$, its vanishing is equivalent to formality. While explicit computations of the Bianchi-Massey tensor in geometric contexts are rare, our work illustrates that it can be effectively used in the setting of $\Gtwo$ resolutions.

 This paper is organized as follows. In section \ref{sec:preliminaries}, we provide the necessary preliminaries on mapping tori, formality of $7$-dimensional manifolds, intersection theory and resolutions of closed $\Gtwo$ structures. Later, in section \ref{sec:topological}, we study triple Massey products on resolutions and prove the formula for the first Pontryagin class. Theorem \ref{theo:2} for equivariant resolutions is proved in section \ref{sec:equiv}.
 Finally, the examples are provided in section  \ref{sec:examples}, where we perform a detailed analysis of their topological properties and prove Theorem \ref{theo:1}.

\textbf{Acknoweledgements}: This project started with a question of Giovanni Bazzoni and Vicente Muñoz. I thank Spiro Karigiannis and Fabian Lehmann for feedback on some sections of this manuscript.

\section{Preliminaries} \label{sec:preliminaries}

The present paper extends the results of \cite{LMM}, making use of several techniques introduced therein. For the relevant background on orbifolds, $\Gtwo$ structures, and rational homotopy theory, 
 we refer the reader to \cite[Section 2]{LMM}.
Here, we focus primarily on results and methods not used in that work.
In section \ref{subsec:map-review}, we review the topological properties of mapping tori. We then introduce the Bianchi-Massey tensor  in section \ref{subsec:formality}, and in section \ref{sec:intersection},  we establish some results in intersection theory relevant to the computation of triple Massey products and the Bianchi-Massey tensor. Finally, in section \ref{subsec:pre-resol}, we fix notation related to $\Gtwo$ structures and review the results from \cite{LMM} used in the remaining sections.

\subsection{Mapping tori} \label{subsec:map-review}

Let $N$ be a smooth manifold, the mapping torus of a diffeomorphism $F \colon N \to N$ is
\begin{equation} \label{eq:MT-def}
N_F= \dfrac{ \R \times N}{\{(t,p) \sim (t+m,F^m(p))\colon m \in \Z\}}=\dfrac{[0,1] \times N}{(0,p) \sim (1,F(p))}.
\end{equation}
We denote $q \colon [0,1]\times N \to N_F$ the quotient projection; we also write $[t,p]= q(t,p)$. If $N$ is oriented
and $f\colon N \to N$ is an orientation preserving diffeomorphism, then $N_F$ is oriented and its orientation is induced by the product orientation on $\R \times N$.
In section \ref{sec:set-up}, we define objects in $N_F$ from those in $N$, more precisely:
\begin{enumerate}
\item Let $\a_t \in \Omega^{k-1}(N)$, $\b_t \in \Omega^k(N)$ a smooth path of forms with $t\in [0,1]$, then $\gamma=dt \wedge \a_t + \b_t$ induces a $k$-form on $N_F$ if
$
F^*\a_1= \a_0$, and $F^*\b_1= \b_0.
$
\item Let $\xi,\eta \colon N  \to N$ be diffeomorphisms, then $\iota[t,p]=[t,\xi(p)]$ is well-defined when
$
\xi \circ  F= F\circ \xi
$. In addition,
$\kappa[t,p]=[1-t,\eta(p)]$ is well-defined when
$
\eta \circ  F= F^{-1}\circ \eta.
$
Their fixed point sets are:
\begin{align}
\label{eq:fix-1}
\Fix(\iota)=&q([0,1]\times \Fix(\xi) ),\\
\label{eq:fix-2}
\Fix(\kappa)=&q(\{0\}\times \Fix(\eta \circ F) \cup \{1/2\}\times \Fix(\eta) ).
\end{align}
\item Assume that $\gamma=dt \wedge \a_t + \b_t\in \Omega^k(N_F)$, and $\iota,\kappa\colon N_F \to N_F$ are well-defined. Then $\iota^*\gamma=\gamma$  when 
$\xi^*\alpha_t=\alpha_{t}$ and $\xi^*\beta_t=\beta_{t}$, and
$\kappa^*\gamma=\gamma$ 
if $\eta^*\alpha_t=-\alpha_{1-t}$ and $\eta^*\beta_t=\beta_{1-t}$.
\end{enumerate}

Note that $N_{F}$ is diffeomorphic to $N_{F^{-1}}$. We now
review some topological properties of mapping tori. 
It is well-known that
\begin{equation}\label{eqn:fund-map-torus}
\pi_1(N_F)\cong \pi_1(N) \rtimes_{F_*} \Z,
\end{equation}
where $F_* \colon \pi_1(N) \to \pi_1(N)$ is the map induced by $F$. Here, we choose the basepoint $[1/2,x_0]$, where
$x_0 \in N$ is a fixed point of $F$.
We obtain a set of generators by identifying $N$ with the level set $\{\frac{1}{2}\} \times N$, and considering the loop 
\begin{equation} \label{eqn:gamma_0}
{\g}_0\colon [0,1] \to N_F, \qquad {\g}_0(s)= 
[{1}/{2} - s, x_0 ],
\end{equation}
whose homotopy class spans the $\Z$ factor in \eqref{eqn:fund-map-torus}.
As usual, we denote $\gamma^{-1}(s)=\gamma(1-s)$. The product relation $ [\rho] \sim [\gamma_0]\cdot F_*[\rho]\cdot  [\gamma_0]^{-1} $ for $[\rho]\in \pi_1(N)$ can be deduced as follows. Denote $\tilde{\rho}(s)=[1/2,\rho(s)]$, and $\widetilde{F \circ \rho}(s)=[-1/2, \rho(s)]=[1/2,(F\circ \rho)(s)]$. A homotopy $H_r$ between $\tilde{\rho}$ and $\g_0 \cdot \widetilde{F \circ \rho} \cdot \g_0^{-1}$  is constructed as a sequence of three paths. First travel in the $t$-direction from $[1,2,[x_0]]$ to $[1/2-r,[x_0]]$. Then move horizontally along $s \mapsto [1/2-r,\rho(s)]$ and finally reverse the first path to return. 

Some cohomological properties of mapping tori are contained in \cite{BFM}. The Mayer--Vietoris sequence, applied choosing the open cover $U_1=q((\frac{1}{5}, \frac{4}{5})\times N)$, $U_2=q([0,\frac{1}{3}) \times N \cup (\frac{2}{3},1] \times N)$, yields:

\begin{lemma} \cite[Lemma 12]{BFM}
Let $N$ be a smooth manifold and let $F\colon N \to N$ be a diffeomorphism. Then, there is a split short exact sequence
\begin{equation}\label{eqn:short-mp}
0 \to C^{m-1} \to H^m(N_F) \to K^m \to 0,
\end{equation}
where $K^m$ and $C^m$ are the kernel and cokernel of the map $(F^*- \rId) \colon H^m(N) \to H^m(N)$.
\end{lemma}

The inclusion $\delta^* \colon C^{m-1} \to H^m(N_F)$ is determined by the connecting isomorphism of the sequence. Given an increasing bump function
$\rho\colon [0,1]\to [0,1]$ which equals $0$ on $t \leq \frac{1}{3}$ and $1$ on $t \geq \frac{2}{3}$, we can write 
$\delta^*[\beta]= [d\rho \wedge \beta]$. In particular, $\delta^*(1)=[dt]$. The surjection $H^m(N_F) \to K^m$ is
the restriction to $q(\{\frac{1}{4}\} \times N)$. To find a section, we pick $[\alpha] \in K^m$, and $\beta \in \Omega^{m-1}(N)$ with $d\beta= \alpha-F^*\alpha$, the form  
\begin{equation} \label{eqn:tilde-alpha}
\widetilde{\alpha} = F^*\alpha + d(\rho \beta),
\end{equation}
is a well-defined closed form on $N_F$ with $[\widetilde{\alpha}]|_{t=1/4}=F^*[\alpha]=[\alpha]$.

In section \ref{sec:top} we need to find the cohomology classes that remain invariant under the action of $\iota,\kappa \colon N_F \to N_F$. From  the discussion above, we observe that the action preserves the subspaces $C^{m-1}$ and $K^{m}$. For $[\beta] \in C^{m-1}$, we have $\iota^*\delta^* [\beta]=\delta^*[\beta]$  when $\xi^*[\beta] \equiv [\beta] \mod \mathrm{Im}(F^* - \rId)$ and $\kappa^* \delta^* [\beta]=\delta^*[\beta]$ if $\eta^*[\beta] \equiv -[\beta]\mod \mathrm{Im}(F^* - \rId)$. Moreover, for forms in $K^m$,
$\iota$-invariance (respectively, $\kappa$-invariance) is equivalent to $\xi$-invariance (respectively, $\eta$-invariant). 

Finally, the formality of the mapping torus $N_F$ depends on the multiplicity of $1$ as an eigenvalue of $F^*$. In \cite{BFM},
the multiplicity of an eigenvalue $\lambda$ of an endomorphism $A \colon V \to V$ is defined by its multiplicity as a root of the minimal polynomial of $A$. 

\begin{theorem} \cite[Corollary 16]{BFM} \label{theo:BFM-non-formality}
Assume that there is some $m\geq 2$ such that $K^n=\{0\}$ for $n\leq m-1$ and that $F^* \colon H^m(N) \to H^m(N)$ has the eigenvalue $\lambda=1$ with multiplicity $r\geq 2$. Then, $N_F$ is not formal. 
\end{theorem}

\subsection{Formality of 7-dimensional manifolds with $b_1=0$}\label{subsec:formality}

The Bianchi-Massey tensor, introduced in \cite{CN},  determines whether a $7$-dimensional manifold with $b_1=0$ is formal. In addition, for those that are simply connected, the Bianchi-Massey tensor and the cohomology algebra encode the rational homotopy type. In section \ref{sec:examples-formality} we use this tensor to  show that some of the manifolds constructed in section \ref{subsec:ex} are formal. To establish the non-formality of the others, we find a non-vanishing triple Massey product on them. Below, we briefly recall the notions and results needed. We remark that, 
while \cite{CN} works over $\Q$, we work over $\R$ because our focus is on formality, which descends from $\R$ to $\Q$ (see \cite[Theorem 12.1]{Sullivan}).

Triple Massey products encode linking phenomena between submanifolds and provide obstructions to formality. Given cohomology classes $\xi_1,\xi_2,\xi_3\in H^*(M)$, their triple Massey product $\langle \xi_1,\xi_2,\xi_3 \rangle$ is defined when
$\xi_1\wedge \xi_2=0$ and $\xi_2\wedge \xi_3=0$. In this case,
we choose
representatives $\alpha_i\in \Omega^*(M)$ of $\xi_i$, and primitives $\b_{12}$ and $\b_{23}$ of  $\alpha_1\wedge\alpha_2$ and $\alpha_2\wedge \alpha_3$. Define the cohomology class
$$
y=  [\beta_{12} \wedge \alpha_3 - (-1)^{\deg(\a_1)}\alpha_1  \wedge \beta_{23}],
$$
and  let $\pi \colon H^*(M) \to H^*(M)/\langle \xi_1,\xi_3 \rangle H^*(M)$ be the quotient projection. Then, one can check that $\pi(y)$ does not depend on the chosen representatives. We define $\langle \xi_1,\xi_2,\xi_3 \rangle= \pi(y) $, and we say it vanishes if
$y\in \langle \xi_1,\xi_3\rangle H^*(M)$. 

\begin{lemma} \label{lem:tmp-vanish} \cite[Proposition 2.90]{FOT} Triple Massey products vanish on formal manifolds.
\end{lemma} 

 We now adapt the definition provided in \cite{CN} of the Bianchi-Masey tensor to the particular set-up of a $7$-dimensional manifold $M$ with $b_1=0$. Let $\mathcal{Z}^*(M)$ be the subspace of closed forms, and pick a linear map $\alpha \colon H^2(M) \to \mathcal{Z}^2(M)$ that chooses a representative of each cohomology class.
Consider the wedge product map
$$
\wedge\colon \mathrm{Sym}^2(H^2(M)) \to H^{4}(M),
$$
and denote $E^4(M)=\ker (\wedge)$.
Choose a homomorphism $\gamma \colon E^4(M) \to \Omega^{3}(M)$ such that $d\gamma(e)=\alpha^2 (e)$ for every $e\in E^4(M)$. Of course, we understand $\a^2(\xi_1 \odot \xi_2)= \a(\xi_1)\wedge \a(\xi_2)$.  The graded symmetrization of 
 $$
E^4(M) \otimes E^4(M) \to H^7(M), \qquad e\otimes e \mapsto [\gamma(e)\wedge \alpha^2(e')],
 $$
 is the extension of $e\odot e'\mapsto [\gamma(e)\wedge \alpha^2(e')]$ because 
 $d(\gamma(e)\wedge \gamma(e'))= \alpha^2(e)\wedge \gamma(e') - \gamma(e)\wedge \alpha^2(e')$.  Consider the full symmetrization tensor
$$
\mathfrak{S} \colon \mathrm{Sym}^2(\mathrm{Sym}^2(H^2(M)) \to \mathrm{Sym}^4(H^2(M)), \qquad (\xi_1 \odot \xi_2)\odot (\xi_3 \odot \xi_4) \longmapsto \xi_1 \cdot \xi_2 \cdot \xi_3 \cdot \xi_4,
$$
and define $\mathcal{B}^8(M)= \ker (\mathfrak{S} |_{\mathrm{Sym}^2(E^4(M))})$. The Bianchi Massey tensor of $M$ is
$$
\mathcal{F}_M \colon \mathcal{B}^8(M)\to H^7(M), \quad \mathcal{F}_M(e\odot e')=[\gamma(e)\wedge (\alpha^2)(e')],
$$
which turns out to be independent of the choices for $\alpha$ and $\gamma$. The following results are relevant for us:

\begin{proposition}\cite[Corollary 3.10, Theorem 1.14]{CN} \label{prop:CN-formality}
Let $M$ be an oriented $7$-dimensional manifold with $b_1=0$, then
\begin{enumerate}
\item $M$ is formal if and only if $\mathcal{F}_M$ vanishes.
\item If $b_2\leq 3$ and there is  $\xi \in H^3(M)$ such that $(\cdot)\wedge \xi \colon H^2(M) \to H^5(M)$ is an isomorphism, then $M$ is formal.
\end{enumerate}
\end{proposition}

\begin{remark} 
\cite[Theorem 1.14]{CN} is stated for simply-connected $7$-manifolds but it just needs $b_1=0$. This result relies on \cite[Corollary 3.10]{CN} together with
 \cite[Proposition 5.1]{CN}, which shows $\mathcal{B}^8(M)=0$ under these assumptions. 
 However, the framework for rational homotopy theory via minimal models only applies to nilpotent spaces, namely spaces $M$ for which $\pi_1(M)$ is a nilpotent group, and the action of $\pi_1(M)$ on $\pi_n(M)$ is nilpotent. 
\end{remark}

\subsection{Intersection theory} \label{sec:intersection} 

\textbf{The Thom class of a compact oriented submanifold:}
Let $L$ be a compact manifold (maybe with boundary) and $E \to L$ an oriented oriented vector bundle with rank $r$. A compactly supported closed form $\upsilon \in \Omega_c^{r}(E)$ is a Thom form of $E$ if
it integrates to $1$ on each fiber.
The Thom isomorphism theorem states \cite[Theorem 6.17]{Bott-Tu} that integration along the fibers yield an isomorphism $H_c^{i+r}(E) \to H^i(L)$. Hence, all Thom forms determine the same cohomology class, called the Thom class, which generates $H_c^r(E)$. In addition, we can assume that the support of a Thom form is arbitrarily close to the zero section (see \cite[Remark  12.4.1]{Bott-Tu}). A property that we will use later is that the Thom class of the direct sum of two vector bundles is the wedge product of their Thom classes (\cite[Proposition 6.19]{Bott-Tu}).

Let $M$ be a closed oriented manifold and let $L$ be a  compact submanifold of codimension $r$ that may have boundary. Fixed some metric on $M$, we consider the normal bundle $\nu(L)$ of $L$, oriented so that the splitting $TM|_{L}=TL\oplus \nu(L)$ preserves orientations. Let $U$ be a tubular neighborhood of $L$, and $\upsilon$ a Thom form of $\nu(L)$ supported on $\exp^{-1}(U)$. We say that $\tau=(\exp_L)_*(\upsilon)\in \Omega^r_c(U)$ is a Thom form for $L$. Of course, $\tau$ yields to a generator of $H^r_c(U)$, characterized by the property that it integrates to 
$1$ over the fibers of the nearest point projection $U\to L$.
If $\partial L=\emptyset$, then there is a natural inclusion $(\Omega_c^*(U),d) \hookrightarrow (\Omega^*(M),d)$, and we denote by $\mathrm{Th}[L]\in H^r(M)$ the cohomology class determined by $\tau$.  It turns out that $\mathrm{Th}[L]$ is the Poincaré dual of $L$ (see \cite[Proposition 6.24]{Bott-Tu}).
Otherwise, $\tau$ does not extend to $M$, as it is non-zero on a disk bundle over $\partial L$. 

\begin{remark} \label{rem:ThomOrbibundle} 
Let $\iota\colon M \to M$ be an orientation-preserving involution with fixed points, and let $g$ be a $\iota$-invariant metric. Suppose $L\subset \Fix(\iota)$ is an oriented connected component with  (even) codimension $r$; then $\iota^*$ acts on $\nu(L)$ as $-\rId$.  
The normal bundle of $L$ on the orbifold $X=M/\iota$ is $\nu(L)/\iota_*$. By the tubular neighborhood theorem for orbifolds, the exponential map yields a diffeomorphism from a neighborhood of 
$L$ on $\nu(L)/\iota_*$ and  $U=V/\iota$, for some $\iota$-invariant tubular neighborhood $V\subset M$ of $L$. Similarly to \cite[Section 2]{LMM}, we have $(\Omega_c(U),d)=(\Omega_c(V)^{\iota_*},d)$, and an average argument yields $H^*_c(U)\cong H^*_c(V)^\iota$. Since $\iota_*$ is an orientation-preserving isomorphism of $\nu(L)$, the pullback $\iota^*\upsilon$ of a Thom form $\upsilon$ of $L$ on $M$ is again a Thom form. Thus $\frac{1}{2}(\upsilon + \iota^*\upsilon)$ is a $\iota$-invariant Thom form, and $H^r_c(U) \cong \R$ with the isomorphism given by integration along the fibers. 

We say that $\tau \in \Omega_c(U)$ is a Thom form of $L$ if it integrates to $1$ on the fibers of the normal bundle $\nu(L)/\iota_*$. Note that, $\tau\in \Omega_c^r(U)$ is a Thom form of $L$ on $X$ if and only if $\frac{1}{2}\tau \in \Omega_c^r(V)$ is a $\iota$-invariant Thom form of $L$ on $M$. This follows from 
$$
\int_{\nu_x(L)/\iota_*}{\exp^*_L(\tau)}=\frac{1}{2}\int_{\nu_x(L)}\exp^*_L(\tau).
$$
We also denote by $\mathrm{Th}[L]\in H^r(X)$ the class of a Thom form $\tau$ of $L$.
This coincides with the Poincar\'e Dual of $L$ in $X$ (see \cite{Satake} for a proof of Poincar\'e Duality for orbifolds) because for any $\alpha\in \Omega^{\dim(M)-r}(X)$ we have
$$
\int_X \alpha \wedge \tau= \frac{1}{2}\int_M \alpha \wedge \tau = \int_L \alpha,
$$
where we used $\frac{1}{2}[\tau]=PD[L]$ on $M$.
\end{remark}

A direct consequence of \cite[Proposition 6.19]{Bott-Tu} is that, if $ \tau_j \in \Omega(M_j) $, is a Thom form of a closed oriented submanifold $ L_j \subset M_j$ (for $j=1,2$), then 
 a Thom form of $ L_1 \times L_2 $, is given by $ (-1)^{\mathrm{codim}(L_1)\dim(L_2)} \tau_1 \wedge \tau_2 $.

\textbf{Intersections:} Consider two compact oriented submanifolds $L_1$ and $L_2$, with $\partial L_2 =\emptyset$. If they intersect transversally outside  $\partial L_1$, then $L_1\cap L_2$ is a closed submanifold (see \cite{Guillemin-Pollack}, p. 60). Transversality implies $\nu(L_1\cap L_2)\cong \nu(L_1)|_{L_1 \cap L_2}\oplus \nu(L_2)|_{L_1\cap L_2}$. Hence, we provide $\nu(L_1\cap L_2)$ with the direct sum orientation and then orient $T(L_1\cap L_2)$ so that $TM|_{L_1\cap L_2}=T(L_1\cap L_2)\oplus \nu(L_1\cap L_2)$ as oriented vector bundles.  If $\tau_{1}$ is a Thom form for $L_1$ on a tubular neighborhood of $L_1$ (which might not extend to $M$) and $\tau_{2}$ is a Thom form of $L_2$, then $\tau_{1}\wedge \tau_{2}$ is well-defined on $M$ and \cite[Proposition 6.19]{Bott-Tu} implies that $[\tau_{1}\wedge \tau_{2}]=\mathrm{Th}[L_1\cap L_2]$.

\textbf{Nullcobordant submanifolds:} If a closed oriented submanifold $L$ is nullhomologous, then $\mathrm{Th}[L]=PD[L]=0$. In addition, if $L=\partial C$, Lemma \ref{lem:intersections} shows that we can use the Thom form of $C$ to construct a primitive of a Thom form of $L$. Recall that 
given a $(k+1)$-dimensional compact oriented manifold  with boundary $C$, we endow $\partial C$ with the orientation determined by the normal vector field pointing inwards $V_i$. That is, $(v_1,\dots, v_{k})$ is a positive basis of $T\partial C$ if and only if $(v_1,\dots,v_{k},V_i)$ is a positive basis of $TC$. If $k=3$, this is the same as saying that $(v_1,v_2,v_3)$ is positive if $(V_o,v_1,v_2,v_3)$ is positive; where $V_o=-V_i$ is the normal vector pointing outwards.

\begin{lemma} \label{lem:intersections}
Let $M$ be a smooth oriented manifold. Let $C$ be
a compact oriented submanifold with boundary, and let $\tau_{\partial C}$ be a Thom form of $\partial C$ supported on a tubular neighborhood $V$ of $\partial C$. Given a tubular neighborhood $U$ of $C$, there is a form $\beta$ supported on $U\cup V$ such that  $\beta|_{U-V}$ is a Thom form of $C- V$ and $d\beta=\tau_{\partial C}$.
 
In addition, if $\tau'$ is a Thom form of a submanifold $N$ intersecting $C$ transversally outside of $V$, then $\beta \wedge \tau'$ is a Thom form of $C\cap N$.
\end{lemma}
\begin{proof}
 We  define the subsets $W_s^t$ by:
$$ 
D_s=\exp_{\partial C}(\{ \lambda V_i, \, \lambda \in [0,s]\}), \qquad  W_s^t=\exp_C( \{ v_p \in \nu(C), \, p\in D_{s}, \|v_p\|<t\}).
$$
Shrinking $U$, and choosing $\e,\delta>0$ small enough we assume that  $U=\exp_C(\{ v_p \in \nu(C), \, \|v_p\|<\delta\})$ and $W_\e^\delta \subset V$ . We now claim that a Thom form of the oriented vector bundle $\la V_i \ra|_{\partial C} \to \partial C$ is $\exp_{\partial C}^*(db)$, where $b$ is a bump function defined in $U$, that satisfies $b=0$ on $W_{\e/4}^{\delta}$, and $b=1$ on $U-W_{\e/2}^{\delta}$. To verify this claim we observe that if $p\in \partial C$, then
$$
\int_{\R V_i(p)}{\exp_{\partial C}^* d{b}}
={b}(\exp_p((\e/2) V_i(p))- {b}(\exp_p(({\e}/{4}) V_i(p))=1.
$$
Given $\tau_C \in \Omega_c^{n-k}(U)$ a Thom form of $C$, then  $db \wedge \tau_C$ extends to $M$ as a Thom form of $\partial C$ supported on $V$. The last claim follows from \cite[Proposition 6.19]{Bott-Tu} and the splitting $TC|_{\partial C}=T\partial C \oplus \langle V_i \rangle$ as oriented vector bundles, which implies $\nu(\partial C)= \la V_i \ra \oplus \nu(C)|_{\partial C}$.
Indeed, $\beta'={b}\tau_C$ extends to a well-defined form on $M$
because $b$ vanishes on $W_{\e/4}^\delta$. In addition, $d\beta'= db \wedge \tau_C$.

Since $[\tau_{\partial C}]=[d\b']\in H_c^{n-k}(V)$ are representatives of the Thom class with support in $V$, there is $\beta''\in \Omega_c^{n-k-1}(V)$ such that $d\beta''= \tau_{\partial C} - d\b'$.
The form  $\beta= \beta' + \beta''$ satisfies $d\beta= \tau_{\partial C}$ and 
 $\beta|_{U-V}=b\tau_C|_{U-V}=\tau_C$ because on $b=1$ on $U-W_{\e/2}^\delta \supset U-V$.  This shows that $\beta|_{U-V}$ is a Thom form of $C-V$.  If $N$ intersects $C$ transversally on $C-V$ then $\beta \wedge \tau_N$ is a Thom form for $C\cap N$ as discussed before.
\end{proof}

We now use \cite[Proposition 6.19]{Bott-Tu} to find primitives of Thom forms of certain  nullhomologous submanifolds on the fibers of a mapping torus. This will be useful later.

\begin{lemma} \label{lem:PD-and-MTorus}
Let $N$ be a closed oriented manifold and
$F\colon N \to N$ be an orientation preserving diffeomorphism.
Let $L_1,\dots,L_k \subset N$ be closed oriented submanifolds of dimension $s$, let $t_0 \in \R$ and define $L_{j}^0=q(\{t_0\}\times L_j) \subset N_F$.

Assume that
$\sum_j{\lambda_j [L_j]}=0$ on $H^r(N,\R)$. Consider Thom forms $\tau_{j}$ of $L_j$ on $N$, and  $\beta \in \Omega^*(N)$ such that $d\beta=\sum_{j}\lambda_j \tau_{j}$. There is a function $b\colon [t_0-1/2,t_0+1/2]\to \R$, which is constant outside a neighborhood of $t_0$ and such that
\begin{enumerate}
    \item $\tau_j^0= db\wedge \tau_{L_j}$ is a Thom form of $L_{j}^0$.
    \item $d(-db\wedge \beta)=\sum_{j}{\lambda_j} \tau_j^0$.
\end{enumerate}
\end{lemma}
\begin{proof}
Observe that $\nu(L_j^0)\cong \langle (-1)^s \partial_t \rangle \oplus \nu_N(L_j)$ as oriented vector bundles. We claim that a Thom form for $\langle (-1)^s \partial_t \rangle \to   q(\{t_0\} \x N)$, is $\exp_N^*(db)$ where
 $b$ is a monotonic function
 that equals $0$ on $[t_0-1/2, t_0-\e)$ and $(-1)^s$ if $t\in (t_0+\e,t_0+1/2]$. The claim holds because for any $p\in N$ we have
$$
\int_{\langle \R(-1)^s\partial_t|_{p} \rangle}\exp^*_N(db)=(-1)^s \int_{[t_0-\e,t_0+\e]}{db}=(-1)^s (b(t_0+\e)- b(t_0-\e))=1.
$$
In particular, $\exp^*_N(db)$ restricts to a Thom form of $\langle (-1)^s \partial_t \rangle \to L_j^0$. Since $db$ extends to $N_F$, so do the forms
$db\wedge \beta$ and $db\wedge \tau_{Lj}$. In addition, 
 \cite[Proposition 6.19]{Bott-Tu} ensures that a Thom form of $L_j^0$ is $\tau_j^0=db\wedge \tau_{j}$. Finally, $d(-db\wedge \beta)= \sum_j \lambda_j db\wedge \tau_{j}= \sum_j \lambda_j \tau_j^0$. 
\end{proof}

\textbf{Linking number}: The following discussion is contained in \cite[p. 230-- 234]{Bott-Tu}.

\begin{definition}\label{def:lk}
Let $M$ be a compact oriented manifold without boundary and let $L_1$ and $L_2$ be two disjoint compact oriented submanifolds without boundary that are nullhomologous (on $H_*(M,\R)$) and  such that $\mathrm{dim}(L_1) + \mathrm{dim}(L_2)=\mathrm{dim}(M)-1$. The linking number of $L_1$ and $L_2$ is defined as
$$
\mathrm{lk}_M(L_1,L_2)=\int_{M}{\beta_{1} \wedge \tau_{2}},
$$
where $\tau_{1}$ and $\tau_{2}$ are Thom forms of $L_j$ and $d\beta_1=\tau_{1}$.
\end{definition}
The linking number does not depend on the choices of $\tau_{j}$ and $\beta_{1}$.  Let $\tau_{1}'$ and $\tau_{2}'$ be Thom forms of $L_1$ and $L_2$ and let $\beta_{1}'$  with $d\beta_{1}'=\tau_{1}'$. There are forms $\alpha_1$, $\alpha_2$ supported on tubular neighborhoods of $L_1$, $L_2$ respectively with $d\alpha_j=\tau_{j}'- \tau_{j}$. 
Hence, there is a closed form $c$ on $M$ such that $\beta_{1}'-\beta_{1}= c+ \alpha_1$.
Using that $\alpha_1\wedge \tau_{2}'=0$ we obtain:
\begin{align*}
\beta_{1}' \wedge \tau_{2}' =& (\beta_{1}'-\beta_{1})\wedge \tau_{2}' + \beta_{1}\wedge \tau_{2}' \\
=&(c + \alpha_1)\wedge \tau_{2}' + \beta_{1}\wedge (\tau_{2}'- \tau_{2}) + \beta_{1}\wedge \tau_{2}\\
 =&
c \wedge \tau_{2}' + \beta_{1}\wedge d\alpha_2 + \beta_{1}\wedge \tau_{2}
.
\end{align*}
Since $[\tau_{2}]=0$ and $c$ is closed, $[c\wedge \tau_2 ]=0$. In addition, $\int_M{ \beta_{1}\wedge d\alpha_2} = \pm \int_M{ \tau_{1} \wedge \alpha_2} =0$, as $\tau_1 \wedge \alpha_2$ vanishes pointwise.
Hence, $\int_{M} {\beta_{1}' \wedge \tau_{2}' } = \int_{M}{\beta_{1} \wedge \tau_{2}}$.

In the set-up of Definition \ref{def:lk}, if we also suppose that $L_1=\partial C$, and $C$ intersects $L_2$ transversally. Then $C\cap L_1$ is a finite collection of points $\{p_j\}$ with orientations $\e_j \in \{\pm 1\}$. Lemma \ref{lem:intersections} ensures that $\mathrm{lk}(L_1,L_2)=\sum_{j} \e_j$.

\subsection{Resolution of closed $\Gtwo$ orbifolds}\label{subsec:pre-resol}

\textbf{$\Gtwo$ structures:} The standard cross product $\x\colon \Lambda^2 \R^7 \to \R^7 $, and the scalar product $\la \cdot, \cdot \ra$ determine the $3$-form $\vp_0(u,v,w)= \la u\x v, w \ra $. In terms of the standard basis $(e^1,\dots, e^7)$ of $(\R^7)^*$, is expressed as follows:
\begin{equation}\label{eqn:std}
\vp_0= e^{123}+ e^{145}+ e^{167}+ e^{246}- e^{25
}-e^{347} - e^{357}.
\end{equation}
The group $\Gtwo$ is the stabilizer of $\vp_0$ under the action of $\mathrm{Gl}(7,\R)$ on $\L^3(\R^7)^*$. It is a simply connected, $14$-dimensional subgroup of $\SO{7}$. 
There is an inclusion $\SU{2} \subset \Gtwo$ associated with the splitting $\R^7=\R^3 \oplus \R^4$ for which $\varphi_0= e^{123}+ e^1\wedge \omega_1^0 + e^2\wedge \omega_2^0 -e^3 \wedge \omega_3^0$, where
$(\o_1^0,\o_2^0,\o_3^0)$ is the standard $\SU{2}$ structure on $\R^4$:
\begin{equation}\label{eqn:su2}
\begin{aligned}
\o_1^0 = e^{45} + e^{67}, \qquad
\o_2^0= e^{46} - e^{57}, \qquad
\o_3^0 =e^{47} + e^{56}.
\end{aligned}
\end{equation}

\begin{definition}
A $\Gtwo$ form on a $7$-dimensional manifold $M$
is $\vp \in \O^3(M)$ such that for every $p \in M$ there is an isomorphism $u_p \colon T_pM \to \R^7$ that satisfies $u_p^*(\vp_0)=\vp_p$. If $\Gamma$ is a finite subgroup of diffeomorphisms acting non-freely, a $\Gtwo$ form on the orbifold $M/\Gamma$ is a $\Gamma$-invariant $\Gtwo$ form on $M$.
\end{definition}

A $\Gtwo$ form $\varphi$  determines an orientation and a metric $g$ by the identity,
\begin{equation}\label{eqn:metric-formula}
i(X)\vp \wedge i(Y) \vp \wedge \vp = 6 g(X,Y)\mathrm{vol}_g,
\end{equation}
where $\mathrm{vol}_g$ denotes the unit-length oriented volume form
associated to $g$. A vector cross product is then determined  by the equation $\varphi(X,Y,Z)=g(X \times Y,Z)$, $X,Y,Z\in \mathfrak{X}(M)$. A $\Gtwo$ structure is a triple $(M,\varphi,g)$. According to the classification by Fernandez and Gray, \cite{FG} we say it is closed if $d\varphi =0$, and torsion-free when $\nabla \varphi=0$.  

To construct a closed $\Gtwo$-structure on the resolution of an orbifold, we need to interpolate between closed $\Gtwo$ forms that are sufficiently close. To ensure that the result is a $\Gtwo$ form, we use the following Lemma, which follows from the fact that the orbit of $\varphi_0$ under the action of $\mathrm{GL}(7,\R)$ is open in $\Lambda^3(\R^7)^*$.
\begin{lemma} \label{lem:universal-m} \cite[Chapter 10,
Section 3]{Joyce2} \cite[Lemma 2.4]{LMM}
There exists a universal constant $m$ such that if $(M,\varphi,g)$ is a $\Gtwo$ structure and $\| \phi - \varphi \|_{C^0,g}< m$ then $\phi$ is a $\Gtwo$ form.
\end{lemma}

\textbf{Associative submanifolds:}
Let $(M,\varphi,g)$ be a closed $\Gtwo$ structure,
a $3$-dimensional submanifold $L$ is associative if 
$$
\varphi(v_1,v_2,v_3)=\pm 1, \text{ for all } p\in L  \text{ and for all } (v_1,v_2,v_3) \text{ orthonormal basis of } T_pL. 
$$
Of course, $L$ is oriented by $\varphi|_L$, which is the unit-length volume form induced by the orientation and $g|_{L}$. The group $\Gtwo$ acts transitively on the Grassmanian of associative $3$-planes in $\R^7$ (see \cite[Corollary 8.3]{SW}). The local model of an associative submanifold is that of $\R^3 \subset \R^7=\R^3\oplus \R^4$ described above. This implies, in particular, that
 if $L$ is associative then 
\begin{equation}\label{eqn:associative-local}
\varphi|_{TM|_L} \in \Lambda^3(T^*L)\oplus \Lambda^1T^*L \otimes \Lambda^2\nu(L)^*,
\end{equation}
where $\nu(L)$ denotes the normal bundle of $L$.
A unit-length vector field $V \in \mathfrak{X}(L)$ determines a complex structure on $\nu(L)$ by:
\begin{equation} \label{eqn:acs-norml}
X \longmapsto J_{\nu(L)}(X)=V \times X,
\end{equation}
where $\times$ is the cross product determined by $\varphi$.
Observe that the associated $2$-form is $\omega(X,Y)=g(J_{\nu(L)}(X),Y)=i(V)\varphi|_{\nu(L)}$. 
We now show that the orientation induced by $J_{\nu(L)}$ on $\nu(L)$ coincides with the one determined by the orientation of $L$ and $M$. We complete $v_1=V(p)$ to an orthonormal oriented basis $(v_1,v_2,v_3)$ of $T_pL$, and choose an orthonormal oriented basis $(\xi_1,\xi_2,\xi_3,\xi_4) \in \nu_p(L)$. Then,
using equations \eqref{eqn:metric-formula} and \eqref{eqn:associative-local} we obtain:
\begin{align*}
6&=6g(v_1,v_1)\mathrm{vol}_{g}(v_1,v_2,v_3,\xi_1,\xi_2,\xi_3,\xi_4)=  (i(v_1)\varphi)^2\wedge \varphi \, (v_1,v_2,v_3,\xi_1,\xi_2,\xi_3,\xi_4)\\
&= (i(v_1)\varphi)^3 (v_2,v_3,\xi_1,\xi_2,\xi_3,\xi_4)
= 3 \varphi(v_1,v_2,v_3) (i(v_1)\varphi)^2(\xi_1,\xi_2,\xi_3,\xi_4)\\
&=3 \omega^2(\xi_1,\xi_2,\xi_3,\xi_4).
\end{align*}
Hence, $\frac{1}{2}\omega^2$ is the oriented unit-length volume form of $\nu(L)$.

\subsubsection*{Resolution of compact closed $\Gtwo$ orbifolds  $M/\Z_2$}

Let $(M,\varphi,g)$ be a closed $\Gtwo$ structure, and let $\iota \colon M \to M$ be an involution with fixed points such that $\iota^*(\varphi)=\varphi$. We now discuss whether we can find a closed $\Gtwo$ resolution  $\rho \colon (\widetilde M, \widetilde \vp, \widetilde g) \to (X=M/\iota,\varphi,g)$. By this, we mean a smooth
 manifold endowed with a closed $\Gtwo$ structure $(\widetilde M, \widetilde \vp, \widetilde g)$ and a map $\rho \colon \widetilde M \to X$ such that: 
\begin{enumerate}
\item The map $\rho  \colon \widetilde M- \rho^{-1}(L) \to X-L$ is a diffeomorphism.
\item There exists a small neighbourhood $U$ of $L$ such that $\rho^*(\vp)=\widetilde \vp$ on $\widetilde M- \rho^{-1}(U)$. 
\end{enumerate}

To begin, the next result states that the fixed locus of $\iota$ is an associative submanifold:

\begin{lemma}\cite[Lemma 2.8]{LMM} \label{lem:3-dim}
Let $(M,\varphi,g)$ be a $\Gtwo$ structure and assume that $\iota \colon M \to M$ satisfies $\iota^*\varphi=\varphi$. Then $\iota$ is either fixed point-free or the submanifold $L=\Fix(\iota)$ is an associative $3$-dimensional submanifold.
\end{lemma}

\emph {From now on, we assume that $L=\Fix(\iota)\neq \emptyset$, and that $L$ is oriented by $\vp|_L$. }

The resolution of $M/\iota$ occurs on the normal bundle of $L$, where $\iota_*$ acts as $-\rId$. The choice of a unit-length vector field determines the almost complex structure $J_{\nu(L)}$ on $\nu(L)$ described by equation \eqref{eqn:acs-norml}, which is invariant under conjugation by $\iota_*$. This allows us to identify the fibers of 
 $\nu(L)/\iota_*$ with $\CC^2/\ZZ_2$. 
 The resolution of $\CC^2/\ZZ_2$ is $\chi_0\colon \widetilde{\CC^2}/\Z_2 \to \C^2/\Z_2$, where 
 $
\widetilde \CC^2= \{ (z_1,z_2,\ell)\in \CC^2 \x \CP^1 \mbox{ s.t. } (z_1,z_2)\in \ell\},
$
and the $\Z_2$ action is $(z_1,z_2,\ell) \longmapsto (-z_1,-z_2,\ell)$. This space,  known as the Eguchi-Hanson space, admits a family of metrics with holonomy $\SU{2}$. Given 
 $a>0$, one such metric is determined by the standard complex structure and the $2$-form:
$$
\widehat \o_1^a= -\frac{1}{4} dId \left( (r_0^4 + a^2)^{1/2}- a\log((r_0^4 + a^2)^{1/2}+a) + 2a\log(r_0) \right),
$$
where $r_0$ denotes the radial function.
These share the complex volume form  
$\chi_0^*(\o_2^0+i\o_3^0)$. 

 The fiberwise blow-up $P \to \nu(L)/\iota_*$ is a smooth bundle over $L$, with fiber $\widetilde{\CC^2}/\Z_2$. More precisely, let $P_{\U{2}}(\nu(L))$ be the principal $\U{2}$ bundle associated to $(\nu(L),J_{\nu(L)})$, we can describe $P$ as follows:
 \begin{equation} \label{eqn:P-definition}
 P= P_{\U{2}}(\nu(L))\times_{\U{2}} \widetilde{\CC^2}/\Z_2=\{ ([v_x]_{\ZZ_2},\ell_x) \colon \, (v_x,\ell_x) \in \nu_x(L)\times \mathbb{P}(\nu_x(L)), \,  v_x \in \ell_x  \},
 \end{equation}
 where $[v_x]_{\Z_2}$ denotes the equivalence class of $v_x$ under the action of $\iota_*|_{\nu(L)}= -\mathrm{Id}$.
 The exceptional divisor is $Q=P_{\U{2}}(\nu(L))\times_{\U{2}} \CP^1=\mathbb{P}(\nu(L))$.
The manifold $P$ carries an orientation such that the projection $P-Q\to \nu(L)/\Z_2-L$ is orientation-preserving. In addition, the orientation of $Q$ is determined by the orientation of $\nu(L)$.

When $L$ admits a nowhere-vanishing closed $1$-form, one can construct a closed $\Gtwo$ resolution by replacing a tubular neighborhood of $L$ on $X$, with a neighborhood of $Q$ in $P$:

\begin{theorem}\cite[Section 3]{LMM} \label{theo:resol-exist}
Let $(M,\vp,g)$ be a closed $\Gtwo$ structure on a compact manifold. Suppose that $\iota \colon M \to M$ is an involution such that $\iota^*\vp=\vp$ with $L=\Fix(\iota) \neq \emptyset$.
If there is a nowhere-vanishing closed $1$-form $\h \in \O^1(L)$, then there exists a closed $\Gtwo$ resolution of the orbifold $(X=M/\iota,\varphi,g)$.
\end{theorem}
\begin{remark}
The proof shows the existence of a one-parameter family of closed $\Gtwo$ structures on the resolution, but this was not stated in the theorem. We will make this explicit in the equivariant version of this method. 
\end{remark}
According to \cite{T}, a nowhere-vanishing closed $1$-form exists on the oriented $3$-dimensional manifold $L$  if and only if every connected component $L_j$ is the mapping torus of an orientation-preserving diffeomorphism $f_j\colon \Sigma_j\to \Sigma_j$ of an orientable surface $\Sigma_j$ with genus of $g_j$.
We now choose an orientation on $\Sigma_j$; since 
$$
TL_j=\langle \partial_t \rangle \oplus \cV_j, \qquad  \cV_j=[0,1]\times T\Sigma_j/(0,v_x)\sim (1,d_xf_j(v_x)),
$$
and $TL_j$ is oriented by $\varphi|_{L_j}$, we orient $\cV_j$ so that the equality $TL_j=\langle \partial_t \rangle \oplus \cV_j$ holds as oriented vector bundles. The fixed orientation is that of $\cV_j|_{t=0}=T\Sigma_j$.

We now recall some facts needed in subsection \ref{sec:cohomology-algebra}.  From the short exact sequence \eqref{eqn:short-mp} and the equality $H^3(L_j)=\R$, it follows that $f_j^*- \rId=0$ on $H^2(\Sigma_j)$. Hence, a volume form $\mathrm{vol}_j$ on $\Sigma_j$
that integrates to $1$ determines, by equation \eqref{eqn:tilde-alpha},  a closed form $\omega_j\in \Omega^2(L_j)$ that yields a non-zero cohomology class. In addition, the splitting $TL_j=\langle \partial_t \rangle \oplus \cV_j$ might not be orthogonal, but an orthogonal orientation-preserving splitting is given by 
\begin{equation} \label{eqn:e1-definition}
    TL=\langle e_1 \rangle \oplus \mathcal{V}_j, \qquad e_1=\frac{dt^\sharp}{\|dt\|},
\end{equation}
because $\ker(dt)=\cV_j$ and $g(e_1,\partial_t)>0$. Indeed,  $J_{\mathcal{V}_j}\colon {\mathcal{V}_j} \to {\mathcal{V}_j}$, $J_{\mathcal{V}_j}(W)=e_1 \times W$, is an almost complex structure compatible with the orientation. 
Finally, the diffeomorphism type of $L_j$ depends on the class of the mapping diffeomorphism $f_j$ in the quotient $\mathrm{Diff}^+(\Sigma_j)/\mathrm{Isot}_0(\Sigma_j)$. Here $\mathrm{Diff}^+(\Sigma_j)$ is the group of orientation-preserving diffeomorphisms and $\mathrm{Isot}_0(\Sigma_j)$ is the subgroup of diffeomorphisms isotopic to the identity. 

\section{Topological properties of resolutions of compact closed $\Gtwo$ orbifolds  $M/\Z_2$} \label{sec:topological}

In this section, we compute the cohomology algebra of the resolution, certain triple Massey products arising from the exceptional divisor, and the real Pontryagin class.  
We work under the hypotheses of Theorem \ref{theo:resol}, and we follow the notation and assumptions in section \ref{subsec:pre-resol}.  
In particular, we highlight that the fixed locus $L$ is assumed to be oriented by $\varphi|_L$. Since the connected components are mapping tori, we also assume that the form used to construct the resolution is $dt$.  
Remarks \ref{rmk-cohomology-c-1} and \ref{rmk-cohomology-c-3} show how the formulas change if a different closed 1-form is chosen.

\subsection{Cohomology algebra of the resolution} \label{sec:cohomology-algebra}
This discussion is based on \cite[section 4]{LMM}, 
with minor corrections to \cite[Lemma 4.2, Proposition 4.3, Proposition 4.5]{LMM}.
We first comment on \cite[Proposition 4.1, Lemma 4.2]{LMM}. First, \cite[Proposition 4.1]{LMM} shows that the real first Chern class of $\nu(L_j)$ with respect to the complex structure determined by $e_1=\frac{dt^\sharp}{\|dt\|}$ through equation \eqref{eqn:acs-norml} is 
$$
c_1(\nu(L_j))=(2-2g_j)[\omega_j].
$$
We sketch the proof, since it will be used later. First, one shows that $\nu(L_j)$ is complex isomorphic to $\underline{\CC} \oplus \cV_j$, where $\underline{\CC}$ is the trivial line bundle. The complex structures considered on $\cV_j$ is that induced by $e_1$. The trivial factor is determined by a unit-length section $\eta \colon L_j \to \nu(L_j)$, which exists because the rank of the bundle is higher than the dimension of the manifold.  The explicit isomorphism is:
$$
A \colon \underline{\CC} \oplus \cV_j \to \nu(L_j), \qquad A(u+iv,W_p)=u\,\eta(p) + v\,e_1(p) \times \eta (p) +   W_p \times \eta(p).
$$
This preserves the complex structure because
$(e_1 \times W_p)\times \eta(p)= e_1 \times (W_p \times \eta(p))$ (see \cite[Lemma 2.6]{SW}). Hence, $c_1(\nu(L_j))=c_1(\cV_j)$. Recall that for complex line bundles, the first Chern class coincides with the Euler class of the bundle, viewed as an oriented real rank 2 bundle. We compute the Euler class $e(\cV_j)$ by
 pulling back a Thom form of $\cV_j \to L_j$ via the zero section $s\colon L_j \to \mathcal{V}_j$. This Thom form is obtained from a Thom form $\upsilon_j$ of $T\Sigma_j \to \Sigma_j$
using that $\cV_j$ is the mapping torus of $df_j\colon T\Sigma_j \to T\Sigma_j$. More precisely,
$(df_j)^*(\upsilon_j)$ is another Thom form of $T\Sigma_j \to \Sigma_j$ because $df_j$ preserves the orientation of the fibers. Hence, $[\upsilon_j]\in \ker ((df_j)^* - \rId)$, and equation \eqref{eqn:tilde-alpha} allows us to extend $\upsilon_j$:
 $$
 \tau_j=(df_j)^*(\upsilon_j) + d(\rho(t)\beta_j)\in \Omega_c^2(\cV_j), \mbox{ where } d\beta_j=\upsilon_j - (df_j)^*\upsilon_j.
 $$
This is a Thom form because it integrates to $1$ over the fibers. 
Using that $s[t,p]=[t,s_0(p)]$, where $s_0\colon \Sigma_j \to T\Sigma_j$ is the zero section, and $df_j\circ s_0 = s_0 \circ f_j$ one deduces that
$$
s^*(\tau_j)
= f_j^*(s_0^*(\upsilon_j)) + d(\rho(t)s_0^*(\beta_j)).
$$
The right hand side is the closed form determined by $s_0^*(\upsilon_j)$ using equation \eqref{eqn:tilde-alpha}. The claimed identity follows from $[s_0^*(\upsilon_j)]=c_1(\Sigma_j)=(2-2g_j)[\mathrm{vol}_j]$.

The proof of \cite[Lemma 4.2]{LMM} requires that the first Chern class, in $H^2(L_j, \Z)$, is even. This holds, in particular, if the Chern class of $\nu(L_j)$ coincides with $(2-2g_j)PD[S^1]$, where a $S^1$ is a vertical loop (for instance, $t \mapsto [t,x]$ for some fixed point $x\in \Sigma_j$ of $f_j\colon \Sigma_j \to \Sigma_j$). However, this is not guaranteed by the the proof of \cite[Proposition 4,1]{LMM}, which is why we now introduce an additional hypothesis.

\begin{lemma}\cite[Correction of Lemma 4.2]{LMM}
If the first Chern class of $\nu(L_j)$ is even,
then $Q_j$ is a trivial $\CP^1$ bundle.
\end{lemma}

This Lemma holds in two special cases that are of particular interest to us: when $\Sigma_j=T^2$, and when $f_j$ is isotopic to the identity. 
\begin{enumerate}
    \item If $\Sigma_j=T^2$, we claim that $\cV_j$ is trivial as a real oriented bundle, and hence its first Chern class vanishes.
To check this, we set $T^2=\R^2/\Z_2$, 
so that $\mathrm{Diff}^+(T^2)/\mathrm{Iso}_0(T^2)\cong \mathrm{SL}(2,\Z)$. We therefore assume $f_j[x_1,x_2]=[A(x_1,x_2)]$ for some $A\in \mathrm{SL}(2,\Z)$.
Let $A_s$ be a path in $\mathrm{SL}(2,\R)$ with $A_0=\mathrm{Id}$ and $A_1=A$. Then 
$$
E_s= [0,1]\times T^2 \times \R^2/(0,x,w)\sim (1,A(x),A_s(w)),
$$
defines a homotopy between the trivial bundle $E_0$ and $\cV_j$, and hence these bundles are isomorphic.
\item If $f_j$ is isotopic to the identity, then $L_j=S^1 \times \Sigma_j$ so that $\mathcal{V}_j$ is the pullback of $T\Sigma_j$. Its first Chern class is then $(2-2g_j)PD[S^1]$, which is even.
\end{enumerate} 

We now fix a tubular neighborhood $U_j$ of $L_j$, and set $O_j=\rho^{-1}(U_j)$. Let $\mathrm{pr}_j\colon O_j \to L_j$ denote the composition of $\rho$ and the nearest point projection $\bar{\pi}_j\colon U_j \to L_j$: this is a smooth map. We denote by $\tau_j$ a Thom form of $Q_j$ supported in $O_j$.

Since $Q_j=\mathbb{P}(\nu(L_j)) $, 
the standard theory of Chern classes (see \cite[p. 269--271]{Bott-Tu}) allows us to find the cohomology algebra of $Q_j$.

\begin{proposition}\cite[Correction of Proposition 4.3]{LMM}\label{prop:cohomology-exceptional-divisor}
Consider the tautological line bundle over $Q_j$,
$
S_j=\{ (v_x,\ell_x)\in \nu_x(L_j)\times \mathbb{P}(\nu_x(L_j)), \, v_x \in \ell_x \},
$
and let 
$\rE_j\in H^2(Q_j)$ be its Euler class. Denote by $H^*(L_j)[\by_j]$ the algebra of polynomials with coeffiecients in $H^*(L_j)$.
Then,
$$
F_j \colon H^*(L_j)[\by_j]/\la \by^2_j+ (2-2g_j) [\o_j] \by_j \ra \to H^{*}(Q_j),\qquad  F_j(\b)= \mathrm{pr}^*\b|_{Q_j}, \quad F_j(\by_j)=-\rE_j,
$$
is an isomorphism of algebras.
\end{proposition}

\begin{remark} There is a typo in \cite[Proposition 4.3]{LMM}, where we defined  $F(\by_j)=\rE_j$.
\end{remark}

\begin{remark} \label{rmk: Euler-class}
We now argue that $\rE_j= \frac{1}{2} [\tau_j]|_{Q_j}$.
Let $\rQ_j\colon S_j \to P_j$ be the quotient projection,  we claim that the Thom form of the bundle $S_j \to Q_j$ equals $\frac{1}{2}[\rQ_j^*(\tau_j)]$. This holds because the fiber of the normal bundle of $Q_j$ in $S_j$ at a point $x\in Q_j$ is $F_x=\CC$, whereas in $P_j$ it is $F_x/\ZZ_2$, hence:
$$
\int_{F_x}{\rQ_j^*(\tau_j)}= 2 \int_{F_x/\ZZ_2}{\tau_j}=2.
$$
Therefore, $\rE_j= \frac{1}{2}[\rQ_j^*\tau_j]|_{Q_j}=\frac{1}{2}[\tau_j]|_{Q_j}$. This is consistent with the fact that $P_j|_x \to Q_j|_x\cong \CP^1$ is a line bundle, isomorphic to $S_j|_x\otimes S_j|_x$.
\end{remark}

Concerning the cohomology groups of the resolution, we have the following result
\begin{proposition}\cite[Proposition 4.4]{LMM}\label{prop:cohomology-short-sequence}
There is a split short exact sequence:
$$
\xymatrix{ 0 \ar[r] & H^*(X) \ar[r]^{\rho^*} & H^*(\widetilde{X}) \ar[r]^{\oplus_j r_j \qquad\quad } & \bigoplus_{j} H^{*}(Q_j)/H^*(L_j) \ar[r] & 0},
$$
where $r_j$ denotes the restriction $H^*(\widetilde{X})  \to H^*(Q_j)$ followed by projection to the quotient. 
\end{proposition}

The cohomology of $X$ is calculated from $(\Omega(X),d)=(\Omega(M)^{\Z_2},d)$ (see \cite[Section 2.1]{LMM}). Similarly to
\cite[Lemma 5]{LMM-2}, and using the Poincar\'e's Lemma, representatives of cohomology classes can be chosen such that their pullbacks are smooth. More generally, one can prove the following:
\begin{lemma}\cite[Lemma 5]{LMM-2} \label{lem: smooth extension}
Let $U_j'\subset U_j$ be a tubular neighborhood of $L_j$, and let $\alpha \in \Omega^k(X)$.
If  $d\beta|_{ \sqcup U_j'}=0$, then there is $\beta \in \Omega^{k-1}(X)$, supported in $\sqcup U_j'$ such that $\rho^*(\alpha+ d\beta)$ is smooth.
\end{lemma}

In addition, Proposition \ref{prop:cohomology-exceptional-divisor} implies:
$$
H^*(Q_j)/H^*(L_j)\cong 
H^{*-2}(L_j)\otimes \langle \bx_j \rangle, \qquad \deg(\bx_j)=2.
$$
Here, $\bx_j=-2\by_j$, so that $F_j(\bx_j)=2\rE_j=[\tau_j]|_{Q_j}$ by Remark \ref{rmk: Euler-class}.
A splitting is
$$
H^{*-2}(L_j)\otimes \langle \bx_j \rangle \to H^*(\widetilde{X}), \qquad
[\b] \otimes \bx_j \longmapsto [\mathrm{pr}_j^*\b \wedge \t_j].
$$
Regarding the product structure on $H^*(\widetilde{X})$, we now amend a minor error in \cite[Proposition 4.5]{LMM}. More precisely, we now check:
\begin{equation}\label{eqn:square-thom}
 [\tau_j^2]=-2\rho^*(\rTh[L_j]) +(4-4g_j)[\mathrm{pr}_j^*(\omega_j)\wedge \tau_j].
\end{equation}
According to Proposition \ref{prop:cohomology-exceptional-divisor},  $[\tau_j^2]|_{Q_j}=F_j(4\by_j^2)=(4-4g_j) F_j([\omega_j](-2\by_j))= (4-4g_j) [\mathrm{pr}_j^*(\omega_j)\wedge \tau_j]|_{Q_j}$. From Proposition  \ref{prop:cohomology-short-sequence} we deduce,
 $$
 [\tau_j^2-(4-4g_j)\mathrm{pr_j}^*\omega_j\wedge \tau_j] \in \mathrm{Im}(\rho^*).
 $$
  We construct a new representative supported in $U_j$ that vanishes around $Q_j$.
By Poincaré's Lemma, there is $\beta_j \in \Omega^3(O_j)$ such that $d\beta_j= \left(\tau_j^2 - (4-4g_j)\mathrm{pr}_j^*\omega_j \wedge \tau_j\right)|_{O_j}$. 
Let $b$ be a bump function supported in $O_j$, equal to $1$ on a smaller tubular neighborhood $V_j\subset O_j$. The representative is
$$
\gamma= \tau_j^2-(4-4g_j)\mathrm{pr_j}^*\omega_j\wedge \tau_j - d(b\beta_j).
$$ 
The pushfoward $\rho_*\gamma$ is then well-defined, and $[\rho_*\gamma]\in H^4_c(U_j)$. Equation \eqref{eqn:square-thom} will be proved once we show $[\rho_*\gamma]=-2\rTh[L_j]$ because $\rho^*[\rho_*\gamma]= [\tau_j^2-(4-4g_j)\mathrm{pr_j}^*\omega_j\wedge \tau_j]$. 
For that purpose, it is sufficient to check that 
$\int_{\nu_x(L_j)/\iota_*}{\rho_*(\gamma)}=-2$, for all $x\in L_j$. Using that $\mathrm{pr}_j^*(\omega_j)|_{P_x}=0$, and that $\tau_j|_{Q_x}$ is the Thom form of $Q_x\subset P_x$ we obtain,
$$
\int_{\nu_x(L_j)/\iota_*}{\rho_*\gamma}= \int_{\nu_x(L_j)/\iota_*-\{0_x\}}{\rho_*\gamma}
= \int_{P_x-Q_x}{\gamma}
= \int_{P_x}{\tau_j^2}= \int_{Q_x}{\tau_j}= [Q_x][Q_x]=-2.
$$
This finishes the proof of formula \eqref{eqn:square-thom}, and yields to:

\begin{proposition}\cite[Correction of Proposition 4.5]{LMM} \label{prop:cohom-alg}
There is an isomorphism
$$
 H^*(\widetilde{X})= H^*(X) \bigoplus \oplus_j  H^{*-2}(L_j)\otimes \la \bx_j \ra.
$$
Let $[\a], [\b] \in H^*(X)$, $[\g_j],[\gamma_j'] \in H^*(L_j)$. The wedge product on $ H^*(\widetilde{X})$ determines the following product on the left hand side:  
\begin{enumerate}
\item $[\a]\cdot[\b] = [\a \wedge \beta]$,
\item $[\a] \cdot [\g_j]\otimes \bx_j = [\alpha|_{L_j} \wedge \g_j] \otimes \bx_j$,
\item $[\g_j]\otimes \bx_j \cdot [\g_k'] \otimes \bx_k=0$ if $j\neq k$,
\item $[\g_j]\otimes \bx_j \cdot  [\g_j'] \otimes \bx_j = -2\, [\bar{\pi}_j^*(\g_j\wedge \g_j')]\wedge \mathrm{Th}[L_j] +(4-4g_j)[\o_j] \otimes \bx_j  $.
\end{enumerate}
\end{proposition}

\begin{remark} \label{rmk:sqtau-trivial} When $g_j=1$, we argued that $\nu(L_j)=L_j\times \CC^2$.  Then, similar to the discussion in \cite[Section 3.2]{LMM-2}, letting $r=\|z\|$ we can write:
\begin{equation} \label{eqn:Thom-trivial}
\tau_j= \frac{1}{\pi i} d(b(r^2)\partial \log(r^2))=\frac{2}{\pi i}b'(r^2)rdr \wedge \partial \log(r^2)+ \frac{b(r^2)}{\pi i}\overline{\partial}{\partial}\log r^2,
\end{equation}
where $b(r)$ is an increasing function  that equals $-1$ on $r \leq \e^2$ and
$0$ on $r \geq 4\e^2$.
From this, one deduces that $\tau^2_j$ vanishes on a tubular neighborhood of $L_j\times \CP^1$. Therefore, $\rho_*(\tau^2_j)$ is well-defined and a similar argument shows that $-\frac{1}{2} \rho_*(\tau^2_j)$ is a Thom form of $L_j$.
\end{remark}

\begin{remark}\label{rmk:product-dif}
Assume that the components $L_1$ and $L_2$ of the singular locus satisfy $g_1=g_2$, $\mathrm{Th}[L_1]=\mathrm{Th}[L_2]$ and there is $[\alpha]\in H^2(X)$ such that $[\alpha|_{L_j}]=[\omega_{j}]$ if $g_1\neq 1$; otherwise we consider $\alpha=0$. Then, Proposition \ref{prop:cohom-alg} implies:
$$
(\bx_1-\bx_2)\cdot (\bx_1+\bx_2-(4-4g_1)[\alpha])=\bx_1^2-(4-4g_1)[\omega_1] \otimes \bx_1 - (\bx^2- (4-4g_2)[\omega_2] \otimes \bx_2) = -2\mathrm{Th}[L_1]+2\mathrm{Th}[L_2]=0.
$$
\end{remark}
 \begin{remark} \label{rmk-cohomology-c-1}
If we instead construct the resolution using $\theta\in \Omega^1(L)$, closed and nowhere-vanishing, 
 there is an isomorphism $\nu(L_j)\cong \underline{\C}\oplus \ker(\theta|_{L_j})$, where the second summand is endowed with the complex structure determined by $X \mapsto \frac{\theta^\sharp}{\|\theta\|} \times X$. Hence, $c_1(\nu(L_j))=c_1(\ker(\theta)|_{L_j})$. One can then follow the argument replacing $(2-2g_j)[\omega_j]$ with $c_1(\ker(\theta))$, and obtain the product formula
$$
[\g_j] \otimes \bx_j \cdot [\g_j'] \otimes \bx_j = -2\, [\bar{\pi}_j^*(\g_j\wedge \g_j')]\wedge \mathrm{Th}[L_j] + 2\, c_1(\ker( \theta|_{L_j})) \otimes \bx_j,
$$
where $ \bx_j$ represents the Thom class of the exceptional divisor.
 \end{remark}

\subsection{Triple Massey products}

The following result provides a geometric criterion for the vanishing of certain triple Massey products. We denote by $-L_j$ to the submanifold $L_j$ equipped with the reversed orientation, i.e, the one induced by $-\varphi|_{L_j}$.

\begin{proposition}\label{prop:massey-linking}
Let $L_1,L_2,L_3,L_4$ be different connected components of $L$ diffeomorphic to a mapping torus with fiber $T^2$, and such that $PD[L_1]=PD[L_2]$ and $PD[L_3]=PD[L_4]$ on $H^4(X)$.  
 
 If  $\mathrm{lk}_M( L_3\sqcup -L_4, L_1\sqcup -L_2)\neq 0$, then 
$
\la \rTh[Q_1]+  \rTh[Q_2], \rTh[Q_3]+ \rTh[Q_4], \rTh[Q_3]- \rTh[Q_4] \ra$
does not vanish.
\end{proposition}
\begin{proof}
The linking number is well-defined because $L_1 \sqcup -L_2$ and $L_3 \sqcup -L_4$ are nullhomologous on $H_3(M,\R)$ as a consequence of Remark \ref{rem:ThomOrbibundle}. The triple Massey product is well-defined according to Proposition \ref{prop:cohom-alg} and Remark  \ref{rmk:product-dif}. To compute it, we pick the representatives $\tau_j$ of $\mathrm{Th}[Q_j]$ provided by equation \eqref{eqn:Thom-trivial}. 
Let $\beta\in \Omega^3(M)^{\Z_2}$ such that 
$$
d\beta=\rho_*(\tau_3^2 - \tau_4^2).
$$
Since both $\rho_*\tau_3^2$ and $\rho_*\tau_4^2$ vanish on a neighborhood of the singular locus (see Remark \ref{rmk:sqtau-trivial}), Lemma \ref{lem: smooth extension} allows us to assume that 
$\rho^*(\beta)$ is smooth.  The triple Massey product does not vanish if and only if 
$$
[\gamma]=[(\tau_1+\tau_2)\wedge \rho^*\beta] \notin \mathcal{I}, \mbox{ where }  \mathcal{I}=[\tau_1 + \tau_2]H^4(\widetilde{X}) + [\tau_3 - \tau_4]H^4(\widetilde{X}).
$$
Observe that if $[\gamma]\wedge[\tau_1 -\tau_2]=[(\tau_1^2 - \tau_2^2)\wedge \rho^{*}(\beta)]\neq 0$, then $[\gamma] \notin \mathcal{I}$ because
$[\tau_1+\tau_2]\wedge [\tau_1-\tau_2]=0$ by Remark  \ref{rmk:product-dif} and
$(\tau_3 - \tau_4)\wedge(\tau_1 - \tau_2)=0$.  The result is a consequence of the fact that $\int_{\widetilde{X}}{\gamma\wedge(\tau_1 -\tau_2)}$ and $\mathrm{lk}_M(L_3\sqcup -L_4, L_1\sqcup -L_2)$ are proportional, which we now prove. We first observe:
$$
\int_{\widetilde{X}}{\gamma\wedge(\tau_1 -\tau_2)}=\int_{X}{\rho_*(\tau_1^2-\tau_2^2)\wedge \beta}
=\frac{1}{2}\int_{M}{\rho_*(\tau_1^2-\tau_2^2)\wedge \beta}.
$$
Moreover, $ \upsilon_j = -\frac{1}{2} \rho_*(\tau_j^2) $ is a Thom form for $L_j$ by Remark \ref{rem:ThomOrbibundle}. Remark \ref{rmk:sqtau-trivial} implies that  $ -\frac{1}{4} \rho_*(\tau_j^2)  $ is a Thom form for  $ L_j  $ in  $ M$.
 Hence, for $(j,k)\in \{(1,2),(3,4)\}$, we have 
$
\rTh[L_j\sqcup -L_k]= \rTh[L_j]- \rTh[L_k]=-\frac{1}{4}[\rho_*(\tau_j^2-\tau_k^2)],
$
on $M$, and $-\frac{1}{4}\beta$ is a primitive of the representative  $-\frac{1}{4}\rho_*(\tau_3^2-\tau_4^2)$ of $\rTh[L_3\sqcup -L_4]$. Hence,
$$
\mathrm{lk}(L_3 \sqcup -L_4, L_1\sqcup -L_2)= \frac{1}{16} \, \int_{M}{\beta\wedge \rho_*(\tau_1^2-\tau_2^2)}= \frac{1}{8} \int_{\widetilde{X}}{\gamma\wedge(\tau_1 -\tau_2)}.
$$
\end{proof}

\subsection{First Pontryagin class}

To compute $p_1(\widetilde{X})$ we adapt an argument from \cite{Gei-Pasq}, originally used to  
obtain the Chern classes of a blow-up in complex and symplectic geometry. We use some basic properties of the real Pontryagin classes, defined as $p_k(E)=(-1)^k c_{2k}(E\otimes \C)$, that can be found in 
\cite[p. 285--291]{Bott-Tu} (but note that they use a different sign convention). Note that $c(E\otimes \C)= 1 + \sum_{k\geq 1} (-1)^k p_k(E)$ satisfies
the Withney product formula, naturality under pullback, and the identity $c(E\otimes \C)=c(E)c(E^*)$ for any complex bundle $E\to B$, where $E^*$ denotes the hermitian dual bundle. The last formula implies $p_1(E)=c_1(E)^2-2c_2(E)$, using $c_k(E^*)=(-1)^k c(E)$. It is also convenient to denote $\tilde{p}(M)=c(TM\otimes \C)$, when $M$ is a smooth manifold.
We also use two Lemmas, the first is well-known, but the proof is included for completeness. The second is specific to our context, and its proof mimicks a method used to compute the Chern classes of $\CP^n$ (see \cite[pages 280--281]{Bott-Tu}).

\begin{lemma}\label{lem:relative-to-complement}
Let $X$ be a smooth manifold. Let $Y=\sqcup Y_j$ be a disjoint union of submanifolds with codimension $r$, and let $\tau_j$ be a Thom form for $Y_j$ supported on a tubular neighborhood $B_\delta(Y_j)$. Consider the nearest point projection $\mathrm{pr}_Y \colon \sqcup B_\delta(Y_j) \to Y$, then
$$
H^*(X,X-Y)\cong \oplus_j \mathrm{pr}^* H^{*-r}(Y_j)\wedge [\tau_j].
$$
\end{lemma}
\begin{proof}
Applying excision (taking the open set to be $X-\sqcup D_\d(Y_j)$, where $D_\d(Y_j)$ denotes the closure of $B_\d(Y_j)$), and using the fact that $D_\delta(Y_j)-Y_j$ retracts to the sphere $S_\d(Y_j)$, we obtain
$$
H^*(X,X-Y)=\oplus_j H^*(D_\d(Y_j),D_\d(Y_j)-Y_j)
= \oplus_j H^*(D_\d(Y_j),S_\d(Y_j))= \oplus_j H_c^*(B_\d(Y_j)).
$$
Finally, Thom isomorphism theorem ensures $H_c^*(B_\d(Y_j))\cong \mathrm{pr}^*H^{*-r}(Y_j)\wedge [\tau_j]$.
\end{proof}

\begin{lemma}
Let $\pi \colon E\to L$ be a complex vector bundle of complex rank $r$. Denote $Q=\mathbb{P}(E)$ and $\pi \colon Q \to L$ the projection map. Let $S\to Q$ be the tautological line bundle and let $\rE$ be its Euler class. Then,
\begin{equation} \label{eqn:formula-p-exceptional}
\tilde{p}(Q)= \pi^*(\tilde{p}(L)) \left( \sum_{j,k=0}^r {(-1)^{j+k}\pi^*(c_j(E)c_k(E))(1-\rE)^{r-j}(1+\rE)^{r-k}}\right).
\end{equation}  
In particular, if $L$ is $3$-dimensional, then
\begin{equation} \label{eqn:formula-p-exceptional-3d}
p_1(Q)= 0, \quad \mbox{ if r=2}, \qquad p_1(Q)=3\rE^2- 2\pi^*(c_1(E))\rE, \quad \mbox{ if r=3}.
\end{equation}

\end{lemma}
\begin{proof}
The short exact sequence,
$
0 \to \ker(d\pi) \to TQ \to \pi^*(TL) \to 0,
$
yields $\tilde{p}(Q)=c(\ker(d\pi)\otimes \C) \pi^*(\tilde{p}(L))$. We now claim that $\ker(d\pi)=\mathrm{Hom}_\C(S,E') $
where $E'$ is the cokernel of the inclusion $S \hookrightarrow \pi^*(E)$, which can be identified (once a hermitian metric on $E$ is fixed) with
$
E'=\{ (v_x,\ell_x) \in E_x\times Q_x, v_x \perp \ell_x \}.
$
The isomorphism is
$$
\mathrm{Hom}_\C(S,E') \to \ker(d\pi),\quad 
F\colon S_{\ell_x} \to E_{\ell_x}' \longmapsto \frac{d}{dt}\bigg|_{t=0} [v_x+ tF(v_x)] , \quad \mbox{ if } 0\neq v_x \in \ell_x.
$$
Tensoring the short exact sequence $0 \to S \to \pi^*(E) \to E' \to 0$ with $S^*$ we obtain,
$$
0 \to \underline{\C} \to S^*\otimes \pi^*E \to \ker(d\pi) \to 0.
$$
Hence $c(\ker(d\pi)\otimes \C)=c(S^*\otimes \pi^*E\otimes \C)=c(S^*\otimes \pi^*E)c(S\otimes (\pi^*E)^*)$.
Applying the formula for the Chern classes of the tensor product of two bundles (see for instance \cite[formula (21.10), page 279]{Bott-Tu}), 
we obtain:
$$
c(S^*\otimes \pi^*E)=\sum_{k=0}^r {\pi^*(c_k(E))(1-\rE)^{r-k}},
\quad c(S\otimes (\pi^*E)^*)=\sum_{k=0}^r {(-1)^k \pi^*(c_k(E))(1+\rE)^{r-k}},
$$
where $c_0(E)=1$.
Equation \eqref{eqn:formula-p-exceptional} follows from $\tilde{p}(Q)=c(S^*\otimes \pi^*E)c(S\otimes (\pi^*E)^*)\pi^*(\tilde{p}(L))$. If $L$ is $3$-dimensional, then $c(E)=1+c_1(E)$, $c_1(E)^2=0$, and $\tilde{p}(L)=1$. Equation \eqref{eqn:formula-p-exceptional} yields,
$$
\tilde{p}(Q)=(1-\rE^2)^{r-1}(1-\rE^2 + 2 \rE\, \pi^*(c_1(E)) ) .
$$
If $r=2$, using that $\rE^2- \pi^*(c_1(E))\rE=0$ (by \cite[p. 269--271]{Bott-Tu}) and $\dim(Q)=5$ (so that $\rE^j=0$ for $j\geq 3$), we obtain 
$\tilde{p}(Q)=1-2\rE^2 + 2\pi^*(c_1(E))\rE=1$. Similarly, if $r=3$ then $\dim(Q)=7$ (so that $\rE^j=\pi^*(c_1(E))\rE^{j-1}=0$ for $j\geq 4$), and
$
\tilde{p}(Q)=(1-2\rE^2)(1-\rE^2 + 2 \rE\, \pi^*(c_1(E)))= 1-3\rE^2+ 2\pi^*(c_1(E))\rE .
$ Equation \eqref{eqn:formula-p-exceptional-3d} follows from this.
\end{proof}

We let $\cY=M\times (\C/\Z\langle 1,i\rangle)$,  and define the involution $\iota_\cY=(\iota,-\rId)$. The singular locus of $\cX= \cY/\iota_\cY$ is $\cL=\sqcup L_{j,\be}$, with  $L_{j,\be}=L_j\times \{\be\}$, where $L_j$ is a singular component of $M/\iota$, and $\be \in \Z\langle \frac{1}{2},\frac{i}{2}\rangle/\Z\langle 1,i\rangle$. We endow $\cY$ with the product orientation, and orient $L_{j,\e}$ as $L_j$. Note that $\nu(L_{j,\be})=(\nu(L_j)\times \{\be\})\oplus \underline{\C}$ as oriented vector bundles.

A tubular neighborhood of $\cL$ in the orbifold is diffeomorphic to a neighborhood of the zero section of the orbibundle $\nu(L_{j,\be})/\Z_2$, where $\Z_2=\{\pm \rId\}$. Using the notation introduced in section \ref{subsec:pre-resol}, this can be described as
$$
\nu(L_{j,\be})/\Z_2 \cong P_{U(2)}(L_{j,\be}) \times_{\U{2}} \C^3/\ZZ_2, \qquad P_{\U{2}}(L_{j,\be})=P_{\U{2}}(\nu(L_j))\times \{\be\}.
$$
Of course, $\U{2}\subset \U{3}$ in the standard way. We resolve $\nu(L_{j,\be})/\Z_2$ by blowing-up the zero section. Namely, we replace the zero section with a neighborhood of $\cQ_{j,\be}= P_{\U{2}}(L_{j,\be})\times_{\U{2}} \CP^2$ on $\cP_{j,\be}= P_{\U{2}}(L_{j,\be}) \times_{\U{2}} \widetilde{\C^3}/\Z_2$, where
$$
\widetilde{\CC^3}/\ZZ_2 = \{ (z,\ell)\in \C^3 \times \CP^2, z \in \ell \}/ (z,\ell)\sim (-z,\ell).
$$
Set $\cP=\sqcup_{j,\be} \cP_{j,\be}$ and $\cQ=\sqcup_{j,\be} \cQ_{j,\be}$. The orientation of 
$\cP$ is such that diffeomorphism $\cP-\cQ\to \nu(\cL)/\Z_2-\cL$ preserves orientations, the orientation of $\cQ$ is determined by the orientation of $\nu(\cL)$.
The fiberwise blow-up produces a resolution  $\rho \colon \widetilde{\cX}\to \cX$.
Recall that the strict transform of $\mathcal{Z}^\vee \subset \cX$ is the closure of $\rho^{-1}(\mathcal{Z}^\vee-\cL)$ on $\widetilde{\cX}$.
The strict transform of $M\times \{0\}/\iota_{\cY}$, 
is diffeomorphic to  $\widetilde{X}$ due to the following commutative diagram:
\begin{equation}\label{diagram:resol-pont}
\begin{tikzcd}
P_{\U{2}}(L_{j,\bz}) \times_{\U{2}} \widetilde{\CC^3}/\ZZ_2 \arrow{r}{} 
& P_{\U{2}}(L_{j,\bz}) \times_{\U{2}} \C^3/\Z_2\\
P_{\U{2}}(\nu(L_j)) \times_{\U{2}} \widetilde{\C^2}/\Z_2 \arrow{u} \arrow{r}{} & P_{\U{2}}(\nu(L_j)) \times_{\U{2}} \C^2/\Z_2\arrow{u}.
\end{tikzcd}
\end{equation}
To ease notations we identify $X$ and $M$ with  $M\times \{0\}/\iota_{\cY}$ and $M\times \{0\}$ respectively. We also denote by $\widetilde{X}$ to the strict transform of $M\times \{0\}/\iota_{\cY}$.  We fix a tubular neighborhood $B_\delta(\cQ_{j,\be})$ of $\cQ_{j,\be}$, and we
 denote $\mathrm{pr}\colon \sqcup B_{\delta}(\cQ_{j,\be}) \to \sqcup L_{j,\be}$ to the composition of the nearest point projection to $\mathrm{pr}_{\cQ}\colon \sqcup B_{\delta}(\cQ_{j,\be}) \to \cQ$, and the resolution map $\rho|_\cQ \colon \cQ \to \cL$.

Let $\upsilon_{j,\be}$ be a Thom form of $\cQ_{j,\be}$ supported on $B_\delta(\cQ_{j,\be})$. Observe that
$\upsilon_{j,\bz}|_{\widetilde{X}}$ is a Thom form for $Q_j$ on $\widetilde{X}$ because  $P_j|_{x}=(\cP_{j,\bz})|_{x}$ if $x \in Q_j$. 
Thus 
\begin{equation} \label{eqn:restriction-tau-square}
    [\upsilon_{j,\bz}^2]|_{\widetilde{X}}=[\tau_j^2]=-2\rho^*(\mathrm{Th}[L_j]) + (4-4g_j)[\mathrm{pr}_j^*\omega_j \wedge\tau_j] .
\end{equation}
Similar to Proposition \ref{prop:cohomology-exceptional-divisor}, using $\nu(L_{j,\be})\cong \nu(L_{j})\oplus \underline{\C}$, and denoting by $[\o_{j,\be}]$ the class determined by the unit-length volume form of the fiber (see section \ref{subsec:pre-resol}), we obtain:

\begin{proposition}\label{prop:cohomology-exceptional-divisor-v2}
Consider the tautological line bundle
$$
\cS_{j,\be}=\{ (v_x,\ell_x)\in \nu_x(L_{j,\be})\times \mathbb{P}(\nu_x(L_{j,\be})), \quad v_x \in \ell_x \} \to \cQ_{j,\be} ,
$$
and let $\rE_{j,\be}$ be its Euler class. Denote by $H^*(L_{j,\be})[\by_{j,\be}]$ the algebra of polynomials with coeffiecients in $H^*(L_{j,\be})$.
Then,
$$
F_{j,\be} \colon H^*(L_{j,\be})[\by_{j,\be}]/\la \by^3_j+ (2-2g_j) [\o_{j,\be}] \by_{j,\be} \ra \to H^{*}(\cQ_{j,\be}),\qquad  F_{j,\be}(\b)= \mathrm{pr}^*\b|_{\cQ_{j,\be}}, \quad F_j(\by_{j,\be})=-\rE_{j,\be},
$$
is an algebra isomorphism.
\end{proposition}

As in Remark \ref{rmk: Euler-class}, one shows  
\begin{equation} \label{eqn: Euler-Thom}
\rE_{j,\be}= \frac{1}{2} [\upsilon_{j,\be}]|_{\cQ_{j,\be}}
\end{equation}
 Similarly to Proposition \ref{prop:cohomology-short-sequence}, one obtains a short exact sequence
 \begin{equation} \label{eqn:short-p1}
     0 \to H^*(\cX)  \to H^*(\widetilde{\cX}) \to  \oplus_{j,\be} H^*(\cQ_{j,\be})/H^*(L_{j,\be}) \to 0,
  \end{equation}
where the first map is $\rho^*$, and the second  is the composition of the
restriction $H^*(\widetilde{\cX})  \to \oplus_{j,\be} H^*(\cQ_{j,\be})$ and the projection to the quotient. Proposition \ref{prop:cohomology-exceptional-divisor-v2} implies
\begin{equation} \label{eqn:forms-resolution}
H^k(\widetilde{\cX})\cong H^k(\cX) \oplus \bigoplus_{j,\be} \left(\mathrm{pr}^*H^{k-2}(L_{j,\be})\wedge [\upsilon_{j,\be}] \oplus \mathrm{pr}^*H^{k-4}(L_{j,\be}) \wedge [\upsilon_{j,\be}]^2\right)  .
\end{equation}
We establish some identities for the product in $H^*(\widetilde{\cX})$, sketching the proofs only, as they are analogous to those in Proposition \ref{prop:cohom-alg}. 
\begin{lemma}
The following identities hold on $H^*(\widetilde{\cX})$:
\begin{align} \label{eqn:thom-cube}
    [\upsilon_{j,\be}]^3 =& 4 \mathrm{Th}[L_{j,\be}] -(8-8g_j)[\mathrm{pr}^*\omega_{j,\be}\wedge \upsilon_{j,\be}],\\
 \label{eqn:mixed-product}
    \rho^*[\alpha]\wedge[\upsilon_{j,\be}]=&  [\mathrm{pr}^*(\alpha|_{L_{j,\be}}) \wedge \upsilon_{j,\be}].
\end{align}
\end{lemma}
\begin{proof}
Using $F_{j,\be}(\by_{j,\be})=-\rE_{j,\bz}=-\frac{1}{2} [\upsilon_{j,\be}]|_{\cQ_{j,\be}}$, Proposition \ref{prop:cohomology-exceptional-divisor-v2}, and equation \eqref{eqn:short-p1} we obtain
$[\upsilon_{j,\be}^3 + (8-8g_j)\mathrm{pr}^* \omega_{j,\be}\wedge \upsilon_{j,\be}]=[\rho^*(\gamma)]$, where $\gamma$ is supported in $\rho(B_{\delta}(\cQ_{j,\be}))$. In addition, for $x\in L_{j,\be}$ we have:
$$
\int_{\nu_x(L_{j,\be})/\Z_2} \gamma = \int_{(\cP_{j,\be})_x}{\upsilon_{j,\be}^3}=\int_{(\cQ_{j,\be})_x}{\upsilon_{j,\be}^2},
$$
where we used that $\mathrm{pr}^*\omega_{j,\be}|_{P_x}=0$, and that $[\upsilon_{j,\be}]|_{P_x}$ is the Poincar\'e dual to $(\cQ_{j,\be})_x$. In addition, $\frac{1}{2} [\upsilon_{j,\be}]|_{(\cQ_{j,\be})_x}$ is the Euler class of the bundle $(\cS_{j,\be})_x \to ({\cQ}_{j,\be})_x $. Identifying this with the tautological line bundle over $\CP^2$, the Euler class is $-[\eta_H]$, where $\eta_H=\frac{i}{2\pi}\partial \overline{\partial} \log{r^2}$ represents the class of a hyperplane. Hence,
$$
\int_{\nu_x(L_{j,\be})/\Z_2} \gamma= 4 \int_{\CP^2}{\eta_H^2}=4,
$$
and $[\gamma]=4\mathrm{Th}[L_{j,\be}]$. This implies equation \eqref{eqn:thom-cube}. Equation \eqref{eqn:mixed-product} follows by representing  $[\alpha]$ by a form $\alpha+d\beta$, where $\beta\in \Omega^3(\mathcal{M})^{\Z_{2}}$ is supported in $\sqcup B_{2\delta}(Q_{j,\be})$, and  satisfies $\rho^*(\alpha + d\beta)|_{B_{\delta}(Q_{j,\be})}=\mathrm{pr}^*(\alpha|_{L_{j,\be}})$. The existence of $\beta$ follows from the Poincaré's Lemma. 
\end{proof}
We now compute the first Chern class of the normal bundle of $\widetilde{X}\subset \widetilde{\cX}$.

\begin{lemma}\label{lem:pdualZ-2}
$c_1(\nu(\tilde{X}))= -\frac{1}{2} \sum_j \tau_j$.
\end{lemma}
\begin{proof}
First, $c_1(\nu(\widetilde{X}))|_{\widetilde{X}-Q}=0$ since $\widetilde{X}-Q \cong M/\iota-L$ and $\nu(\tilde{X}-Q)$ is isomorphic to $ \left( (M-L) \times \C \right) /(x,z) \sim (\iota(x),-z)$. This bundle is trivial, as it is homotopic to  $M/\iota -L \times \C$ via
 $$
 E_t=  \left( (M-L) \times \C \right)/(x,z) \sim (\iota(x),e^{\pi i t} z)), \quad t\in [0,1].
 $$
 From
 the short exact sequence of the pair $(\widetilde{X},\widetilde{X}-Q )$ it follows that $c_1(\nu(\widetilde{X})) \in H^2(\widetilde{X},\widetilde{X}-Q )$. By
 Lemma \ref{lem:relative-to-complement}, $c_1(\nu(\widetilde{X}))= \sum_j \lambda_j [\tau_j]$, for some $\lambda_j\in \R$.  It remains to show  $\lambda_j=-1/2$ for all $j$.

 Let $x\in L_j$, diagram  \eqref{diagram:resol-pont} shows that $\nu(\widetilde{X})|_{Q_j|_{x}}$ is isomorphic to the normal bundle of $Q_j|_x=\CP^1$ in $\cQ_{j,\bar{0}}|_{x}=\CP^2$. This is the restriction to $Q_x=\CP^1$ of the hyperplane bundle of $\cQ_{j,\bar{0}}|_{x}=\CP^2$, i.e, the dual of the tautological line bundle $\cS_{j,\bar{0}}|_{\cQ_{j,\bar{0}}|_{x}} \to \cQ_{j,\bar{0}}|_{x} $. By our previous discussion and equation \eqref{eqn: Euler-Thom}, 
 $
c_1(\cS_{j,\bar{0}}|_{\cQ_{j,\bar{0}}|_{x}})= \frac{1}{2}[\upsilon_{j,\bar{0}}]|_{\cQ_{j,\bar{0}}|_{x}}:
$
$$
\lambda_j [\tau_j]|_{Q_j|_x} = c_1(\nu(\tilde{X})|_{Q_j|_x})=-c_1(\cS_{j,\bar{0}}|_{\cQ_{j,\bar{0}}|_{x}})|_{Q_j|_x}=- \frac{1}{2}[\tau_j]|_{Q_{j}|_{x}}.
$$
The conclussion follows from $[\tau_j]|_{Q_j|_x} \neq 0$.
\end{proof}

We finally prove that $p_1(\widetilde{X})$ depends on the classes $p_1(M)$, $[(1-g_j)\mathrm{pr}_j^*\omega_j\wedge \tau_j]$, and $PD[L_j]$. Of course, $\iota^*(p_1(M))=p_1(M)$ because
 $\iota$ is a diffeomorphism. Hence, it induces a cohomology class on $X$.

\begin{proposition} \label{prop:pont}
Under the assumptions of Theorem \ref{theo:resol-exist},  we have:
$$
p_1(\widetilde{X})=\rho^*(p_1(M)) +  \sum_j {-3\rho^*(\mathrm{Th}[L_j]) + (4-4g_j)[\mathrm{pr}_j^*\omega_j\wedge \tau_j}].
$$
\end{proposition}
\begin{proof}
We first note that $p_1(\cY)=p_1(M)$.
We now claim that $p_1(\widetilde{\cX})|_{\widetilde{\cX}-\cQ}= \rho^*(p_1(M))|_{\widetilde{\cX}-\cQ}$. Since
$ 
\widetilde{\cX}-\cQ \cong \cX  - \cL
$,
 we have
\begin{equation}\label{eqn:pont-1}
p_1(\widetilde{\cX})|_{\widetilde{\cX}-\cQ }
=\rho^*(p_1(\cX- \cL)).
\end{equation}
Since the quotient map
  $
  q\colon \cY-  \cL \to \cX -  \cL $ is a cover, we have
  $
  q^*(T(\cX -  \cL ))=T(\cY-  \cL )
  $. Thus,
  \begin{equation}\label{eqn:pont-2}
p_1(\cX-  \cL)=q_*(p_1(\cY-  \cL))=q_*(p_1(\cY)|_{\cY-  \cL})=p_1(M)|_{\cY-  \cL},
\end{equation}
  and equations \eqref{eqn:pont-1}, \eqref{eqn:pont-2} imply the claim. The long exact sequence for the relative cohomology of the pair 
  $(\widetilde{\cX}, \widetilde{\cX}- \cQ)
  $ ensures that
  $
  p_1(\widetilde{\cX})- \rho^*(p_1(M)) \in H^{4}(\widetilde{\cX},\widetilde{\cX}- \cQ)
  $.  Lemma \ref{lem:relative-to-complement} implies 
  \begin{equation} \label{eqn:difference-rl}
  p_{1}(\widetilde{\cX})- \rho^*(p_1(M))=\sum_j [\mathrm{pr}_{\cQ}^*\b_{j,\be}\wedge \upsilon_{j,\be}], \qquad [\b_{j,\be}] \in H^{2}(\cQ_{j,\be}),
  \end{equation}
  where $\mathrm{pr}_{\cQ} \colon \sqcup_{j,\be} B_{\delta}(\cQ_{j,\be}) \to  \cQ$ denotes the nearest point projection. 
  
  We now focus on the restriction of $p_{1}(\widetilde{\cX}) - \rho^*(p_1(M))$ to ${\cQ_{j,\be}}$. First, from the short exact sequence \eqref{eqn:short-p1} we obtain $\rho^*(p_1(M))|_{\cQ_{j,\be}}=0$.
 In addition,
$
0 \to T\cQ_{j,\be} \to T\widetilde{\cX}|_{\cQ_{j,\be}} \to \nu(\cQ_{j,\be}) \to 0
$
yields $p_1(\widetilde{\cX})|_{\cQ_{j,\be}}=  p_1(\cQ_{j,\be})+ p_1(\nu(\cQ_{j,\be}))$.
Equations \eqref{eqn:formula-p-exceptional-3d}, \eqref{eqn: Euler-Thom}
imply 
$$
p_1(\cQ_{j,\be})=  \left[ \frac{3}{4}\upsilon_{j,\be}^2 - (2-2g_j)\mathrm{pr}^*\omega_{j,\be}\wedge \upsilon_{j,\be}\right]\bigg|_{\cQ_{j,\be}}.
$$
Since $c_1(\nu(\cQ_{j,\be}))=[\upsilon_{j,\be}]|_{\cQ_{j,\be}}$, we have $p_1(\nu(\cQ_{j,\be}))=[\upsilon_{j,\be}^2]|_{\cQ_{j,\be}}$.
Hence, $p_1(\widetilde{\cX})|_{\cQ_{j,\be}}= \frac{7}{4}[\upsilon_{j,\be}^2]|_{\cQ_{j,\be}}- (2-2g_j)[\mathrm{pr}^*\omega_{j,\be}\wedge \upsilon_{j,\be}]|_{\cQ_{j,\be}}$.
Proposition \ref{prop:cohomology-exceptional-divisor-v2} and equation \eqref{eqn: Euler-Thom} guarantee that
the restriction map $\mathrm{pr}_{\cQ}^*H^2(\cQ_{j,\be})\wedge [\upsilon_{j,\be}] \to H^4(\cQ_{j,\be})$ is injective. Thus,
\begin{equation} \label{eqn:p_1y}
p_1(\widetilde{\cX})- \rho^*(p_1(M))= \sum_{j,\be} {\frac{7}{4} [\upsilon_{j,\be}^2] - (2-2g_j)[\mathrm{pr}^*\omega_{j,\be}\wedge \upsilon_{j,\be}] }  .
\end{equation}

In addition,  from 
$
0 \to T\widetilde{X} \to T\widetilde{\cX}|_{\widetilde{X}} \to \nu(\widetilde{X}) \to 0,
$
we deduce $p_1(\widetilde{X})= p_1(\widetilde{\cX})|_X- p_1(\nu(\widetilde{X}))$. From
equations \eqref{eqn:restriction-tau-square} and \eqref{eqn:p_1y} we deduce
\begin{equation} \label{eqn:intermediate-eq-p1}
p_1(\widetilde{\cX})|_{\widetilde{X}}= \rho^*(p_1(M))+  \sum_j { -\frac{7}{2} \rho^*(\mathrm{Th}[L_j])+ (5-5g_j)\mathrm{pr}^*\omega_{j}\wedge \tau_{j} }. 
\end{equation}
Finally, from Lemma \ref{lem:pdualZ-2} we obtain
\begin{equation} \label{eqn:intermediate-eq-p1-2}
p_1(\nu(\widetilde{X}))=\frac{1}{4}\sum_j [\tau_j^2]= \sum_j{ -\frac{1}{2}\rho^*(\mathrm{Th}[L_j]) + (1-g_j)[\mathrm{pr}^*\omega_j\wedge \tau_j]},
\end{equation}
The statement follows subtracting \eqref{eqn:intermediate-eq-p1-2} from   \eqref{eqn:intermediate-eq-p1}.
\end{proof}
\begin{remark}
We cannot compute directly $p_1(\widetilde{X})-\rho^*(p_1(M))$ using the short exact sequence
$
0 \to TQ_j \to T\widetilde{X}|_{L_{j}\times \CP^1} \to \nu(Q_j).
$
A similar argument shows
 $$
  p_{1}(\widetilde{X})- \rho^*(p_1(M))=\sum_j [\mathrm{pr}_Q^*(\b_j)\wedge \tau_{j}], \qquad [\b_j] \in H^{2}(Q_j),
$$
  where $\mathrm{pr}_Q \colon \sqcup B_{\delta}(Q_j)\to Q$ denotes the nearest point projection. In the same way, $\rho^*(p_1(M))|_{Q_j}=0$. In addition,
  using the short exact sequence
  $0\to TQ_j \to T\widetilde{X}|_{Q_j} \to \nu(Q_j) \to 0$,
  equation \eqref{eqn:formula-p-exceptional-3d}, and that $c_1(\nu(Q_j))=\tau_j$ one obtains:
  $$
  p_1(\widetilde{X})|_{Q_j}=[\tau_j^2]|_{Q_j}= (4-4g_j)[\mathrm{pr}^*(\omega_j)\wedge \tau_j]|_{Q_j}.
  $$
  We cannot deduce $ p_{1}(\widetilde{X})- \rho^*(p_1(M))$ from this because the restriction map $\mathrm{pr}_{Q}^*(H^{2}(Q_j))\wedge[\tau_j] \to H^4(Q_j)$ is not injective. Increasing the codimension of the singular locus fixes that.
\end{remark}
 \begin{remark}
 \label{rmk-cohomology-c-3}
 If we instead construct the resolution using $\theta\in \Omega^1(L)$, closed and nowhere-vanishing, and use 
the notations of Remark \ref{rmk-cohomology-c-1}, we have  
 $$
p_1(\widetilde{X})=\rho^*(p_1(M)) +  \sum_j {-3\, \rho^*(\mathrm{Th}[L_j]) + 2\, c_1(\ker(\theta|_{L_j}))\otimes \mathbf{x}_j}.
$$
 \end{remark}

\section{Equivariant $\Z_2$-resolutions} \label{sec:equiv}

In this section we prove Theorem \ref{theo:2}. For this, we assume that $(M,\vp,g)$ is a compact manifold eqquiped with a closed $\Gtwo$ structure, and that $\iota \colon M \to M$  is an involution satisfying $\iota^*\vp=\vp$ and $L=\Fix(\iota) \neq \emptyset$. We consider the orbifold $X=M/\iota$.
We further assume the existence of a nowhere-vanishing closed $1$-form  $\theta \in \Omega^1(L)$, a finite order diffeomorphism $\kappa \colon M \to M$ and a continuous function $\sigma\colon L \to \{\pm 1\}$ such that
$$ \kappa^*(\varphi)= \varphi, \quad \iota \circ \kappa=\kappa \circ \iota, \quad  \kappa^*(\theta)=\sigma\theta.
$$

\subsection{$\Gtwo$ structure on the normal bundle}

Let $\pi \colon \nu(L) \to L$ be the normal bundle. Consider $R>0$ such that the neighborhood  $\nu_R(L)=\{ v_p \in \nu(L)\colon \, \|v_p\|<R\}$ is diffeomorphic to a neighborhood $\mathcal{U}$ of $L$ in $M$ via the exponential map; we identify these. Then, $\phi=\exp_L^* \varphi$ is a closed $\Gtwo$ form on
 $\nu_R(L)$, and it induces the metric $\exp_L^*g$, which we also call $g$. The maps $\iota$ and $\kappa$ induce vector bundle homomorphisms of $\nu(L)$ via their differentials, which we continue to denote by $\iota$ and $\kappa$.

Let $\nabla$ be a connection on $\nu(L)$ such that both $I$ and the hermitian metric determined by $I$ and $g$ are parallel with respect to $\nabla$. This condition will be needed to define the $\Gtwo$ structure in the resolution (see subsection \ref{subsec:resol}). The connection yields a splitting
$T\nu(L)= V \oplus H$ where $V=\ker(d\pi)$
and an isomorphism of vector bundles $\mathcal{T}\colon T\nu(L) \longmapsto \pi^*(TM|_L)$. More precisely, if $u_p \in T_pL$ and $w_p,v_p \in \nu_p(L)$, then $u_p+w_p \in \pi^{*}(TM|_L)_{v_p}$, and
\begin{equation} \label{eqn:T}
\mathcal{T}^{-1}(u_p+w_p)= (d\pi|_{H_{v_p}})^{-1}(u_p) + \frac{d}{dt}\bigg|_{t=0}{(v_p + tw_p)}.
\end{equation}
The splitting of $T\nu(L)$ induces a decomposition $T^*\nu(L)= V^* \oplus H^*$, where $V^*=\mathrm{Ann}(H)$ and $H^*=\mathrm{Ann}(V)=\pi^*(T^*L)$, which yields
\begin{equation} \label{eqn:splitting}
\Lambda^k T^*\nu(L) = \oplus_{p+q=k} W^k_{p,q}, \qquad W^k_{p,q}= \Lambda^p V^*\otimes \Lambda^q H^*.
\end{equation}
Given $\beta \in \Omega^k(\nu(L))$, we denote its projection to $W^k_{p,q}$ as $[\beta]_{p,q}$. The complex structure on $\nu(L)$
\begin{equation}
I(v_p)=\frac{1}{\|\theta_p\|}\theta^\sharp_p \times v_p,
\end{equation}
defined similarly to equation \eqref{eqn:acs-norml} satisfies $ I \circ \kappa=\sigma I\circ\kappa$ because
 $\kappa^*\varphi=\varphi$, and $\kappa^*\theta=\s \theta$.
We extend $I$ as a complex structure on $V$; if $v_p\in \nu_p(L)$ then
$$
I\left( \frac{d}{dt}\bigg|_{t=0}{(v_p + tw_p)}\right)= \frac{d}{dt}\bigg|_{t=0}{(v_p + tI(w_p))}.
$$
In particular, since $\kappa_*(V)=V$, we have
$\kappa_*|_V \circ I = \sigma I \circ \kappa_*|_V$.
Similarly, we extend $I$ to $V^*$ by pullback. That is, we set $I\alpha(X)=\alpha(IX)$ if $X\in V$.  The $\Gtwo$ structure on $\nu(L)$ given by $\phi_1=\mathcal{T}^*(\pi^*(\varphi|_{TM|_L}))$ determines the metric $ g_1= \mathcal{T}^*(\pi^*g|_{TM|_L})$. This makes the summands of the decomposition \eqref{eqn:splitting} orthogonal to each other. 

The radius function $r\colon \nu(L) \to [0,\infty)$, $r(v_p)=g(v_p,v_p)^{\frac{1}{2}}$, is $\iota$- and $\kappa$-invariant.  
We denote dilations of ratio $t$ by  $F_t \colon \nu(L) \to \nu(L)$, $F_t(v_p)=tv_p$.
We observe that $\iota=F_{-1}$, while $\kappa$ satisfies $F_t \circ \kappa=\kappa \circ F_t$ because $\kappa$ is a bundle homomorphism. 
The dilations allow us to find the Taylor expansions of $F_t^*\phi$ and $F_t^*g$ near $t=0$; that is,
\begin{equation}
 F_t^*(\p) \sim \sum_{k=0}^{\infty}{t^{2k}\p^{2k}}, \quad F_t^*g \sim \sum_{k=0}^\infty {t^{2k}g^{2k}},
\end{equation}
where the even terms vanish because $\iota=F_{-1}$. In \cite{LMM} we also prove that: 
\begin{enumerate}
\item  $\|[\phi^{2k}]_{p,q}\|_{g_1}=O(r^{2k-p})$ and $[\phi^{2k}]_{p,q}=0$ if $p>2k$. 
\item The forms $\phi^{2k}$ are closed, and if $k\geq1$ then $\phi^{2k}= \frac{1}{2k} d(i(\mathcal{R})\phi^{2k})$, where $\mathcal{R}(v_p)= \frac{d}{dt}\mid_{t=0}(e^tv_p)$.
\end{enumerate}
From $F_t \circ \kappa=\kappa \circ F_t$, we also deduce $\kappa^* \phi^{2k}= \phi^{2k}$ and $\kappa^*g^{2k}=g^{2k}$.
These properties imply the estimate $\phi-\phi^0-\phi^2=O(r)$, which allows us to interpolate $\phi$ with the closed form $\phi_2=\phi^0 + \phi^2$. The following result is the $\kappa$-equivariant version of \cite[Proposition 3.5]{LMM}, and it can be proved similarly.

\begin{proposition}\label{prop:inter-nu}
There exists $0<\e_0 < \frac{R}{4} $ such that for each $\e < \e_0$ there is a closed $\Gtwo$ form $\phi_{3,\e}\in \Omega^3(M)$, invariant under both $\iota$ and $\kappa$, satisfying $\phi_{3,\e}= \phi_2$ if $r \leq 2\e$ and $\phi_{3,\e}= \phi$ if $r \geq 4\e$. 
\end{proposition}

\subsection{Average of forms}

While $\kappa(V)=V$, there is no reason to expect that $\kappa(H)=H$ in general. However, this is true along the zero section $Z\subset \nu(L)$, as $H_{0_p}=T_{0_p}Z$ and $\kappa$ preserves $Z$.
In addition, the forms introduced in \cite{LMM} are not $\kappa$-invariant and we need to average them. We therefore introduce the following notation:
$$
\mu(\alpha)= \frac{1}{N}\sum_{\ell=1}^N{(\kappa^\ell)^*\alpha}, \qquad  \alpha \in \Omega^*(\nu(L)),
$$
where $N$ is the order of $\kappa$. Since $\kappa$ and $\iota$ commute, the average of an $\iota$-invariant form is both $\iota$- and $\kappa$-invariant. The following Lemmas allow us to estimate the errors introduced when we average forms. 

\begin{lemma} \label{lem:p-vert}
Let $0\leq r \leq k$. If $\beta \in \oplus_{p=0}^r{W_{p,k-p}}$,
then $(\kappa^\ell)^*\beta \in \oplus_{p=0}^r{W_{p,k-p}}$ and
$
[(\kappa^\ell)^* \b]_{r,k-r} = [(\kappa^\ell)^* [\b]_{r,k-r}]_{r,k-r}.
$
Hence, $\mu(\beta) \in \oplus_{p=0}^r{W_{p,k-p}}$ and
$
[\mu(\beta)]_{r,k-r} = [\mu([\b]_{r,k-r})]_{r,k-r}
$.
\end{lemma}
\begin{proof}
We prove the first statement for $r=0$ and $k=1$. If $\gamma\in W_{0,1}$, then 
$
\gamma= \sum_j f_j\pi^*\alpha_j
$,
for some $f_j \in C^\infty(\nu(L))$, and $\alpha_j \in \Omega^1(L).
$
Since
$\kappa^* \pi^* \alpha_j= \pi^* \kappa^* \alpha_j$ we have
$$
 (\kappa^\ell)^* (\gamma)= \sum_j{ (f_j \circ \kappa^\ell) \, \pi^*((\kappa^\ell)^* \alpha_j)}\in H^*.
$$
Of course, $[(\kappa^\ell)^*\gamma]_{0,1}=(\kappa^\ell)^*\gamma= [(\kappa^\ell)^*[\gamma]_{0,1}]_{0,1}$. The case $r=k=1$ follows from linearity. That is, if $\beta \in \Omega^1(\nu(L))$, then
$[(\kappa^\ell)^*\beta]_{1,0}= [(\kappa^\ell)^*[\b]_{1,0}]_{1,0} + [(\kappa^\ell)^*[\b]_{0,1}]_{1,0}= [(\kappa^\ell)^*[\b]_{1,0}]_{1,0}$. For the general case, observe that if $\beta_1,\dots, \b_p \in V^*$ and $\g_1, \dots, \g_q \in H^*$, the previous arguments ensure:
\begin{align*}
[(\kappa^\ell)^*(\beta_1 \wedge \dots  \wedge  \b_p \wedge \g_1  \wedge  \dots  \wedge  \g_q)]_{p,q}=&\, [(\kappa^\ell)^*\b_1]_{1,0}\wedge \dots \wedge [(\kappa^\ell)^* \b_p]_{1,0} \wedge (\kappa^\ell)^* \g_1 \wedge \dots \wedge  (\kappa^\ell)^* \g_q,\\
[(\kappa^\ell)^*(\beta_1 \wedge \dots \b_p \wedge \g_1\wedge \dots \wedge \g_q)]_{s,p+q-s}=&\,0, \mbox{  if } s>p.
\end{align*}
This implies the general case, and  the second statement   follows from this by averaging.
\end{proof}

\begin{lemma}\label{lem:1-forms}
Define $e_1=\|\theta \|^{-1}\theta$ and let $\alpha \in \Omega^1(\nu(L))$, then 
\begin{enumerate}
\item $[d\alpha]_{2,0}= [d[\alpha]_{1,0}]_{2,0}$.
\item If $\alpha \in V^*$ then $[(\kappa^\ell)^* I\alpha]_{1,0}= \s^\ell I [(\kappa^\ell)^* \alpha]_{1,0} $.%
\item Let $f\colon \nu(L) \to \R$ be a $\kappa$-invariant function on $\nu(L)$, then 
 $[\mu(\pi^*e_1 \wedge I[df]_{1,0})]_{1,1}=\pi^*e_1 \wedge I[df]_{1,0} $,
 and 
$[\mu(\pi^*e_1 \wedge dI[df]_{1,0})]_{2,1}=\pi^*e_1 \wedge [dI[df]_{1,0}]_{2,0}$.
\end{enumerate}
\end{lemma}
\begin{proof}
If $\alpha \in \Omega^1(\nu(L))$, then
 $[\alpha]_{0,1}= \sum_jf_j\pi^*\alpha_j$ with $\alpha_j\in \Omega^1(L)$ and $f_j \in C^{\infty}(\nu(L))$. Hence,
$$
d[\alpha]_{0,1}= \sum_j [df_j]_{1,0} \wedge \pi^*\alpha_j + \sum_j [df_j]_{0,1} \wedge \pi^*\alpha_j + f_j\pi^*(d\alpha_j) \in W_{1,1}\oplus W_{0,2}.
$$
So $[d[\alpha]_{0,1}]_{2,0}=0$. By linearity,
$[d\alpha]_{2,0}= [d[\alpha]_{1,0}]_{2,0} + [d[\alpha]_{0,1}]_{2,0} = [d[\alpha]_{1,0}]_{2,0}$.

Assume that $\alpha\in V^*$, and let $X \in V$. The equality $\kappa_*|_V \circ I = \sigma I \circ \kappa_*|_V$ yields:
\begin{equation*}
[(\kappa^\ell)^* I \a]_{1,0}(X)=\a((I(\kappa^\ell_*X))= \s^\ell \a (\kappa^\ell_*(IX))
= \s^\ell [(\kappa^\ell)^*\a]_{1,0}(IX)= \s^\ell I[(\kappa^\ell)^*\a]_{1,0} (X).
\end{equation*}
therefore, $[(\kappa^\ell)^* I\alpha]_{1,0}= \s^\ell I [(\kappa^\ell)^* \alpha]_{1,0} $.

Let $f$ be a $\kappa$-invariant function. Using the second statement and Lemma  \ref{lem:p-vert},  we obtain 
\begin{equation} \label{eqn:kIdf}
[(\kappa^\ell)^* (I [df]_{1,0})]_{1,0}= \sigma^\ell I[(\kappa^\ell)^* [df]_{1,0}]_{1,0}= I [df]_{1,0}.
\end{equation}
From $\kappa^*e_1= \sigma e_1$ we deduce
$
[(\kappa^\ell)^*(\pi^*e_1 \wedge I[df]_{1,0})]_{1,1}=
\s^\ell \pi^*e_1 \wedge [(\kappa^\ell)^* I [df]_{1,0}]_{1,0}=
\pi^*e_1 \wedge I [df]_{1,0}.
$
This ensures  $[\mu(\pi^*e_1 \wedge I[df]_{1,0})]_{1,1}=\pi^*e_1 \wedge I[df]_{1,0} $. We finally define $\b= \pi^*e_1 \wedge dI[df]_{1,0} \in W_{2,1}\oplus W_{1,2}\oplus W_{0,3}$. The first statement and equation \eqref{eqn:kIdf} show:
\begin{align*}
[(\kappa^\ell)^*(dI[df]_{1,0})]_{2,0}=&
[d((\kappa^\ell)^* I[df]_{1,0})]_{2,0}=
[d[(\kappa^\ell)^* I[df]_{1,0}]_{1,0}]_{2,0}=
\s^\ell [dI[df]_{1,0}]_{2,0}.
\end{align*}
Therefore,
$
[(\kappa^\ell)^*\b]_{2,1}
= \s^\ell \pi^*e_1 \wedge [(\kappa^\ell)^*(dI[df]_{1,0})]_{2,0} 
= \pi^*e_1 \wedge [(dI[df]_{1,0})]_{2,0} 
=[\b]_{2,1}.
$
The stated equality follows from this.
\end{proof}

\subsection{Local formulas for a family of $\Gtwo$ structures}
\label{subsec:loc}

We collect in a Proposition some of the local formulas we found on \cite[Section 3.3.1]{LMM}

\begin{proposition} \label{lem:homog}
The following equalities hold:
\begin{enumerate}
\item $\p_1 = \p^0 + [\p^2]_{2,1}$
\item $ g_1 = g_{0,2}^0 + g_{2,0}^2$
\end{enumerate}
In addition, let $e_1=\|\h\|^{-1} \h$ and consider an orthonormal oriented frame $(e_1,e_2,e_3)$ of $T^*L$ on a neighbourhood $U\subset L$. Then
$$
\p^0= \pi^*(e_1\wedge e_2 \wedge e_3), \qquad [\p^2]_{2,1}= \pi^*e_1\wedge \o_1 + \pi^* e_2 \wedge \o_2 - \pi^* e_3 \wedge \o_3,
$$
where $\o_1,\o_2,\o_3 \in W_{2,0}$ satisfy
$$
\o_1= -\frac{1}{4}[d[Idr^2]_{1,0}]_{2,0}, \qquad \o_2 + i\o_3 = [d\vartheta]_{2,0},
$$
and $\vartheta\in V^*\otimes \C $
 is a $\iota$-invariant $1$-form such that $d_{\nu_p(L)}(\vartheta|_{\n_p(L)})= (\o_2 + i \o_3)|_{\n_p(L)}$ and $\vartheta|_{T\n(L)|_{Z}}=0$.
\end{proposition}
\begin{remark}
  In \cite{LMM}, $\vartheta$ is denoted by $\mu$.
\end{remark}

Let $t>0$, the form $F_t^*\phi_1= \p^0 + t^2[\p^2]_{2,1}$ is a $\Gtwo$ structure on $\nu(L)$ with associated metric $g_t= F_t^*(g_1)= g_{0,2}^0 + t^2 g_{2,0}^2$. Since $F_t^*\phi_1$ might not be closed,
in \cite[Section 3.3.2]{LMM} we defined a family of closed $\iota$-invariant forms $\widehat{\p}_2^t$ on $\nu(L)$, close to $F_t^*\phi_1$, whose local formula is manageable. This is,
\begin{align*}
\widehat{\p}_2^t =& \pi^*(e_1\wedge e_2 \wedge e_3) + t^2\b, \qquad\qquad
\b= \b_1 + \b_2, \\
\b_1 =& -\frac{1}{4} \pi^*\h \wedge d( (\|\h\|^{-1}\circ \pi) I[dr^2]_{1,0}), \quad \b_2= d(\pi^*e_2\wedge \vartheta_2 - \pi^*e_3\wedge \vartheta_3 ).
\end{align*}
Note that $\b, \widehat{\p}^t_2 \in W_{2,1}\oplus W_{1,2}\oplus W_{0,3} $. From the definitions we deduce,
\begin{equation} \label{eqn:2,1part}
 [\b]_{2,1}= [\phi_1]_{2,1}= [\phi^2]_{2,1}.
\end{equation}
Let $s>0$. Using equation \eqref{eqn:2,1part} and the equality
$\|\gamma\|_{g_t}= \frac{1}{t^3}\| [\gamma]_{3,0}  \|_{g_1} +\frac{1}{t^2}\| [\gamma]_{2,1} \|_{g_1}+ \frac{1}{t} \| [\gamma]_{1,2} \|_{g_1} + \| [\gamma]_{0,3} \|_{g_1}$
for $\gamma \in \Omega^3(\nu(L))$, we deduce that $\|\widehat{\phi}_2^t -F_t^*\phi_1\|_{g_t}=O(t)$ on $\nu_{2s}(L)$. From Lemma \ref{lem:universal-m} we deduce that   $\widehat{\phi}_2^t$ determines a $\Gtwo$ structure on $\nu_{2s}(L)$ for small values of $t>0$. For our purposes, we need to show that $\mu(\widehat{\p}^t_2)$ also does so, and then interpolate it with the $\iota$- and $\kappa$-invariant form
 $$
 \phi_2^t=F_t^*(\phi_2)=\phi^0 + t^2 \phi^2.
 $$
The following Lemma generalizes equation \eqref{eqn:2,1part} and will help us to achieve this in Proposition \ref{prop:interpolation-prep}.
 
\begin{lemma}\label{lem:2,1nu}
The following equality holds:
\begin{equation}\label{eqn:2,1part-invt}
 [\mu(\b)]_{2,1}=[\b]_{2,1}= [\phi_1]_{2,1}= [\phi^2]_{2,1}.
\end{equation}
In addition, $[\mu(\b_k)]_{2,1}=[\b_k]_{2,1}$ for $k=1,2$.
\end{lemma}
\begin{proof}
Equation \eqref{eqn:2,1part-invt} is a consequence of Lemma \ref{lem:p-vert}, equation \eqref{eqn:2,1part}, $\phi^2_{3,0}=0$, and $\kappa^*\phi^2=\phi^2$:
\begin{equation*}
[\m(\beta)]_{2,1}=
[\m([\beta]_{2,1})]_{2,1}=
[\m([\phi^2]_{2,1})]_{2,1}=
[\phi^2]_{2,1}=[\beta]_{2,1}.
\end{equation*}
We now prove  $[\mu(\b_1)]_{2,1}=[\b_1]_{2,1}$. 
 Since $d(\|\theta\|\circ \pi) \in W_{0,1}$ we have
$[\b_1]_{2,1}= -\frac{1}{4}\pi^*e_1 \wedge [d I[dr^2]_{1,0}]_{2,0}= -\frac{1}{4}[\pi^*e_1 \wedge d I[dr^2]_{1,0}]_{2,1}$, and $\pi^*e_1 \wedge d I[dr^2]_{1,0}\in W_{2,1}\oplus W_{1,2} \oplus W_{0,3}$. Lemmas \ref{lem:p-vert} and \ref{lem:1-forms} allow us to obtain:
\begin{align*}
[\m(\beta_1)]_{2,1}=&
[\m([\beta_1]_{2,1})]_{2,1}=
-\frac{1}{4}[\m([\pi^*e_1 \wedge d I[dr^2]_{1,0}]_{2,1})]_{2,1}
\\
=&
-\frac{1}{4} [\m(\pi^*e_1 \wedge d I[dr^2]_{1,0})]_{2,1}
= -\frac{1}{4}
\pi^*e_1 \wedge [d I[dr^2]_{1,0}]_{2,0}
= [\beta_1]_{2,1}.
\end{align*}
Finally, $[\m(\b_2)]_{2,1}=[\m(\b - \b_1)]_{2,1}= [\b - \b_1]_{2,1}= [\b_2]_{2,1}$.
\end{proof}

\begin{proposition}  \label{prop:interpolation-prep}
Let $s>0$, there exists $t_s>0$ such that for each $t<t_s$, there is a closed  $\Gtwo$  form $\widehat \p_{3,s}^t$ on $\n_{2s}(L)$ which is invariant under $\iota$ and $\kappa$, and satisfies
$\widehat \p_{3,s}^t|_{r\leq \frac{s}{4}}=\m(\widehat \p_2^t)$ and $\widehat \p_{3,s}^t|_{r\geq \frac{s}{2}}=\p_2^t$.
\end{proposition}

\begin{proof}
In \cite[Proposition 3.7]{LMM}, we proved that there is a $\iota$-invariant form
$\xi \in W_{0,2}$ such that $\|\xi\|_{g_1}=O(r^2)$, and $\p^2 = \b +d\xi$. In particular, equation \eqref{eqn:2,1part} implies $d\xi \in W_{1,2}\oplus W_{0,3}$.
Let $\varpi$ be a smooth function such that $\varpi=0$ if $x\leq \frac{1}{4}$ and $\varpi=1$ if $x\geq \frac{1}{2}$, and let $\varpi_s(x)=\varpi(\frac{x}{s})$. The form
$$
\widehat \p_{3,s}^t = \p^0 + t^2 \m(\b) +  t^2 d(\varpi_s(r) \m(\xi))
$$
is  closed, $\iota$- and $\kappa$-invariant, and it coincides with $\mu(\widehat \p_2^t)$ on $r\leq \frac{s}{4}$ and with $\mu(\p_2^t)=\p_2^t$ on $r\geq \frac{s}{2}$.  We now check that this is a $\Gtwo$ form on $\nu_{2s}(L)$ when $t>0$ is small enough. For this, we compare $\widehat{\p}{}^t_{3,s}$ with $F_t^*\p_1$ and use Lemma \ref{lem:universal-m} to conclude. Using Lemma \ref{lem:p-vert} we obtain $d(\varpi_s \mu(\xi)) \in  W_{1,2}\oplus W_{0,3}$. 
Equality \eqref{eqn:2,1part-invt} and Proposition \ref{lem:homog} imply:
\begin{align*}
\| \widehat \p_{3,s}^t - F_t^*\p_1\|_{g_t} =& \, 
t \| [\m(\beta)+ d(\varpi_s(r) \m(\xi))]_{1,2}\|_{g_1} + t^2 \| [\m(\beta) + d(\varpi_s(r) \m(\xi))]_{0,3} \|_{g_1}.
\end{align*}
Lemma \ref{lem:universal-m} implies that $\widetilde{\phi}^t_{3,s}$ is a
$\Gtwo$ form on $\nu_{2s}(L)$ when $t<t_s$, and $t_s<1$ satisfies 
$$
 \mathrm{max}_{r\leq 2s}(\| [\mu(\b)+ d(\varpi_s(r) \m(\xi))]_{1,2}\|_{g_1} + \| [ \mu(\b)+  d(\varpi_s(r) \m(\xi))]_{0,3} \|_{g_1} ) <  \frac{m}{t_s}.
$$
\end{proof}

\subsection{The resolution of $\nu(L)/\iota$}\label{subsec:resol}

We now consider the bundle $P=P_{\UU(2)}(\nu(L))\times_{\UU(2)} \widetilde{\C^2}/\Z_2$ and the projection $\chi \colon P \to \nu(L)/\iota$. The exceptional divisor is $Q=\chi^{-1}(L)$.
We denote elements in $(P_{\UU(2)}(\nu(L)))_p$ as complex maps $F_p \colon \nu_p(L) \to \CC^2$. Explicitly, $\chi[F_p, [v,\ell]]=[(F_p)^{-1}(v)]_{\iota}$.
It is also convenient to denote $\mathrm{pr}=\pi \circ \chi$.

\begin{proposition} \label{prop:extension}
There is a diffeomorphism $\tilde\kappa \colon P \to P $ such that $\kappa \circ \chi= \chi \circ \tilde \kappa$. 
\end{proposition}
\begin{proof}
 Given $L_j$, we set $L_k=\kappa(L_j)$ (possibly $j=k$), and we define the extension $\tilde \kappa \colon P|_{L_j} \to P|_{L_k}$. 
We first assume $\s|_{L_j}=1$, case in which $\kappa \colon \nu(L_j) \to \nu(L_k)$ is a complex homomorphism. There is a induced equivariant morphism $\kappa \colon P_{\UU(2)}(\nu(L_j))\to P_{\UU(2)}(\nu(L_k))$ covering $\kappa \colon L_j \to L_k$, namely
$$
\kappa(F_p)=F_p \circ (\kappa|_p)^{-1} \in (P_{\UU(2)}(\nu(L_k)))_{\kappa(p)},
$$
where $\kappa|_p\colon \nu_p(L_j) \to \nu_{\kappa(p)}(L_k)$ is the restriction of $\kappa$ to $\nu_p(L)$.
Since $P=P_{\UU(2)}(\nu(L))\times_{U(2)} \widetilde{\C^2}/\Z_2$, this induces 
$\tilde\kappa[F_p,[v,\ell]]=[\kappa(F_p),[v,\ell]]$, which is the claimed extension as:
$$
(\chi \circ \tilde \kappa) [F_p,[v,\ell]]=
 [(F_p \circ (\kappa|_p)^{-1})^{-1}(v)]_{\iota}=
  [\kappa_p (F_p^{-1}(v))]_{\iota}= 
  (\kappa \circ \chi)[F_p, [v,\ell]].
$$

Suppose $\sigma|_{L_j}=-1$. Let $\overline{\nu(L_j)} \to L$ be the conjugate bundle of $\nu(L_j)$; then
$\kappa \colon \overline{\nu(L_j)} \to \nu(L_k)$ is a complex homomorphism. We denote elements in $(P_{\UU(2)}(\overline{\nu(L_j)}))_p$ as complex maps $G_p\colon \overline{\nu(L)}_p \to \CC^2$.
Similarly, we define a map $\tilde{\kappa}'\colon P_{\UU(2)}(\overline{\nu(L_j)})\times_{\UU(2)} \widetilde{\C^2}/\Z_2 \to P|_{L_k}$, $\tilde\kappa'[G_p,[v,\ell]]=[G_p\circ (\kappa|_p)^{-1},[v,\ell]]$.

Observe that the diffeomorphism
$c \colon P|_{L_j}  \to P_{\UU(2)}(\overline{\nu(L_j)})$,  $c(F_p)(w_p)=\overline{F_p(w_p)}$
satisfies $c(A \cdot F_p)= \overline{A} \cdot c(F_p)$, $A\in \UU(2)$.
In addition, since conjugation maps complex lines to complex lines, it induces a map $\bar{\cdot} \colon \widetilde{\C^2}/\Z_2 \to \widetilde{\C^2}/\Z_2$. The diffeomorphism,
$$
c \colon P|_{L_j}  \to  P_{\UU(2)} (\overline{\nu(L_j)})\times_{\UU(2)}  \widetilde{\C^2}/\Z_2,
 \quad [F_p,[v,\ell]] \longmapsto [c(F_p),\overline{[v,\ell]}],
 $$
 is well-defined because 
 $
 c[ A \cdot F_p, A^{-1}[v,\ell]]=[\overline{A} \cdot F_p, \overline{A}^{-1}\cdot \overline{[v,\ell]}]
 $ when $A\in \UU(2)$.
The extension is $\tilde{\kappa}= \tilde{\kappa}' \circ c $
because, using $c(F_p)^{-1}(\bar{v})=F_p^{-1}(v)$ (which follows from
$c(F_p)(F_p^{-1}(v))=\bar{v}$), we obtain:
$$
(\chi \circ \tilde \kappa) [F_p,[v,\ell]]=
 [(c(F_p) \circ (\kappa|_{p})^{-1})^{-1}(\bar{v})]_{\iota}= [\kappa|_p \circ c(F_p)^{-1}(\bar{v})]_{\iota}=  
 [\kappa_p (F_p^{-1}(v))]_{\iota}=  
 (\kappa \circ \chi)[F_p, [v,\ell]].
$$
\end{proof}

As discussed in \cite{LMM}, $\nabla$ lifts to the bundle $P \to L$ because  $\nabla I=0$. The splitting $TP=V'\oplus H'$ determined by $\nabla$ is compatible with $T\nu(L)=V\oplus H$ on $P-Q$. Namely, if $p\in P-Q$ and $\chi(p)=[x]_{\iota}$, then $d\chi_p (H'_p)$ is the image of $H_{x}$ by the projection $T\nu(L)-Z \to (T\nu(L) - Z)/\iota$.
 This induces a decomposition $\L^k T^*P= \oplus_{i+j=k} \L^i V' \otimes \L^j H'$. We also denote by $[\a]_{i,j}$ to the projection of $\a$ to $\L^i V' \otimes \L^j H'$. Due to the compatibility condition on $P-Q$ we have:
$$
\chi^*([\alpha]_{p,q})=[\chi^{*}(\alpha)]_{p,q}.
$$
Note that we view forms on $(\nu(L)-Z)/\iota$ as $\iota$-invariant forms on $\nu(L)-Z$. In addition, there is a well-defined complex structure $I$ on $(V')^*$. 
The radius function lifts to $P$; its fourth power $r^4$ is smooth. Finally, Proposition \ref{prop:extension} allows us to lift $\mu$ to $P$ as
$$
\tilde\mu(\alpha)= \frac{1}{N}\sum_{\ell=1}^N{(\tilde{\kappa}^\ell)^*\alpha},
$$
as the order of $\tilde{\kappa}$ is also $N$.

We collect in a Proposition the formulas for the construction of a $\Gtwo$ structure on $P$.

\begin{proposition}\cite[Section 3.4]{LMM}
In the set-up of Proposition \ref{lem:homog}, we denote $|\h|=\| \theta \| \circ \mathrm{pr}$. Then,
\begin{enumerate}
\item Let $\ssf_a(x)=  \sg_a(x) + 2a\log(x)$, where $ \sg_a(x)=(x^4 + a^2)^{1/2}- a\log((x^4 + a^2)^{1/2}+a)$.
Then 
$$
\widehat \o_1 =  -\frac{1}{4}d (|\h|^{-1}  I[d\ssf_{|\h|}(r)]_{1,0}), 
$$ is smooth on $P$.
\item The forms $\chi^*(\vartheta_k),\chi^*(\omega_k)$, $k=2,3$ are smooth on $P$.
\item The form $\Phi_1= \mathrm{pr}^*(e_1 \wedge e_2 \wedge e_3) + \mathrm{pr}^* \theta \wedge [\widehat \o_1]_{2,0} + \mathrm{pr}^*e_2 \wedge \chi^*\o_2 + \mathrm{pr}^* e_3 \wedge \chi^* \o_3 $ determines a $\Gtwo$ structure on $P$.
\item The metric $h$ associated to $\Phi_1$ satisfies $h_{1,1}=0$. In addition, $\chi_*(h_{0,2})=g_{0,2}^0$ and $\|\chi_*(h_{2,0})-g_{2,0}^2\|_{g_{0,2}^2}= O(r^{-2})$ on $r>0$.
\end{enumerate}
\end{proposition}
\begin{remark} \label{rmk:curvature}
We verify that $dI[d\log(r^2)]_{1,0}$ is related to the curvature $R$ of the connection on the tautological bundle ${\pi}_1\colon S\to Q$, determined by projecting to $S$ the pullback connection of $\pi^*(\nu(L)) \to Q$. More precisely, $dI[d\log(r^2)]_{1,0}=2i\pi^*(R)$. Given a local unitary frame $(v_1,v_2)$ of $\nu(L)$ we write $\nabla_X v_j = A_{j}^1(X)v_1 + A_j^2(X)v_2$, with $A_j^j(X) \in i\R$ and $\bar{A}_{1}^2=-A_2^1$. In particular, $V^*$ is spanned over $\R$ by the real and imaginary parts of
$\eta_j = dz_j +  z_1 A_1^j
+ z_2 A_2^j $,  $j=1,2$. Using that $I\eta_j=  i \eta_j$ and  $I\bar{\eta}_j= -i \bar{\eta}_j$ we find,  
$$
I[d \log (r^2)]_{1,0} = \frac{i}{|z_1|^2+ |z_2|^2}\left({\bar{z}_1 \eta_1 - z_1 \bar{\eta}_1 + \bar{z}_2 \eta_2 - z_2 \bar{\eta}_2}\right).
$$
A trivialization of $S \to Q$ on $U_Q = \{ \ell_x, \, x\in U, \, v_1(x) \not \perp \ell_x\} \subset Q$ is given by $(x,[1:z],\lambda)  \mapsto ([v_1(x) + z v_2(x)], \lambda(v_1(x) + z v_2(x))$. Its connection $1$-form is 
$$
\alpha= \frac{1}{1+|z|^2} \left( \pi^*(A_1^1) + z\pi^*(A_2^1)  + \bar{z} (dz + \pi^*(A_1^2)+ z\pi^*(A_2^2) \right).
$$
The curvature $R=d\alpha$ takes values in $i\R$, as both $\alpha - \frac{\bar{z}dz}{1+|z|^2}$ and  $d\left(\frac{\bar{z}dz}{1+|z|^2}\right)$ do. Denote $\pi_2 \colon S \to \nu(L)$; on $S-Q$ we have $\pi_2^*(dI[d \log (r^2)]_{1,0})=2 \pi_1^*d\alpha$ because:
$$
\pi_2^*(I[d \log (r^2)]_{1,0}) = \frac{i d\lambda}{\lambda} -  \frac{id{\bar{\lambda}}}{\bar{\lambda}}- \frac{2}{1+|z|^2}
\mathrm{Im}\left( \pi^*(A_1^1) + z\pi^*(A_2^1)+ \bar{z} (dz + \pi^*(A_1^2)+ z\pi^*(A_2^2) \right) .
$$
Hence, $\pi_2^*(dI[d \log (r^2)]_{1,0})$ extends as $2i\pi_1^*(R)$, and the same holds for $dI[d \log (r^2)]_{1,0}$ on $P=S/\Z_2$.
\end{remark}

For convenience, consider the closed form $\widehat{\beta}= \mathrm{pr}^*\theta \wedge \widehat\omega_1 + \chi^*(\beta_2)$. With this notation, $\Phi_1= \mathrm{pr}^*(e_1 \wedge e_2 \wedge e_3) + [\widehat{\beta}]_{2,1}$. 
In \cite[Section 3.4]{LMM}, we considered the parametric family of $\Gtwo$ structures on $P$ given by
 $F_t^*\Phi_1= \mathrm{pr}^*(e_1 \wedge e_2 \wedge e_3) + t^2[\widehat{\beta}]_{2,1}$; its associated metric is $h_t= F_t^*(h_1)= h_{0,2}^0 + t^2 h_{2,0}^2$. Hence, $\|\gamma\|_{h_t}= \frac{1}{t^3}\| [\gamma]_{3,0}  \|_{h_1} + \frac{1}{t^2}\| [\gamma]_{2,1} \|_{h_1}+ \frac{1}{t} \| [\gamma]_{1,2} \|_{h_1} +  \|[\gamma]_{0,3} \|_{h_1}$ if $\gamma \in \Omega^3(P)$. In loc. cit., we also obtained 
a family of closed $\iota$-invariant forms: 
\begin{align*}
\widehat{\P}_2^t =& \mathrm{pr}^*(e_1\wedge e_2 \wedge e_3) + t^2\widehat\b.
\end{align*}
Note that $\widehat{\b}, \widehat{\P}^t_2 \in W_{2,1}\oplus W_{1,2}\oplus W_{0,3} $, and we have $[\widehat{\b}]_{2,1}=[\Phi_1]_{2,1}$. In Proposition  \ref{prop:inter-resol}, we interpolate $\tilde\mu(\widehat{\P}^t_2)$ with $\chi^*(\widehat \phi^t_{3,s})$ on $\chi^{-1}(\nu_{2s}(L)/\iota)$. To show that this is a $\Gtwo$ form, we need the following result, which is analogous to Lemma \ref{lem:2,1nu}.

\begin{lemma}
The following equality holds:
\begin{equation}\label{eqn:2,1resol}
 [\tilde\mu(\widehat{\b})]_{2,1}=[\widehat{\b}]_{2,1}.
\end{equation}
\end{lemma}
\begin{proof}
We prove  $[\tilde\mu(\mathrm{pr}^*\theta \wedge \widehat \o_1)]_{2,1}=[\mathrm{pr}^*\theta\wedge \widehat \o_1]_{2,1}$ and $[\tilde\mu(\chi^* \b_2)]_{2,1}=[\chi^* \b_2]_{2,1}$.
By continuity, it suffices to prove the equality on $P-Q$. In addition, since $\tilde \kappa$ lifts $\kappa$ and the splittings $TP=V'\oplus H'$ and $T\nu(L)=V\oplus H$ are compatible, it suffices to check a similar equality for the pullback by $\chi^{-1}$ of the forms above on $\nu(L)-Z$, replacing $\tilde \m$ with $\m$.
To check the second identity we note that $(\chi^{-1})^* (\chi^*\b_2)=\b_2$;  Lemma \ref{lem:2,1nu} ensures $[\mu(\b_2)]_{2,1}=[\b_2]_{2,1}$. To prove the first, we observe
 $$
 (\chi^{-1})^*(\mathrm{pr}^*\theta \wedge \widehat \o_1)=-\frac{1}{4} \pi^*\theta \wedge d((\|\theta\|^{-1}\circ \pi)I[d\ssf_{\|\theta\|\circ \pi}]_{1,0}).
 $$
Since
 $d(\|\theta\|\circ \pi) \in W_{0,1}$ we have
$[(\chi^{-1})^*(\mathrm{pr}^*\theta \wedge \widehat \o_1)]_{2,1}= -\frac{1}{4}\pi^*e_1 \wedge [d I[d\ssf_{\|\theta\|\circ \pi}]_{1,0}]_{2,0}= -\frac{1}{4}[\pi^*e_1 \wedge d I[d\ssf_{\|\theta\|\circ \pi}]_{1,0})]_{2,1}$. The function $\ssf_{\|\theta\|\circ \pi}(r)$ is $\kappa$-invariant because both the functions $\|\theta\|$ and $r$ are. Taking Lemmas \ref{lem:p-vert} and \ref{lem:1-forms} into account we get:
\begin{align*}
[\m((\chi^{-1})^*(\mathrm{pr}^*\theta \wedge \widehat \o_1))]_{2,1}
=&
[\m([(\chi^{-1})^*(\mathrm{pr}^*\theta \wedge \widehat \o_1)]_{2,1})]_{2,1}
=
-\frac{1}{4}[\m([\pi^*e_1 \wedge d I[d\ssf_{\|\theta\|\circ \pi}]_{1,0})]_{2,1})]_{2,1}
\\
=&
-\frac{1}{4} [\m(\pi^*e_1 \wedge d I[d\ssf_{\|\theta\|\circ \pi}]_{1,0})]_{2,1}
=
 -\frac{1}{4}
\pi^*e_1 \wedge [d I[d\ssf_{\|\theta\|\circ \pi}]_{1,0}]_{2,0}
\\
=& 
[(\chi^{-1})^*(\mathrm{pr}^*\theta \wedge \widehat \o_1)]_{2,1}.
\end{align*}
\end{proof}

\begin{proposition} \label{prop:inter-resol}
There is $s_0>1$, such that for every $s>s_0$ there exists $t_s'$ with the property that for all $t<t_s'$ there is a $\tilde\kappa$-invariant closed $\Gtwo$ structure $\P_{3,s}^t$ on $\chi^{-1}(\nu_{2s}(L))$ satisfying $\P_{3,s}^t= \tilde\m(\P_2^t)$ on $r \leq \frac{s}{8}$ and $\P_{3,s}^{t}= \chi^*(\widehat \p_{3,s}^t)$ on $r \geq \frac{s}{4}$. 
\end{proposition}
\begin{proof}
On the annulus $\frac{s}{8} < r < \frac{s}{4}$ we have:
$$
\tilde\m(\P_2^t) - \chi^*(\widehat \p_{3,s}^t)= \frac{1}{4} t^2 \tilde\m(d(\mathrm{pr}^*e_1 \wedge I[d\bar{\ssf}_{|\h|}(r) ]_{1,0})), 
$$
where $\bar{\ssf}_{a}(x)=\ssf_a(x)-x^2 =  \frac{a^2}{(x^4 + a^2)^{1/2} + x} - a \log((x^4 + a^2)^{1/2} + a) +   2a\log(x)$. Consider a smooth function $\varpi$ such that $\varpi=0$ if $x\leq \frac{1}{8}$ and $\varpi=1$ if $x\geq \frac{1}{4}$ and $\varpi_s(x)=\varpi(\frac{x}{s})$; then $|\varpi_s'| \leq \frac{C}{s}$. Define 
$$
\xi_s = \varpi_s \mathrm{pr}^*e_1 \wedge I[d\bar{\ssf}_{|\h|}(r)]_{1,0}.
$$
In the proof of \cite[Proposition 3.10]{LMM} we obtained:
\begin{align*}
\|[d\xi_s]_{2,1}\|_{h_1} =& \| \varpi_s \mathrm{pr}^*e_1 \wedge [dI[d\bar{\ssf}_{|\h|}(r)]_{1,0}]_{2,0}  + [d\varpi_s]_{1,0}\wedge \mathrm{pr}^*e_1 \wedge I[d\bar{\ssf}_{|\h|}(r)]_{1,0}\|_{h_1} \\ =&  O(r^{-2}) + O(r^{-1}s^{-1}).
\end{align*}

We now check $[\tilde\mu(d\xi_s)]_{2,1}=[d\xi_s]_{2,1}$. As argued in the proof of equation \eqref{eqn:2,1resol}, it suffices to prove a similar equality for $(\chi^{-1})^*d\xi_s$, namely  $[\mu ((\chi^{-1})^*d\xi_s)]_{2,1}=[(\chi^{-1})^*d\xi_s]_{2,1}$, on $(\nu(L)-Z)/\iota$.
Since $\kappa^*\varpi_s= \varpi_s$, 
\begin{align*}
    \mu((\chi^{-1})^*d\xi_s)=&d\varpi_s \wedge \mu(\pi^* e_1 \wedge I[d\bar{\ssf}_{\|\theta\|\circ \pi}(r)]_{1,0}) + \varpi_s \mu(\pi^*de_1 \wedge I[d\bar{\ssf}_{\|\theta\|\circ \pi}(r)]_{1,0}) \\ &+
\varpi_s \mu(\pi^*e_1 \wedge dI[d\bar{\ssf}_{\|\theta\|\circ \pi}(r)]_{1,0}).
\end{align*}
 Since $\ssf_{\|\theta\|\circ \pi}$ is $\kappa$-invariant, Lemmas \ref{lem:p-vert} and \ref{lem:1-forms} allow us to conclude:
$$
[\mu((\chi^{-1})^*d\xi_s)]_{2,1}=[d\varpi_s]_{1,0} \wedge \pi^* e_1 \wedge I[d\bar{\ssf}_{\|\theta\|\circ \pi}(r)]_{1,0} +
\varpi_s  \pi^*e_1 \wedge [dI[d\bar{\ssf}_{\|\theta\|\circ \pi}(r)]_{1,0}]_{2,0}= [d((\chi^{-1})^*\xi_s)]_{2,1}.
$$
 Therefore, 
$\|[\tilde\mu(d\xi_s)]_{2,1}\|_{h_1}= \|[d\xi_s]_{2,1}\|_{h_1}=O(r^{-2}) + O(r^{-1}s^{-1})$.

We set $s_0>1$ such that for every $s>s_0$ we have
$| O(r^{-2}) +  O(r^{-1}s^{-1})|< m$
on $\frac{s}{8} \leq r \leq \frac{s}{4}$. Here, $m$ is the constant from Lemma \ref{lem:universal-m}. Given $s>s_0$, we can find $t_s'<t_s$ such that on $\frac{s}{8} \leq r \leq \frac{s}{4}$ we have,
$$
t\| [\tilde\m(d\xi_s)]_{1,1} \|_{h_1} <m, \qquad t^2 \| [\tilde\m(d\xi_s)]_{0,2} \|_{h_1} <m.
$$
Fixed $s>0$, using equation \eqref{eqn:2,1resol}, and the fact that $[\tilde\m(\widehat\b)]_{3,0}=0$, we obtain
$
\| \tilde\m(\P_2^t) - \P_1^t \|_{h_t} = t \|[\tilde\m(\widehat \b)]_{1,2}\|_{h_1} + t^2\|[\tilde\m(\widehat \b)]_{0,3}\|_{h_1}.
$
Hence, we can assume that on $\nu_{2s}(L)$ we have
$
\|\m(\P_2^t) - \P_1^t \|_{h_t} < \frac{m}{4}
$ 
for $t<t_s'$. In particular, $\tilde\m(\P_2^t)$ is a $\Gtwo$ form on $\nu_{2s}(L)$. The closed form
$$
\P_{3,s}^t = \tilde\m(\P_2^t) -  \frac{t^2}{4} \tilde\m(d\xi_s),
$$
coincides with $\tilde\m(\P_2^t)$ on $r\leq \frac{s}{8}$, and with $\chi^*(\widehat \p_{3,s}^t)$ on $r\geq \frac{s}{4}$. Hence, $\P_{3,s}^t$ is a $\Gtwo$ form outside the neck $\frac{s}{8} \leq r  \leq \frac{s}{4}$. There, we have that:
$$ 
\| \P_{3,s}^t - \P_1^t\|_{h_t} 
\leq
 \| \P_{3,s}^t - \tilde\m(\P_2^t) \|_{h_t} 
 + \| \tilde\m(\P_2^t) - \P_1^t \|_{h_t} < \| \P_{3,s}^t - \tilde\m(\P_2^t) \|_{h_t} + \frac{m}{4}.
$$
The fact that $\P_{3,s}^t$ is a $\Gtwo$ form also on $\frac{s}{8} \leq r  \leq \frac{s}{4}$ follows from Lemma \ref{lem:universal-m} and the following inequality:
\begin{align*}
4\| \P_{3,s}^t - \tilde\m(\P_2^t) \|_{h_t}
&= t^2\| \tilde\m(d\xi_s) \|_{h_t}
= \| [\tilde\m(d\xi_s)]_{2,1} \|_{h_1} + t\| [\tilde\m(d\xi_s)]_{1,2} \|_{h_1} + t^2 \| [\tilde\m(d\xi_s)]_{0,3} \|_{h_1} \\
&= O(r^{-2}) + O(r^{-1}s^{-1}) + t\| [\tilde\m(d\xi_s)]_{1,2} \|_{h_1} + t^2 \| [\tilde\m(d\xi_s)]_{0,3}\|_{h_1} < 3m.
\end{align*}
\end{proof}

Finally, we glue an annulus around $Z$ on $(\nu(L)/\iota ,\p_2)$ with an annulus around $Q$ on $(P, {\P}_{3,s}^t)$ using the map $F_t \circ \chi $.
\begin{theorem} \label{theo:resol}
Let $\e_0<\frac{R}{4}$ and $s_0>0$ be the values provided by Propositions \ref{prop:inter-nu} and 
\ref{prop:inter-resol}.
Let $s>s_0$, and let $D_s(Q)$ be the disc in $P$ centered at $Q$ with radius $s$. If $t<\min \{t_s', \frac{\e_0}{s}\}$, then there is a closed $\Gtwo$ resolution $\rho \colon (\widetilde{X}_t, \tilde{\varphi}_t,\widetilde{g}) \to ({X}, {\varphi},{g})$, of the form
\begin{equation} \label{eqn:resolution-definition}
    \widetilde X_t= X- \exp(\nu_{\e}(L)/\iota) \cup_{\exp \circ F_t \circ \chi} D_s(Q),
\end{equation}
for $\e=ts>0$. In addition, there are finite-order diffeomorphisms $\tilde{\kappa}_t\colon \widetilde X_t \to \widetilde X_t$ that lift $\kappa$ and such that $\tilde{\kappa}_t^*(\tilde{\varphi_t})=\tilde{\varphi}_t$.
\end{theorem}
\begin{proof}
Given such values for $s$ and $t$, we define $\e=st \in (0,\e_0)$. The map $F_t \circ \chi$ identifies $\frac{s}{2}\leq r \leq 2 s$ on $P$ with $\frac{\e}{2} \leq r \leq 2\e$ on $\n(L)/\iota$. In terms of equation \eqref{eqn:resolution-definition}, the projection $\rho \colon \widetilde{X}_t \to X$ is defined by $\rho|_{X}=\rId$ and $\rho|_{D_s(Q)}=\exp \circ F_t \circ \chi$.
The lift of $\kappa$ on $D_s(Q)$ is determined by $\tilde{\kappa}$, it is well-defined because $\kappa$ commutes with both $F_t$ and $\exp$.  

Consider $\tilde\kappa$-invariant $\Gtwo$ form $\P_{3,s}^t$ on $\chi^{-1}(\n(L)_{2s}/\iota)$ and the $\kappa$-invariant form $\p_{3,\e}$ on $X$. On the region $\frac{s}{2} \leq r \leq 2s$ of $\chi^{-1}(\n_{2s}(L)/\iota)$ we have that $\P_{3,s}^t = \chi^*(\widehat \p_{3,s}^t)= \p_2^t$, and on the annulus $\frac{\e}{2} \leq r \leq 2\e$ of $\n(L)/\iota$ we have $\p_{3,\e}=\p_2$. Since $(F_t\circ \chi)^* \p_2 = \chi^*(\p_2^t)$, there is a $\tilde\kappa$-invariant $\Gtwo$ form $\tilde\varphi_t\in \Omega^3(\widetilde{X}_t)$. 
\end{proof}

\begin{remark}
The manifolds $\widetilde{X}_t$ are canonically diffeomorphic. Hence, we denote them by $\widetilde{X}$.
\end{remark}

\subsection{Cohomology of $\widetilde{X}/\tilde{\kappa}$}
Proposition \ref{prop:cohom-alg} allows us to compute $H^*(\widetilde{X})$, assuming that at each $L_j$, there is a projection $\mathrm{q}_j \colon L_j \to S^1$ such that $\theta|_{L_j}=\mathrm{q}_j^*(dt)$. To determine the $\tilde{\kappa}$-invariant part, we first study
the action of $\tilde{\kappa}$ on the Thom classes $[\tau_j]$ of $Q_j=Q|_{L_j}$.

\begin{lemma}
    \label{lem:Thom}
If $\kappa(L_j)=L_k$, then
 $\sigma|_{L_j}\tilde{\kappa}^*(\tau_k)$ is a Thom form of $Q_j$.
\end{lemma} 
\begin{proof}
Let $x=[F_p,[0,\ell]]\in Q$. The fiber of the normal bundle of $Q\subset P$ at $x$ is 
the tangent space of $\{ [F_p,[v,\ell]], \quad v\in \CC \}$ at $v=0$ and it is oriented by $\frac{i}{2}dv\wedge d\bar{v}$.
From our description, we see that if $\sigma|_{L_j}=1$, then $\tilde{\kappa}[G_p,[v,\ell]]=[G_p \circ(\kappa|_p)^{-1},[v,\ell]]$. Therefore, $\tilde{\kappa}$ preserves the orientation of the fibers of the normal bundle. 
 However, if $\sigma|_{L_j}=-1$ then $\tilde{\kappa}[G_p,[v,\ell]]=[c(G_p)\circ (\kappa|_p)^{-1},\overline{[{v},{\ell}]}]$, so that $\tilde{\kappa}$ reverses the orientation of the fibers. Hence $\tilde{\kappa}^*\tau_k$ integrates to $\sigma|_{L_j}$ over the fibers of the normal neighborhood of $L_j$, and is supported around $Q_j$. Hence, $\sigma|_{L_j} \tilde{\kappa}^*\tau_k$ is a Thom form of $Q_j$. 
\end{proof}
\begin{lemma} \label{lem:cohomology-class-tildevarphi}
In terms of the splitting in Proposition \ref{prop:cohom-alg}, 
$[\tilde{\varphi}_t]=
 [\varphi] - \pi t^2 \sum_j [\pi^*\theta]\otimes \mathbf{x}_j.
 $
\end{lemma}
\begin{proof}
Using that $\beta_2$ and $\pi^*\theta \wedge d(|\theta|^{-1} I [d \sg_{|\theta|}]_{1,0}) $  are exact we obtain:
$$
[\tilde{\varphi}_t]|_{Q_j}=[\tilde{\mu}{(\Phi_2^t)}]|_{Q_j}= [\mathrm{pr}^*(\varphi|_{L_j}) + t^2 \tilde{\mu}(\mathrm{pr}^*\theta \wedge \widehat{\omega}_1)]|_{Q_j}
= [\mathrm{pr}^*(\varphi) - \frac{t^2}{2} \tilde{\mu}(\pi^*\theta \wedge d(I [d\log(r ^2)]_{1,0})]|_{Q_j}.
$$
Let $R_j$ be the curvature of the tautological line bundle $S_j\to Q_j$. Notations as in Proposition \ref{prop:cohomology-exceptional-divisor}, 
from Remarks \ref{rmk: Euler-class} and \ref{rmk:curvature} it follows $[\tau_j]|_{Q_j}= 2\mathrm{e}_j=\frac{i}{\pi}[R_j]= \frac{1}{2\pi} d(I [d\log(r ^2)]_{1,0})])|_{Q_j}$.
 Lemma \ref{lem:Thom} shows
\begin{equation*} \label{eqn:class-restricted}
[\tilde{\varphi}_t]|_{Q_j}=
 [\mathrm{pr}^*(\varphi) - \pi t^2 \pi^*\theta \wedge \tau_j]|_{Q_j}.
\end{equation*}
Proposition \ref{prop:cohomology-short-sequence} implies that $[\tilde{\varphi}_t + \pi t^2 \sum_j \pi^*\theta \wedge \tau_j]\in \mathrm{Im}(\rho^*)$. We finally prove that $\xi= [\tilde{\varphi}_t + \pi t^2 \sum_j \pi^*\theta \wedge \tau_j- \rho^*(\varphi)]$ vanishes on $H^*(\widetilde{X})$. Since $\xi \in \mathrm{Im}(\rho^*)$, it suffices to check that $\xi \in \oplus_j \mathrm{pr}^*(H^1(L_j))\wedge [\tau_j]$ by Proposition \ref{prop:cohomology-short-sequence}. Using that 
$[\exp_*(\phi_2)]|_{X-L}=[\varphi]|_{X-L}$, we obtain $\xi|_{\widetilde{X}-D_s(Q)}=[\tilde{\varphi}_t - \rho^*(\varphi)]_{\widetilde{X}-D_s(Q)}=0$. The long exact sequence of the pair $(X,X-D_s(Q))$ shows $\xi \in H^3(\widetilde{X}, \widetilde{X}-D_s(Q))$.
Since $\widetilde{X}-Q$ retracts onto $\widetilde{X}-D_s(Q)$,  Lemma \ref{lem:relative-to-complement} ensures
$\xi \in \oplus_j \mathrm{pr}^*(H^1(L_j))\wedge [\tau_j]$.
\end{proof}
From Propositions \ref{prop:cohom-alg} and  \ref{prop:pont}, using that $L_j$ are $\varphi$-calibrated submanifolds, and the orientations discussed at the end of section \ref{subsec:pre-resol}, it follows
\begin{equation} \label{eqn:integral of pont}
    \int_{\widetilde{X}}{ p_1(\tilde{X})\wedge \widetilde{\varphi}_t}= \frac{1}{2}\int_M {p_1(M)\wedge \varphi} -3\sum_{j} \vol(L_j) +  8t^2\pi \sum_{j} (1-g_j) \int_{L_j} \theta \wedge \omega_j, \quad \int_{L_j} \theta \wedge \omega_j>0.
\end{equation}

For the remainder of the section, we additionally assume:
\begin{enumerate}
\item The map $\kappa$ is an involution.
\item If $\kappa(L_j)=L_k\neq L_j$, then $\theta|_{L_k}=\kappa^*(\theta|_{L_k})$. Hence, by (1),  $\sigma|_{L_j}=\sigma|_{L_k}=1$.
\end{enumerate}

We divide the set of connected components of $L$ into two subsets. Define  $\mathcal{L}_F^{\pm }$ as the set of connected components $L_j$ fixed by $\kappa$ such that $\sigma|_{L_j}=\pm 1$, and $\mathcal{L}_S$ the set of connected components that are not fixed by $\kappa$. As $\kappa$ is an involution,  there are indices $j_1,\dots, j_r$ such that $\mathcal{L}_S= \{L_{j_s}\}_{s=1}^r\sqcup \{\kappa(L_{j_s})\}_{s=1}^r$.
We also denote $H^*(L_j)^{\pm}=\{ [\beta] \in H^* (L_j)\colon \kappa^*[\beta]=\pm[\beta]\}$ if $L_j \in \mathcal{L}_{F}^\pm$. 
\begin{proposition}\label{prop:cohomology-orbifold}
In the notations of Proposition \ref{prop:cohom-alg},
there is an isomorphism between $H^*(\widetilde{X})^{\tilde{\kappa}}$ and
\begin{align*}
    & H^*(M)^{\la \iota, \kappa \ra}\oplus \left( \oplus_{s=1}^r  (\Id + \kappa)^*(H^{*-2}(L_{j_s})\otimes \mathbf{x}_{j_s}) \right) \\
    &\oplus \left( \oplus_{L_j \in \mathcal{L}_F^+} H^{*-2}(L_j)^+ \otimes \mathbf{x}_j  \right) \oplus \left(  \oplus_{L_j \in \mathcal{L}_F^-} H^{*-2}(L_j)^- \otimes \mathbf{x}_j  \right),
\end{align*}
where we defined $\kappa^*(\mathbf{x}_{j}):=\mathbf{x}_k$
if $L_{j} \in \mathcal{L}_S$ and $\kappa(L_j)=L_k$.
\end{proposition}
\begin{proof}
We first note that $\tilde{\kappa}$ acts on the subspaces $H^*(X)$ and $\oplus_j H^{*-2}(L_j)\otimes \mathbf{x}_j$ . It restricts to $\kappa^*$ on the first of these, and as $\kappa^* \otimes \kappa^*$ in the second, where we denote $\kappa^*(\mathbf{x}_j)=\sigma|_{L_j}\mathbf{x}_k$ if $\kappa(L_j)=L_k$. This notation is consistent with Lemma \ref{lem:Thom}. An average argument yields $H^*(X)^{\kappa}=H^*(M)^{\la \iota,\kappa \ra}$. In 
addition, if $L_j \in \mathcal{L}_F^{\pm}$ then $H^{*-2}(L_j)\otimes \mathbf{x}_j$ is $(\kappa^*\otimes \kappa^*)$-invariant; its fixed part is $H^{*-2}(L_j)^\pm \otimes \mathbf{x}_j$ when $L_j \in \mathcal{L}_F^\pm$. Finally, if $L_j \in \mathcal{L}_S$, then $H^*(L_j)\otimes \mathbf{x}_j \oplus H^*(\kappa(L_j))\otimes \kappa^*(\mathbf{x}_j)$ is $(\kappa^*\otimes \kappa^*)$-invariant, and its fixed part is $(\Id + \kappa)^*(H^*(L_{j_s})\otimes \mathbf{x}_{j_s})$. 
\end{proof}

\begin{remark}\label{rem:cohomology-double-resolution}
  Let $K=\Fix(\tilde{\kappa})$, and denote by $K_j$ its connected components. Assume that these are the mapping torus of a certain volume-preserving diffeomorphism $f_j \colon \Sigma_j \to \Sigma_j$, where $\Sigma_j$ is a surface of genus $g_j$. Then, there exists a closed $\Gtwo$ resolution
$Z$ of $\widetilde{X}/\tilde{\kappa}$. Its cohomology algebra is isomorphic to
\begin{align*}
 &H^*(\widetilde{X})^{\tilde{\kappa}}
 \oplus (\oplus_j H^*(K_j)\otimes \mathbf{y}_j),
\end{align*}
Now $\mathbf{y}_j$ represents the Thom class of $\rho^{-1}(K_j)$ on the resolution.
Let  $[\omega_j]\in H^2(K_j)$ be the
the cohomology class
determined by the volume form of $\Sigma_j$ that integrates to $1$, then
$\mathbf{y}_j^2= -2 PD[K_j] + (4-4g_p)[\omega_j]\otimes \mathbf{y}_j$.  In addition, every $L_{j_s}\in \mathcal{L}_S$, is disjoint from $\Fix(\kappa)$ in $M$. Therefore, the connected components of the exceptional divisor of $\widetilde{X}$ corresponding to $\cup L_{j_s}$ are disjoint from $\Fix(\tilde{\kappa})$, and we can assume  $(\Id + \kappa)^*(\mathbf{x}_{j_s}) \cdot \mathbf{y}_j=0$. 
\end{remark}

\section{Compact manifolds with a closed $\Gtwo$ structure satisfying the known topological properties of compact holonomy $\Gtwo$ manifolds} \label{sec:examples} 

\subsection{Set-up}\label{sec:set-up}

Let $N$ be a $6$-dimensional manifold, let $F \in \Diff(N)$, and set $M=N_F$ be the mapping torus of $F$, defined as in expression \eqref{eq:MT-def}.
Suppose that there are both a path $\{(\omega_t, J_t)\}_{t \in [0,1]}$ of Hermitian structures on $N$, and a path $\{ \Theta_t= \psi^+_t + i \psi_t^-\}$ of $J_t$-complex volume forms. If $F^*(\omega_1)=\omega_0$, and $F^*(\psi_1^+)=\psi_0^+$, then
$
\varphi= dt \wedge \omega_t + \psi_t^+,
$
determines a $\Gtwo$ structure on $N_F$. 
It is closed if
\begin{equation} \label{eq:closed}
d\omega_t= -\dot{\psi}_t^+, \mbox{ and } d\psi_t^+=0.
\end{equation} 
Consider an involution $\eta \colon N \to N$ that satisfies $\eta \circ F = F^{-1} \circ \eta$. The map $\kappa \colon N_F \to N_F$, $\kappa[t,p]=[1-t, \eta(p)]$ preserves $\varphi$ if
\begin{equation} \label{eq:j-G2-invol}
\eta^*(\omega_t)= - \omega_{1-t}, \mbox{ and } \eta^*(\psi^+_t)= \psi^+_{1-t}.
\end{equation} 
If there is also an involution $\xi \colon N \to N$ satisfying  $\xi \circ F = F \circ \xi$. The map $\iota \colon N_F \to N_F$, $\iota[t,p]=[t, \xi(p)]$ preserves $\varphi$ if $\xi$ preserves both the paths $\{\omega_t\}$ and $\{\psi_t^+\}$.  We assume $\xi \circ \eta=\eta \circ \xi$ so that the group generated by $\iota$ and $\kappa$ is $\Z_2 \times \Z_2$. 

In this situation, we have closed $\Gtwo$ orbifolds $(X^1=N_F/\kappa,\varphi)$ and $(X^2=N_F/\la \iota,\kappa \ra,\varphi)$. The singular locus of $X^1$ and $X^2$ are respectively $\Fix(\iota)$, and $\Fix(\iota)\cup \Fix(\kappa)\cup \Fix(\iota \circ \kappa)$. Equations \eqref{eq:fix-1} and \eqref{eq:fix-2} allow us to compute these. Further assumptions allow us to resolve the orbifolds $X^1$ and $X^2$ using \cite[Theorem 3.11]{LMM} and Theorem \ref{theo:resol}.

\subsection{Examples} \label{subsec:ex}

Our examples are constructed from three different complex $6$-tori that we now describe. Consider the lattices of $\CC$
\begin{equation*}
\Gamma_1= \Z \la 1, e^{\frac{2\pi i}{3}} \ra, \qquad \Gamma_{2}= \Z \la 1, i \ra, \qquad \Gamma_3= \Z \la 1+i, 1-i \ra \subset \Gamma_2,
\end{equation*}
and the tori 
$$T^6_k = (\CC/\Gamma_k)^3, \qquad k=1,2, \qquad T^6_3=\C^3/(\Gamma_3 \times \Gamma_2^2).
$$
We proceed to define diffeomorphisms $F_{k,a}\colon T^6_k  \to T^6_k$ for $(k,a)\in \{ (1,3),(1,6),(2,4),(3,4)\}$, 
as well as paths $\{(\omega_t, J_t)\}$, and $\{\Theta_t\}$ and involutions $\eta_k,\xi_k \colon T^6_k \to T^6_k$, for $k=1,2,3$ and $\zeta_1,\zeta_2 \colon T^6_k \to T^6_k$ for $k=2,3$. In the sequel, we denote points on $\CC^3$ by $z=(z_1,z_2,z_3)$, and points on $T^6_k$ by $[z]$.

\begin{enumerate}

\item  The mapping diffeomorphisms $F_{k,a}$ are the composition of a complex rotation $\mathrm{R}_a$ on $T^6_k$ and a diffeomorphism $f$, which we now define. Given  $a \in \R$, we let
\begin{equation*} 
\mathrm{R}_a \colon \CC^3 \to \CC^3, \quad \mathrm{R}_a(z)=(e^{\frac{2\pi i}{a}}z_1, e^{\frac{2\pi i}{a}}z_2, e^{-\frac{4 \pi i}{a}}z_3).
\end{equation*}
For the moment, we only need to take $a=3,4,6$.
Observe that $\mathrm{R}_3$ and $\mathrm{R}_6$ descend to $T^6_1$ while $\mathrm{R}_4$ descends to $T^6_2$ and $T^6_3$.  Let
\begin{equation*}
f \colon \CC^3 \to \CC^3, \quad f(z)=(z_1,z_2-z_1,z_3),
\end{equation*}
and note it descends to $T^6_k$ for $k=1,2,3$, and it 
commutes with $\mathrm{R}_a$. The diffeomorphisms are:
\begin{align*}
&F_{1,a} \colon T^6_1 \to T^6_1, \quad  F_{1,a}=f \circ \mathrm{R}_a,  \qquad a\in \{3,6\}, \\
&F_{k,4}\colon T^6_k \to T^6_k, \quad F_{2,a}=f \circ \mathrm{R}_4, \qquad k \in \{2,3\}.
\end{align*}
\item The paths $\{J_t\}$ and $\{\Theta_t\}$ are constant and determined by the standard complex structure on $\C^3$ and the standard volume form respectively. The volume form is $F_{k,a}$-invariant. We define path of  non-degenerate $(1,1)$-forms $\omega_t$:
\begin{equation} \label{eq:def-om}
\omega_t=\frac{i}{2} \left( (1+t^2) dz_{1\bar{1}} + dz_{2\bar{2}} + dz_{3\bar{3}} + t\,(dz_{1\bar{2}} - dz_{\bar{1} 2}) \right).
\end{equation}
The equality  $F_{k,a}^*(\omega_1)= \omega_0$ follows from 
$\mathrm{R}_a^*(\omega_t)= \omega_t$ and $f^*(\omega_1)=\omega_0$, the latter is obtained from $f^*(dz_{2\bar{2}})= dz_{1\bar{1}} - (dz_{1\bar{2}} - dz_{\bar{1} 2}) + dz_{2\bar{2}}$.
\item Consider the involutions $\xi,\eta \colon \CC^3 \to \CC^3$,
 \begin{align*}
  \xi(z)=&(-z_1,-z_2,z_3),\\
 \eta(z)=&(-\overline{z}_1,\overline{z}_2+ \overline{z}_1,-\overline{z}_3).
 \end{align*}
Denote by $\xi_k, \eta_k\colon T^6_k \to T^6_k$ the induced maps. It is clear that $\xi_k$ satisfies all the assumptions discussed in section \ref{sec:set-up}, and that $\xi_k \circ \eta_k= \eta_k \circ \xi_k$. Concerning $\eta_k$, we
 observe that $\eta \circ \mathrm{R}_a= \mathrm{R}_a^{-1} \circ \eta$, and
 $\eta \circ f = (-\overline{z}_1,\overline{z}_2,-\overline{z}_3)$ is an involution; therefore, $f_k \circ \eta= f_k^{-1} \circ \eta$. 
These equalities yield $\eta_k\circ F_{k,a}=F_{k,a}^{-1} \circ \eta_k$. To check that $\eta_k$ satisfies the conditions in equation \eqref{eq:j-G2-invol}, we first observe that $\eta_k^*(\psi_t^+)=\psi_t^+$  because $\eta^*(dz_{123})= dz_{\bar{1}\bar{2}\bar{3}}$ on $\CC^3$.  In addition, the equality $\eta_k^*(\omega_t)=-\omega_{1-t}$ follows from $\eta^*(dz_{p\bar{p}})= -dz_{p\bar{p}}$ for $p\neq 2$, as well as the computations
\begin{align}
\eta^*(dz_{2\bar{2}})=& -dz_{1\bar{1}}+ dz_{1\bar{2}} - dz_{\bar{1}2} -dz_{2\bar{2}}, \\
\eta^*(dz_{1\bar{2}} - dz_{\bar{1}2})=&-2 dz_{1\bar{1}}-dz_{1\bar{2}} + dz_{\bar{1}2}.
\end{align}

\item Consider the maps $\zeta^1,\zeta^2 \colon \CC^3 \to \CC^3$, 
$$ 
\zeta^1(z_1,z_2,z_3)=(-z_1,-z_2,z_3+1/2), \qquad \zeta^2(z_1,z_2,z_3)=(-z_1,-z_2,z_3+i/2).
$$
For $k=2,3$, these maps descend to involutions $\zeta^1_k,\zeta^2_k \colon T^6_k \to T^6_k$, which commute with both $F_{k,4}$ and $\eta_k$. They preserve the paths $\{\omega_t\}$ and $\{\psi_t^+\}$.

\item If $k=1,2$, the fixed points of the involutions $\eta_k$ are determined by the equations: 
 \begin{equation} \label{eqn:fixed-at-1/2}
z_1 \equiv -2i \mathrm{Im}(z_2) \mod \Gamma_{k}, \qquad 2\mathrm{Re}(z_3) \equiv 0 \mod \Gamma_{k}.
\end{equation}
If $k=3$, we have the same equations as for $k=2$. To check that in this case, $z_1 \equiv -\overline{z}_1 \mod  \Gamma_3$ one observes that $z_1+\overline{z}_1 \in 2\Z \subset \Gamma_3$.  Similarly, when $k=1,2$, a fixed point $[z_1,z_2,z_3]$ of $\eta_k\circ \xi_k$ satisfies $z_1 \equiv -2\mathrm{Re}(z_2) \mod \Gamma_k$, and $2\mathrm{Re}(z_3) \equiv 0 \mod \Gamma_k$. For $k=3$, the equations are again the same as for $k=2$.

 \item When $k=2,3$ we will direcly compute $\Fix(\eta_2 \circ f \circ \mathrm{R}_{4})$ and  $\Fix(\xi_2 \circ \eta_2 \circ f \circ \mathrm{R}_{4})$. However, we will rely on the following observation for $k=1$. 
 Given $a=3,6$, using the equalities $\mathrm{R}_{a}=\mathrm{R}_{2a}^2$, and $\mathrm{R}_{2a}^{-1} \circ \eta= \eta \circ \mathrm{R}_{2a}$, we obtain $\eta_1\circ f \circ \mathrm{R}_{a}= \mathrm{R}_{2a}^{-1} \circ \eta_1 \circ f \circ \mathrm{R}_{2a} $. Hence,
 \begin{equation} \label{eqn:fixed-at-0}
 [z]\in \Fix(\eta_1 \circ f \circ \mathrm{R}_{a}) \quad \mbox{ if and only if } \quad
 (\eta_1 \circ f)(\mathrm{R}_{2a}(z)) \equiv R_{2a}(z) \mod R_{2a}(\Gamma_1^3).
 \end{equation}
 In addition, $\eta_1 \circ f[z]=[-\bar{z}_1,\bar{z}_2,-\bar{z}_3]$. Hence, to determine  $\Fix( \eta_1 \circ f \circ \mathrm{R}_{a})$ it suffices to solve  
 $
 \eta_1 \circ f(z) \equiv z \mod  \mathrm{R}_{2a}(\Gamma_1^3)
 $
 and then compute $\mathrm{R}_{-2a}(z)$. The solutions to the congruence equation are $z\in \C^3$ that satisfy
 $$
 2\,(\mathrm{Re}(z_1),i\mathrm{Im}(z_2), \mathrm{Re}(z_3)) \in  \mathrm{R}_{2a}(\Gamma_1^3).
 $$

 Similarly, $[z]\in \Fix(\xi_1 \circ \eta_1 \circ f \circ \mathrm{R}_{a})$ if and only if 
 $(\xi_1 \circ \eta_1 \circ f)(\mathrm{R}_{2a}(z)) \equiv R_{2a}(z) \mod R_{2a}(\Gamma_1^3)$ and
 $\xi_1 \circ \eta_1 \circ f[z]=[\bar{z}_1,-\bar{z}_2,-\bar{z}_3]$. A similar strategy can be used to compute $\Fix(\xi_1 \circ \eta_1 \circ f \circ \mathrm{R}_{a})$, which is given by the classes $ [\mathrm{R}_{-2a}(z)]$ where $ z\in \C^3$ satisfies  
 $2\,(i\mathrm{Im}(z_1),\mathrm{Re}(z_2), \mathrm{Re}(z_3)) \in  \mathrm{R}_{2a}(\Gamma_1^3)
 $.
 Note also that  $\Gamma_1^3$ is $\mathrm{R}_{6}$-invariant but it is not $\mathrm{R}_{12}$-invariant. Indeed, 
 $$
 \mathrm{R}_{12}(\Gamma_1^3)=(\C/\Z \langle e^{\pi i/6}, e^{5\pi i/6}\rangle )^2 \times \Gamma_1= (\C/\Z \langle e^{\pi i/6}, i \rangle )^2 \times \Gamma_1.
 $$

\end{enumerate}

Define $M_{k,a}$ the mapping torus of $F_{k,a}$, and denote by $\varphi$ the $\Gtwo$ structure determined by the paths $\{\omega_t,J_t\}$ and $\{\Theta_t\}$. Let $\iota$ and $\kappa$ be the involutions of $M_{k,a}$ determined by $\xi_k$ and $\eta_k$. In addition, if $k=2,3$ we set $\jota_1$, and $\jota_2$ on $M_{k,4}$ the involutions determined by $\zeta^1_k$ and $\zeta^2_k$. Define the closed $\Gtwo$ orbifolds:
\begin{align*}
&(X_{k,a}^1=M_{k,a}/\kappa, \varphi), \quad &(k,a)\in \{(1,3),(1,6),(2,4), (3,4)\}, \\
&(X^2_{k,a}=M_{k,a}/\la \iota, \kappa \ra, \varphi),\quad  &(k,a)\in \{(1,3),(1,6),(2,4)\}, \\
&(Z_{k,4}=M_{k,4}/\la \jota_1, \jota_2, \kappa \ra, \varphi), \quad &k=2,3.
\end{align*}
 We introduced the orbifolds $X_{3,4}^1$ and $Z_{3,4}$ because they are the universal cover of $X_{2,4}^1$ and $Z_{3,4}^1$ respectively (see the proof of Proposition \ref{prop:fund-group-2}). Since $X_{3,4}^2$ is simply connected, we will not deal with $X_{3,4}^2$.

\begin{remark} \label{rem:nilmanifolds} In \cite[Theorem 4]{CF}, they find a left-invariant closed $\Gtwo$ structure in the Lie group $\R \times \C^3$ endowed with the product $(t,z)\cdot (s,w) = (t+s, z+\rho(t)(w))$, with $\rho(t)=\rId + t(f-\rId)=\exp(tf-t\rId)$, and $f$ defined as above. The Lie algebra of this group is denoted by $(0,0,0,0,12,13,0)$ in that result.
If $k=1,2$ we set $\Lambda_{k,a}$ the subgroup spanned by $a\Z\times \{0\}$ and $\{0\} \times\Gamma_k^3$ . If $k=3$, the lattice $\Lambda_{k,a}$ is spanned by $a\Z\times \{0\}$ and $\{0\} \times\Gamma_3\times  \Gamma_2^2$. Then $M_{k,a}'= (\R \times \C^3)/\Lambda_{k,a}$ is a nilmanifold with a closed $\Gtwo$ structure. Note that $M_{k,a}'=(\R \times T^6_k)/a\Z$,
where $a\Z$ acts as $an\cdot (t,[z])=(t+an, f^{an}[z])$ because $\rho(an)(z)=(z_1,z_2-an z_1, z_3)=f^{an}(z) $. 
In addition, equation \eqref{eq:MT-def} yields
$M_{k,a}= (\R \times T^6_k)/\Z$, with $\Z$ acting as  $n\cdot (t,z)= (t+n,F_{k,a}^n(z))$. Since $F_{k,a}^a=f^a$, there is a $a:1$ covering
$$
M_{k,a}' \to M_{k,a}, \quad [t,[z]]\to [t,[z]].
$$
The $\Gtwo$ structure on $M_{k,a}'$ pushes forward to $M_{k,a}$ and coincides with the closed $\Gtwo$ structure we described.   Another closed $\Gtwo$ orbifold can be constructed from the nilpotent Lie algebra $(0,0,0,12,13,14,15)$ (see loc. cit. Lemma 6). 
\end{remark}

\begin{remark} The torsion of the constructed $\Gtwo$ structures is not small. From 
$\star \varphi= \frac{1}{2} \omega_t \wedge \omega_t -dt\wedge \mathrm {Im}(\Theta)$, we deduce
$d(\star \varphi)=- \frac{t}{2}  dt\wedge dz_{1\bar{1}2\bar{2}} -\frac{t}{2}  dt\wedge dz_{1\bar{1}3\bar{3}} -\frac{1}{4}  dt\wedge dz_{1\bar{2}3\bar{3}}+ \frac{1}{4}  dt\wedge dz_{\bar{1}23\bar{3}} $.
In addition, the associated metric is
$
dt^2 + (1+t^2)dx_1^2 + 2t\, dx_1\cdot dx_2 + dx_2^2 +
(1+t^2)dy_1^2 + 2t\, dy_1\cdot dy_2 + dy_2^2 
+dx_3^2+dy_3^2.
$
\end{remark}

\begin{remark}
Some orbifolds remain to be studied, whose analysis is analogous to that carried out in the next sections. Obvious cases include $M_{k,a}/\kappa \circ \iota$, for all the pairs $(k,a)$ considered, as well as $M_{3,4}/\langle \iota, \kappa \rangle$, and $M_{k,4}/\langle \jota_i,\kappa\rangle$ for $k=2,3$ and $i=1,2$. Similar to $X_{3,4}^1$, one could also consider the lattice $\Gamma_4 = \Z\langle 3, 1+ e^{\pi i /3} \rangle \subset \Gamma_1$, and define $M_{4,3}$ as the mapping torus of $f \circ R_3 \colon \C^3/(\Gamma_4\times \Gamma_1^2) \to \C^3/(\Gamma_4\times \Gamma_1^2)$. One could then study the orbifolds $M_{4,3}/\kappa$, $M_{4,3}/\iota \circ \kappa$, $M_{4,3}/\langle \iota,\kappa \rangle$.
\end{remark}

\subsubsection{Singular locus}\label{subsec:sing-locus}

We first describe the singular locus of $X_{k,a}^1$ and $X^2_{k,a}$ for $k=1,2$. As shown by equations \eqref{eq:fix-1} and \eqref{eq:fix-2}, these loci are determined by $\Fix(\eta_1)$, $\Fix(\eta_1 \circ F_{1,a})$, $\Fix(\xi_1 \circ \eta_1)$ and $\Fix(\xi_1 \circ \eta_1 \circ F_{1,a})$, which we compute using the formulas \eqref{eqn:fixed-at-1/2}, \eqref{eqn:fixed-at-0}, and the observations in the paragraphs containing these. If
$k=1$, we obtain:
\begin{align}
\Fix(\eta_1)=& \{ [-i2y_2, x_2 + iy_2 , iy_3],   \quad 0 \leq 2y_2,y_3 \leq \sqrt{3}, \quad   -y_2/\sqrt{3} \leq x_2 \leq 1-y_2/\sqrt{3} \}, \label{eqn:M1a-half-1} \\
\Fix(\xi_1 \circ \eta_1)=& \{ [-2x_2, x_2 + i y_2 , iy_3], \quad 0 \leq 2y_2,y_3 \leq \sqrt{3}, \quad   -y_2/\sqrt{3} \leq x_2 \leq 1-y_2/\sqrt{3} \}, \label{eqn:M1a-half-2} \\
\Fix(\eta_1 \circ F_{1,a})=&\{ [\mathrm{R}_{-2a}(iy_1, x_2 , iy_3)], \quad 0\leq y_1 \leq \mathbf{s}_a, \quad 0 \leq x_2 \leq \mathbf{t}_a,  \quad 0 \leq y_3 \leq \sqrt{3} \}, \label{eqn:M1a-zero-1}\\
\Fix(\xi_1 \circ \eta_1 \circ F_{1,a})=&\{ [\mathrm{R}_{-2a}(x_1, iy_2 , iy_3)], \quad  0 \leq x_{1} \leq \mathbf{t}_a,  \quad 0\leq y_2 \leq \mathbf{s}_a, \quad 0 \leq y_3 \leq \sqrt{3} \},\label{eqn:M1a-zero-2}
\end{align}
where $\mathbf{t}_3=\mathbf{s}_6=1$, $\mathbf{t}_6=\mathbf{s}_3= \sqrt{3}$. 

 From equality \eqref{eq:fix-2} we obtain that both $\Fix(\kappa)$ and $\Fix(\iota \circ \kappa)$ consist of two submanifolds diffeomorphic to a torus $T^3$. The singular locus of $X_{1,a}^1$ is $\Fix(\kappa)$; we denote the connected components at $t=0$ and $t=1/2$ as $N_1$ and $N_2$ respectively.
 
 For $k=2$, we have:
 \begin{align}
\Fix(\eta_2)=& \{ [-2iy_2, x_2  + iy_2, \e_3 + iy_3 ], \quad \e_3 \in \{0,1/2\} , \quad 0 \leq  x_2,y_2, y_3 \leq 1 \}, \label{eqn:M24-half-1}\\
\Fix(\xi_2 \circ \eta_2 )=& \{ [-2x_2, x_2 + i y_2, \e_3 + iy_3 ], \quad \e_3 \in \{0,1/2\} , \quad 0 \leq  x_2, y_2, y_3 \leq 1 \},  \label{eqn:M24-half-2} \\
\Fix(\eta_2 \circ F_{2,4})=&\{ [x_1 + i x_1, x_2 - i x_2, x_3 + i \e_3 ], \quad \e_3 \in \{ 0,1/2 \}, \quad 0 \leq x_1, x_2, x_3 \leq 1 \}, \label{eqn:M24-zero-1}\\
\Fix(\xi_2  \circ \eta_2  \circ F_{2,4})=&\{ [x_1 - i x_1, x_2 + i x_2, x_3 + i \e_3 ], \quad \e_3 \in \{ 0,1/2 \}, \quad 0 \leq x_1, x_2, x_3 \leq 1 \}. \label{eqn:M24-zero-2}
\end{align}

 Equality \eqref{eq:fix-2} implies that $\Fix(\kappa)$ and $\Fix(\iota \circ \kappa)$ consist of four connected components each, two at $t=0$ and two at $t=1/2$. These are diffeomorphic to a torus $T^3$.
 The singular locus of $X_{2,4}^1$ is $\Fix(\kappa)$. We denote by $N_1$ (respectively $N_2$) the connected component at $t=0$ determined by $\e_3=0$ (respectively $\e_3=1/2$), and by $N_3$ (respectively $N_4$) the connected component at $t=1/2$ determined by $\e_3=0$ (respectively $\e_3=1/2$).

 In both cases, $k=1,2$, the involution $\iota$ preserves every connected component of $\Fix(\kappa)$. Under all our identifications, the action of $\iota$ on the tori reverses the sign of the first two coordinates and preserves the last. Hence, all the connected components of
 $\Fix(\kappa)/\iota$ are diffeomorphic to $T^2/\Z_2\times S^1$, where $T^2/\Z_2$ is the pillowcase orbifold (it is homeomorphic to $S^2$). Similarly, each connected component of 
$\Fix(\iota \circ \kappa)/\iota$ are orbifolds $T^2/\Z_2\times S^1$.  The set $\Fix(\iota \circ  \kappa)\cap \Fix(\kappa)=\Fix(\iota)\cap \Fix(\kappa)$ is the singular locus of $\Fix(\kappa)/\iota$. It consists of eight circles if $k=1$, and sixteen circles if $k=2$; in both cases, half of these lie on $t=0$. Taking into account the descriptions of the singular locus of $X_{k,a}$, we deduce that $\Fix(\iota \circ \kappa)\cup \Fix(\kappa)/ \iota $ has two connected components if $k=1$ and four if $k=2$. These components are diffeomorphic to $\Sigma^\vee \times S^1$, where $\Sigma^\vee$ is consists of $2$ pillowcase orbifolds that intersect each other at their four singular points. 
If $k=1$, these are determined by the parameter values $t=0,1/2$ and if $k=2$, these are described by the pair $(t, \varepsilon_3)\in \{0,1/2\}^2$.

Next Lemma describes $\Fix(\iota)$ on $M_{1,a}$ and $M_{2,a}$:
 \begin{lemma}\label{lem:ki}
 \begin{enumerate}
 \item The fixed locus of $\iota$ on $M_{1,a}$, $a=3,6$ consists of $4$ connected components $L_1,L_2,L_3,L_4$. The component $L_1$ is the mapping torus of $\mathrm{r}_3 \colon \C/\Gamma_1 \to \C/\Gamma_1$, $\mathrm{r}_3(z)=e^{\frac{2\pi i}{3}}z$. The remaining are diffeomorphic to a $3$-torus.
 The involution $\kappa$ preserves all these components.
 \item The fixed locus of $\iota$ on $M_{2,4}$ consists of $10$ connected components $L_1, \dots, L_{10}$. 
 The components  $L_1$, $L_2$, $L_5$ and $L_8$ are diffeomorphic to the mapping torus of $-\rId \colon \C/\Gamma_2 \to \C/\Gamma_2$. The remaining are diffeomorphic to a $3$-torus. In addition, $\kappa$ fixes $L_1,L_2,L_3,L_4$ and swaps $L_j$ with $L_{j+3}$ for $j=5,6,7$.
 \end{enumerate}
 \end{lemma}
 \begin{proof}
 We find $\Fix(\iota)$ using equation \eqref{eq:fix-1}. Since $ \xi_k[z_1,z_2,z_3]=[-z_1,-z_2,z_3]$, we have:
 $$
\Fix(\iota \circ \kappa)= q([0,1]\times \{[p_1,p_2,z], \quad  2p_1,2p_2 \in \Gamma_k,\quad z\in \C\}).
 $$
 To describe its connected components, we define $\mathrm{r}_a\colon \C \to \C$, $\mathrm{r}_a(z)=e^{\frac{2\pi i}{a}}z$, so that
 $F_{k,a}[z_1,z_2,z_3]= [\mathrm{r}_a(z_1),\mathrm{r}_a(z_1-z_2), \mathrm{r}_a^{-2}(z_3)]$, and we set $\mathrm{h}\colon (\C/\Gamma_k)^2  \to (\C/\Gamma_k)^2  $,  $\mathrm{h}[z_1,z_2]=[\mathrm{r}_a(z_1),\mathrm{r}_a(z_1-z_2)]$.
 
Let $p_1,p_2\in \frac{1}{2}\Gamma_k/\Gamma_k$, and let $s\geq 1$ such that 
$\mathrm{h}^{-s}[p_1,p_2]=[p_1,p_2]$ and $\mathrm{h}^{-j}[p_1,p_2]\neq [p_1,p_2]$ for $0<j<s$. The connected component containing $C_{[p_1,p_2]}= q([0,1]\times \{[p_1,p_2,z], \, z\in \C \}$ is
$$
\mathcal{C}_{[p_1,p_2]}=q([0,s]\times \{[p_1,p_2,z], \quad z\in \C \})=\bigcup_{j=0}^{s-1}  C_{\mathrm{h}^{-j}[p_1,p_2]},
$$
which is diffeomorphic to the mapping torus of $r_{a}^{-2s}\colon \C/\Gamma_k \to \C/\Gamma_k$. Before computing the connected components, we observe that if $p_1,p_2\in \frac{1}{2}\Gamma_k/\Gamma_k$ then,
\begin{align}\label{eq:h}
\mathrm{h}^{2r}[p_1,p_2]=[\mathrm{r}_a^{2r}(p_1), \mathrm{r}_a^{2r}(p_2)], \qquad 
\mathrm{h}^{2r+1}[p_1,p_2]=[\mathrm{r}_a^{2r+1}(p_1), \mathrm{r}_a^{2r+1}(p_1+p_2)].
\end{align}
In addition, $\kappa$ maps $\mathcal{C}_{[p_1,p_2]}$ into $\mathcal{C}_{[p_1',p_2']}$  if and only if 
$\underline{\kappa}[p_1,p_2]=[\overline{p}_1,\overline{p}_2 + \overline{p}_1]$ transforms $\{h^s[p_1,p_2], \, s \geq 0 \}$ into $\{ h^s[p_1,p_2], \, s \geq 0 \}$. Here we used $[-\overline{p}]=[\overline{p}]$ if $2p\in \Gamma_k$.

We first deal with the cases $k=1$ and $a=3,6$. The classes of $\frac{1}{2}\Gamma_1/\Gamma_1$ are determined by $0,\frac{1}{2}, \frac{1}{2}e^{\frac{\pi i}{3}}, \frac{1}{2}e^{\frac{2\pi i}{3}}$. First, $h[0,0]=[0,0]=\underline{\kappa}[0,0]$ so that the component $L_1=\mathcal{C}_{[0,0]}=C_{[0,0]}$ is fixed by $\kappa$, and is the mapping torus of $r_{a}^{-2}$; these are both diffeomorphic because $r_3^{-2}= (r_6^{-2})^{-1}$. To deal with the remaining  components, we observe:
\begin{align}\label{eq:R}
 [1/2] \xrightarrow{\mathrm{r}_3} 
[e^{\frac{2\pi i}{3}}/2] \xrightarrow{\mathrm{r}_3}
[e^{\frac{\pi i}{3}}/2] \xrightarrow{\mathrm{r}_3}
[1/2], \qquad
[1/2] \xrightarrow{\mathrm{r}_6} 
[e^{\frac{\pi i}{3}}/2] \xrightarrow{\mathrm{r}_6}
[e^{\frac{2\pi i}{3}}/2] \xrightarrow{\mathrm{r}_6}
[1/2].
\end{align}
Therefore, $\mathrm{r}_a^3=\rId$ on $\frac{1}{2}\Gamma_1/\Gamma$. This implies $\mathrm{h}^3[p_1,p_2]=[p_1,p_1+p_2]$ and
$\mathrm{h}^6=\rId$ on $(\frac{1}{2}\Gamma_1/\Gamma)^2$. 
With this information, we distinguish two cases:
\begin{enumerate}
\item If $[p_1]=[0]$ and $[p_2]\neq [0]$, then $\mathrm{h}^3[0,p_2]=[0,p_2]$, and $\mathrm{h}^k[0,p_2]\neq [0,p_2]$ for $k=1,2$. The connected component
$
L_2=\mathcal{C}_{[0,1/2]}
$
is diffeomorphic to a torus  $S^1 \times \C/\Gamma_1$, and it is preserved by $\kappa$ as $\underline{\kappa}[0,p_2]\in \{ [0,p_2], \, 2p_2\in \Gamma_1 \}$.

\item Let $[p_1]\neq [0]$, we claim that $\mathrm{h}^j[p_1,p_2]\neq [p_1,p_2]$ for $1\leq j \leq 5$. For $j\neq 3$, this holds because the first component of $\mathrm{h}^j[p_1,0]$ is $\mathrm{r}_a^j[p_1]\neq [p_1]$ by equation \eqref{eq:R}, and $\mathrm{h}^3[p_1,p_2]=[p_1,p_1 + p_2]\neq [p_1,p_2]$. Indeed, equations \eqref{eq:h} and \eqref{eq:R} ensure:
\begin{align*}
\{ \mathrm{h}^s[1/2,0], \, 0\leq s \leq 5 \}=&\{ [p_1,0], \,  [p_1]\neq 0\} \cup \{ [p_1,p_1], \,  [p_1]\neq 0\}  \},\\
\{ \mathrm{h}^s[1/2, e^{\frac{\pi i}{3}}/2],  \,  0\leq s \leq 5 \}=&\{ [p_1,p_2],  \, [p_1]\neq [p_2], [p_1]\neq 0, [p_2]\neq [0]\}.
\end{align*}
This yields two connected components diffeomorphic to $S^1\times \C/\Gamma_1$, namely $L_3=\mathcal{C}_{[1/2,0]}$, and $L_4=\mathcal{C}_{[1/2,e^{\frac{i\pi}{3}}/2]}$. One can easily check that $\kappa$ preserves these components. 
\end{enumerate}

We finally deal with $k=2$ and $a=4$.  Note that $\frac{1}{2}\Gamma_2/\Gamma_2=\{ [0],[\frac{1}{2}], [\frac{i}{2}], [\frac{1+i}{2}]\}$. In addition, 
\begin{align}\label{eq:R2}
[1/2] \xrightarrow{\mathrm{r}_4} 
[i/2] \xrightarrow{\mathrm{r}_4}
[1/2] ,\qquad
 [1/2 + i/2] \xrightarrow{\mathrm{r}_4} 
[1/2+ i/2].
\end{align}
Therefore, $\mathrm{r}_4^2=\rId$ on $\frac{1}{2}\Gamma_2/\Gamma$, so that $\mathrm{h}^2=\rId$ on  $(\frac{1}{2}\Gamma_2/\Gamma)^2$ by equation \eqref{eq:h}. We also observe:
\begin{align}\label{eq:comb}
[\mathrm{r}_4(p_1)+ 1/2+ i/2]= [p_1] \qquad \mbox{ if } p_1= 1/2, i/2. 
\end{align}
Equations \eqref{eq:R2} and \eqref{eq:comb} allows us to deduce $\mathrm{h}[p_1,p_2]=[p_1,p_2]$ if
$[p_1,p_2]\in \{[0,0],[0,1/2 + i/2],[1/2+i/2,1/2],[1/2 + i/2,i/2]\}$. This yields four components, $\mathcal{C}_{[0,0]}$, $\mathcal{C}_{[0,1/2 + i/2]}$, $\mathcal{C}_{[1/2+i/2,1/2]}$, $\mathcal{C}_{[1/2 + i/2,i/2]}$ diffeomorphic to the mapping torus of $\mathrm{r}_4^2=-\rId$. We now check that the remaining $12$ points are not fixed by $\mathrm{h}$; so that we obtain $6$ connected components  diffeomorphic to $S^1 \times \C/\Gamma_2$.
We distinguish the cases:
 \begin{enumerate}
 \item  If $[p_1]= [1/2],[i/2]$, equation \eqref{eq:R2} ensures that $\mathrm{h}[p_1,p_2]\neq [p_1,p_2]$.
These points determine $4$ connected components 
 $\mathcal{C}_{[1/2,p_2]}= \mathcal{C}_{[i/2,i/2+p_2]}$, with  $[p_2] \in \C/\Gamma_2$.
 
 \item If $[p_1]=[0]$ and $[p_2]= [1/2],[i/2]$, we have $\mathrm{h}[0,p_2]=[0,\mathrm{r}_4(p_2)]\neq [0,p_2]$, yielding $\mathcal{C}_{[0,1/2]}=\mathcal{C}_{[0,i/2]}$.
 \item If 
  $[p_1]=[1/2+i/2]$ and $[p_2]=[0],[1/2+i/2]$, then $\mathrm{h}[p_1,p_2]=[p_1, p_2+ p_1]\neq [p_1,p_2]$, providing us with
 $\mathcal{C}_{[1/2+i/2,0]}=\mathcal{C}_{[1/2+i/2,1/2+i/2]}$.
\end{enumerate}
In this lattice, $\underline{\kappa}[p_1,p_2]=[p_1,p_1+p_2]$. From this expression one can easily check that the components 
 $L_1=\mathcal{C}_{[0,0]},L_2=\mathcal{C}_{[0,1/2+1/2i]}, L_3=\mathcal{C}_{[0,1/2]},L_4= \mathcal{C}_{[1/2+i/2,0]}$ remain fixed under $\kappa$.
 In addition, $\kappa$ swaps $L_5=\mathcal{C}_{[1/2+i/2, 1/2]}$ with $ L_8=\mathcal{C}_{[1/2+i/2, i/2]}$, 
 and  $L_6=\mathcal{C}_{[1/2,0]}$ with $L_9=\mathcal{C}_{[1/2,1/2]}$, and $L_7=\mathcal{C}_{[1/2,i/2]}$ with $L_{10}=\mathcal{C}_{[1/2,1/2+i/2]}$.
 \end{proof}
 
\begin{remark} \label{rem:L_1}
 In all the cases, $L_1$ is preserved under $\iota$. Indeed, $H^1(L_1,\R)=\la [dt] \ra\oplus K^1$ according to \cite[Lemma 12]{BFM}. In both cases, $K^1=\{0\}$; so that $H^1(L_1)^\kappa=\{0\}$.  Therefore, $L_1$ does not have a nowhere-vanishing closed $1$-form. When resolving $X_{k,a}^2$, we will therefore desingularize
  $M_{k,a}/\iota$ first.
\end{remark}

We now describe the singular locus of $Z_{2,4}$.
It is clear that $\Fix(\gamma)=\emptyset$ if $\gamma \in \la  \jota_1, \jota_2 \ra$. We also have $\Fix(\kappa \circ \jota_1 \circ \j_2)= \emptyset$, as
 $\zeta_1\circ \zeta_2 \circ \eta_2=[-\bar{z}_1, \bar{z}_1+\bar{z}_2,-\bar{z}_3+1/2+i/{2}]$ and $\zeta_1\circ \zeta_2 \circ \eta_2 \circ F_{2,4}=[i\bar{z}_1, -i\bar{z}_2,\bar{z}_3+1/2 +i/2]$. These maps do not fix any point, as they involve the translations $y_3 \mapsto y_3 + 1/2$ and $x_3\mapsto x_3+1/2$, respectively.
The singular locus is then determined by $\Fix(\kappa)$, $ \Fix(\jota_1 \circ \kappa)$ and $\Fix(\jota_2 \circ \kappa)$. The first set is deduced from equations \eqref{eqn:M24-half-1} and \eqref{eqn:M24-zero-1}; the remaining are determined by:
 \begin{align}
\Fix(\zeta_1 \circ \eta_2 )=& \{ [-2x_2, x_2 + i y_2, \e_3 + iy_3 ], \quad \e_3 \in \{-1/4,1/4\} , \quad 0 \leq  x_2, y_2, y_3 \leq 1 \}, \label{eqn:Z24-4}\\
\Fix(\zeta_2 \circ \eta_2  \circ F_{2,4})=&\{ [x_1 - i x_1, x_2 + i x_2, x_3 + i \e_3 ], \quad \e_3 \in \{ -1/4,1/4 \}, \quad 0 \leq x_1, x_2, x_3 \leq 1 \},\label{eqn:Z24-2}\\
\Fix(\zeta_1 \circ \eta_2  \circ F_{2,4})=&\Fix(\zeta_2\circ \eta_2 )=\emptyset.
\end{align}
Again, $\zeta_1 \circ \eta_2  \circ F_{2,4}$ and $\zeta_2\circ \eta_2$ involve translations in the variables $x_3$ and $y_3$, respectively.  We denote the components of the fixed locus at $t=0$ by $N_1^1,N_1^2,N_2^1,N_2^2$, where $N_1^k \subset \Fix(\kappa)$ and the upper index $k=1$ corresponds to $\e_3=0$ or $\e_3=-1/4$. The components at $t=1/2$ are denoted by $N_3^1,N_3^2,N_4^1,N_4^2$, where $N_3^k \subset \Fix(\kappa)$ and we use the same convention for the upper index $k$. All these are diffeomorphic to a torus $T^3$.  The involution
 $\jota_1$ swaps $N_3^1$ with $N_3^2$ and $N_4^1$ with $N_4^2$, and it
preserves every connected component at $t=0$. Similarly, $\jota_2$ interchanges $N_j^1$ with $N_j^2$ for $j=1,2$ and preserves $N_j^k$ where $j=3,4$. Hence, the singular locus of $Z_{2,4}$ consists of $4$ components that we denote by $N_1,N_2,N_3,N_4$. We have a $2:1$ cover $N_j^1 \to N_j$; indeed $H^1(N_1)=H^1(N_2)=\la [dy_3] \ra$ and $H^1(N_3)=H^1(N_4)=\la [dx_3] \ra$.

 We observe that 
the projection $M_{2,4} \to M_{2,4}/\la \jota_1,\jota_2 \ra$ induces a branched covering. $X_{2,4}^1 \to Z_{2,4}$. The ramification locus is the projection of $\Fix(\jota_1\circ \kappa) \cup \Fix(\jota_2 \circ \kappa) $ to $X_{2,4}$.  It is also worth pointing out that the map $F_{2,4}$ descends to $T^6/\langle \zeta_1,\zeta_2 \rangle$. Indeed, 
$M_{2,4}/\langle \jota_1, \jota_2\rangle$ is the mapping torus of $F_{2,4}\colon T^6/\langle \zeta_1,\zeta_2 \rangle \to T^6/\langle \zeta_1,\zeta_2\rangle$.

We finally compute the singular locus of $X_{3,4}^1$ and $Z_{3,4}$. The singular locus of $X_{3,4}^1$ is the projection of $\Fix(\kappa)\subset M_{3,4}$ and it is determined by 
\begin{align*}
\Fix(\eta_3)=& \sqcup_{n=0}^1 \{ [n-2iy_2, x_2  + iy_2, \e_3 + iy_3 ], \quad \e_3 \in \{0,1/2\} , \quad 0 \leq  x_2,y_2, y_3 \leq 1\},\\
\Fix(\eta_3 \circ F_{3,4})=& \sqcup_{n=0}^1 \{ [x_1 + i (x_1+n), x_2 - i x_2, x_3 + i \e_3 ], \quad \e_3 \in \{ 0,1/2 \}, \quad 0 \leq x_1, x_2, x_3 \leq 1\}.
\end{align*}
Therefore, there are $8$ connected components, diffeomorphic to a $3$-torus. We denote by $N_1,N_2,N_3,N_4$ the components at $t=0$ and $N_5,N_6,N_7,N_8$ those at $t=1/2$.
Whereas for $Z_{3,4}$, similar computations show that it is determined by the projection $\Fix(\kappa)\cup \Fix(\jota_1 \circ \kappa)\cup \Fix(\jota_2 \circ \kappa)\subset M_{3,4}$. The sets $\Fix(\jota_1 \circ \kappa)$, $\Fix(\jota_2 \circ \kappa)$ are induced by
 \begin{align*}
\Fix(\zeta_1 \circ \eta_3 )=& \sqcup_{n=0}^1\{ [-2x_2 + in, x_2 + i y_2, \e_3 + iy_3 ], \quad \e_3 \in \{-1/4,1/4\} , \quad 0 \leq  x_2, y_2, y_3 \leq 1\}, \\
\Fix(\zeta_2 \circ \eta_3  \circ F_{2,4})=&\sqcup_{n=0}^1\{ [x_1 - i (x_1+ n), x_2 + i x_2, x_3 + i \e_3 ], \quad \e_3 \in \{ -1/4,1/4 \}, \quad 0 \leq x_1, x_2, x_3 \leq 1 \},\\
\Fix(\zeta_1 \circ \eta_3  \circ F_{2,4})=&\Fix(\zeta_2\circ \eta_2 )=\emptyset.
\end{align*}
A similar argument shows that the singular locus of $Z_{3,4}$ consists of eight connected components, each of which is double-covered by  $T^3$. We denote by $N_1,N_2,N_3,N_4$ the components at $t=0$ and by $N_5,N_6,N_7,N_8$ those at $t=1$.

\begin{lemma} \label{lem:double-cov}
There are double covers $X_{3,4}^1 \to X_{2,4}^1$ and $Z_{3,4} \to Z_{2,4}$.
\end{lemma}
\begin{proof}
Consider the map
$
\mathrm{T} \colon M_{3,4}\to M_{3,4}$, $ \mathrm{T}  [t,[z_1,z_2,z_3]]=[t,[z_1+1,z_2,z_3]],
$
and note that $M_{3,4}/\mathrm{T} =M_{2,4}$ and the action of $\mathrm{T} $ is free on $M_{3,4}$.
There is an induced $\Z_2$ action $\mathrm{T}  \colon X_{3,4}^1\to X_{3,4}^1$ because
\begin{equation}\label{eqn:commutej-k}
(\mathrm{T}  \circ \kappa)[t,[z]]=[1-t,[-\bar{z}_1+1,\bar{z}_2+\bar{z}_1,-\bar{z}_3]] = [1-t,[-\bar{z}_1-1,\bar{z}_2+\bar{z}_1 +1,-\bar{z}_3]]= (\kappa \circ \mathrm{T} )[t,[z]].
\end{equation}
One checks that $\mathrm{T}  \circ \jota_p=\jota_p \circ \mathrm{T} $ in a similar way. Therefore, $\mathrm{T}  \colon Z_{3,4} \to Z_{3,4}$ is a well-defined $\Z_2$-action, and we have $X_{2,4}^1=X_{3,4}^1/\mathrm{T} $, $Z_{2,4}=Z_{3,4}/\mathrm{T} $. In addition, from equation \eqref{eqn:commutej-k} we deduce that the action of $\mathrm{T} $ is free outside the singular locus. From the computations of $\Fix(\kappa)$ and $\Fix(\kappa \circ \jota_p)$ on $M_{3,4}$, we deduce that $\mathrm{T}$ is also free on the singular locus.
\end{proof}

\subsubsection{Resolutions}

Theorem 3.11 in \cite{LMM} yields resolutions $\widetilde{X}^1_{k,a} \to X^1_{k,a}$ $k=1$, $a=3,6$ because the singular locus admits a nowhere-vanishing closed $1$-form. In addition, Lemma \ref{theo:resol} allows us to find resolutions 
$\widetilde{X}^1_{3,4} \to X_{3,4}^1$, $\widetilde{Z}_{3,4} \to Z_{3,4}$, and lifts  $\widetilde{T}\colon \widetilde{X}_{3,4}^1 \to \widetilde{X}_{3,4}^1$, $\widetilde{T}\colon \widetilde{Z}_{3,4} \to \widetilde{Z}_{3,4}$ of the map induced by $
\mathrm{T} \colon M_{3,4}\to M_{3,4}$, $ \mathrm{T}  [t,[z_1,z_2,z_3]]=[t,[z_1+1,z_2,z_3]]
$ defined in the proof of Lemma \eqref{lem:double-cov}. This result implies that $\widetilde{X}_{2,4}^1=\widetilde{X}_{3,4}^1/\widetilde{T}$ and $\widetilde{Z}_{2,4}=\widetilde{Z}_{3,4}/\widetilde{T}$ are the closed $\Gtwo$ resolutions of $X_{2,4}^1$ and $Z_{2,4}$. Of course, $\widetilde{X}_{3,4}^1 \to \widetilde{X}_{2,4}^1$ and $\widetilde{Z}_{3,4} \to \widetilde{Z}_{2,4}$  are double covers.

To deal with $X^{2}_{k,a}$, we first desingularize the closed $\Gtwo$ orbifold $(Y_{k,a}=M_{k,a}/\iota,\varphi)$. 
The nowhere-vanishing $1$-form we use is $\theta=dt|_{\Fix(\iota)}$ on $M_{1,a}$, and $\theta|_{L_j}=dt|_{L_j}$ if $1\leq j\leq 7$ and $\theta|_{L_j}=-dt|_{L_j}$ $j= 8,9,10$ on $M_{2,4}$. It satisfies  $\kappa^*\theta=-\theta$ on $\mathcal{L}_F$ and $\kappa^*\theta=\theta$ on $\mathcal{L}_s$ . Theorem \ref{theo:resol} allows us to obtain 
a closed $\Gtwo$ resolution $\rho_1 \colon \widetilde{Y}_{k,a}\to Y_{k,a}$ and a lift $\tilde{\kappa} \colon \widetilde{Y}_{k,a} \to \widetilde{Y}_{k,a}$.

\begin{remark}\label{rem:K3}
Let $S_k$ be the algebraic resolution of $T^4_k=(\C/\Gamma_k)^2/\Z_2$, which is a $K3$ surface. The biholomorphism $T_k^4$ given by $[z_1,z_2]\to [\mathrm{r}_a(z_1),\mathrm{r}_a(z_2-z_1)]$ lifts to a biholomorphism $\widetilde{f}_{k,a} \colon S_k \to S_k$. It turns out that $\widetilde{Y}_{k,a}$ is diffeomorphic to the mapping torus of $\widetilde{F}_{k,a}= (\widetilde{f}_{k,a}, \mathrm{r}_{a}^{-2})\colon S_k \times \C/\Gamma_k \to S_k \times \C/\Gamma_k$.
This occurs because the resolution procedure consists of blowing up the coordinates $[z_1,z_2]$ around the points $[p_1,p_2]$ with $2p_1,2p_2 \in \Gamma_k$, as $\iota[t,z]=[t,-z_1,-z_2,z_3]$.
\end{remark}

\begin{lemma} \label{lem:sing-2}
All connected components of the fixed locus of $\tilde{\kappa}$ in $\widetilde{Y}_{k,a}$ are diffeomorphic to $\Sigma \times S^1$, where $\Sigma$ is a surface of genus $3$. 
 The number of connected components of $\Fix(\tilde{\kappa})$ is $2$ if $k=1$, and $4$ if $k=2$.
\end{lemma}
\begin{proof}
Since $\rho \circ \kappa=\tilde{\kappa} \circ \rho$, it follows that $\Fix(\tilde{\kappa})$ is the strict transform of $(\Fix(\iota \circ \kappa)\cup \Fix(\kappa))/\iota $. The connected components of this set are diffeomorphic to $\Sigma^\vee \times S^1$, where $\Sigma^\vee$ is a union of two pillowcase orbifolds that intersect each other at their singular points. In addition, $\Fix(\kappa)\cap \Fix(\iota)$ is the union of the products of circles with the singular points of $\Sigma^\vee$. 
Given one such component, $\{[\e,p_1,p_2]\}\times S^1 \subset \Fix(\iota)\cap \Fix(\kappa)$, where $\e\in \{0,1/2\}$, and $ 2 p_k\in \Gamma_k$,
 the coordinates $[z_1,z_2]$ of $\Fix(\iota \circ \kappa)\cup \Fix(\kappa)$ form two transversal real $2$-planes in $\C^2$. Indeed, from the formulas in section \ref{subsec:sing-locus}, we can check that one of these real two planes can be obtained from the other by multiplying by $i$.
 Since the strict transform of those planes after a blow-up in $\C^2/\Z_2$ is $\RP^1$, in the resolution we replace $\{[\e,p_1,p_2]\}\times S^1$ with $\RP^1\times S^1$. Therefore, the strict transform of $S^1 \times \Sigma^\vee$ 
 is $S^1 \times \Sigma$, where $\Sigma$ a union of two
spheres, each of these punctured at four discs, and connected through cylinders. Hence, $\Sigma$ is a surface of genus $3$.
\end{proof}

We denote the connected components of $\Fix(\widetilde{\kappa})$ on $\widetilde{Y}_{1,a}$ at $t=0$ and $t=1/2$ by $K_1$ and $K_2$ respectively. For $\widetilde{Y}_{2,4}$, we label the connected components at $t=0$ by
 $K_1,K_2$, and those at $t=1/2$ by $K_3, K_4$.

Lemma \ref{lem:sing-2}, and Theorem \ref{theo:resol} guarantee that the hypotheses of \cite[Theorem 3.11]{LMM} hold for $\widetilde{Y}_{k,a}/\tilde{\kappa}$. Hence, there is a closed $\Gtwo$ resolution $\rho_2 \colon \widetilde{X}_{k,a}^2 \to \widetilde{Y}_{k,a}/\widetilde{\kappa}$. Indeed, let $\bar{\rho}_1\colon \widetilde{Y}_{k,a}/\tilde{\kappa} \to X^2_{k,a}=Y_{k,a}/\kappa$ be the map induced by $\rho_1$. Then the resolution map is $\bar{\rho}_1 \circ \rho_2$ and we
 have a commutative diagram:
\begin{equation}\label{diag:resol-2}
\begin{tikzcd}
\widetilde{Y}_{k,a} \arrow{r}{} \arrow{d}{\rho_1} 
& \widetilde{Y}_{k,a}/\tilde{\kappa} \arrow{d}{\bar{\rho}_1} & \arrow{l}{\rho_2} \widetilde{X}_{k,a}^2 \\
Y_{k,a} \arrow{r}{} & X^2_{k,a}, &
\end{tikzcd}
\end{equation}
where the horizontal arrows on the left are the quotient maps.

\subsection{Topological properties of the constructed manifolds} \label{sec:top}

The proof of Theorem \ref{theo:1} is contained in Propositions \ref{prop:fund-group-1}, \ref{prop:fund-group-2}, \ref{prop:square-of-a-2-form}, Corollary \ref{cor:pont-examples}, the discussion at the end of section \ref{subsec:properties}, and Theorem \ref{prop:non-formal-non-holonomy}.
Excluding the fundamental group, the properties listed in table \ref{table:1} are shown in Theorems \ref{theo:cohom-ex-1}, \ref{theo:cohom-ex-2}, Propositions \ref{prop:formality-1}, \ref{prop:formality-2}, \ref{prop:formality-3} and \ref{prop:formality-4}.
Along the way, we provide explicit formulas for the cohomology groups and the first Pontryagin class of the constructed examples.

\subsubsection{Fundamental group}

Making use of Seifert-Van Kampen theorem we obtain $\pi_1(\widetilde{X}^1_{k,a})\cong \pi_1(X^1_{k,a})$, $\pi_1(\widetilde{Z}_{k,4})\cong \pi_1({Z}_{k,4})$ and $\pi_1(\widetilde{X}^2_{k,a})\cong \pi_1(\widetilde{Y}_{k,a}/\tilde{\kappa})$.  To compute the fundamental group of an orbifold $M/\Z_2$ we use \cite[Chapter 2, Corollary 6.3]{Bre}, which implies that if an involution $\kappa \colon M \to M$ has a fixed point, then 
the map
$\mathrm{pr}_*\colon \pi_1(M) \to \pi_1(M/\kappa)$ is surjective. For this calculation, we will frequently use that if a loop $\gamma  \colon [0,1]\to M$ satisfies $\kappa \circ \gamma(s)= \gamma^{-1}(s)$ for every $s\in [0,1]$, then there is a base-point preserving homotopy from $q\circ \gamma$ to the constant loop. The reason is that $\mathrm{pr}\circ \gamma(s)=\mathrm{pr} \circ \bar{\gamma}(s)$, where
$$ 
\bar{\gamma}\colon [0,1]\to M, \, \bar{\gamma}(s)=\begin{cases}
\gamma(s), & s \leq 1/2, \\
\gamma(1-s), & s \geq 1/2.
\end{cases}
$$
and $\bar{\gamma}$ is homotopic to the a constant loop. The homotopy $H_r(s)$ between these consists of two parts: we follow $\gamma$ from $\gamma(0)$ to $\gamma(r/2)$ and then reverse that path.  Finally, we compute the fundamental group of $M_{k,a}$, and $\widetilde{Y}_{k,a}$ using equation \eqref{eqn:fund-map-torus} in section \ref{subsec:map-review} and Remark \ref{rem:K3}.  To describe these, we introduce some notations. 
Let $\mathrm{i}_p \colon \CC \to \CC^3$ be the inclusion to the $p$-th component,  $p=1,2,3$. A set of generators of $\pi_1(T^6_k,[0])$ is determined by the homotopy classes of the loops $\gamma_{p,k}$:
\begin{align} \label{eqn:generators-t1}
&\gamma_{p,1},\gamma_{p,2} \colon [0,1] \to T^6_1, \quad \gamma_{p,1}=[\mathrm{i}_p(t)], \quad \gamma_{p,2}=[\mathrm{i}_p(t e^{\frac{2\pi i}{3}})],  &p=1,2,3, \\
 \label{eqn:generators-t2}
&\gamma_{p,1},\gamma_{p,2} \colon [0,1] \to T^6_2, \quad \gamma_{p,1}=[\mathrm{i}_p(t)], \quad \gamma_{p,2}=[\mathrm{i}_p(ti)],  &p=1,2,3.
\end{align}
Taking into account Remark \ref{rem:K3}, that the $K3$ surface $S_k$ is simply connected, and abusing slightly of the notation, we say that $[\gamma_{3,1}],[\gamma_{3,2}]$ generate $\pi_1(\C/\Gamma_k)=\pi_1(S_k \x \C/\Gamma_k)$.
Given a loop $\rho \colon [0.1]\to T^6_k$ or $\rho \colon [0.1]\to S_k \times \C/\Gamma_k$ we denote by $\tilde{\rho}$
its inclusion to the mapping torus at $t=\frac{1}{2}$; that is, $\tilde{\rho}(t)=[\frac{1}{2},\rho]$.  We consider $\tilde{\gamma}_0 $ the loop on $M_{k,a}$ or $Y_{k,a}$ given by equation \eqref{eqn:gamma_0}. In particular, $\pi_1(M_{k,a})$ is generated by $\tilde{\gamma}_0 $ and $\tilde{\gamma}_{p,q}$ and the generators are subject to the relations:
\begin{equation} \label{eqn:relations-mapping torus}
    [\tilde{\rho}]= [\tilde\gamma_0](F_*[\tilde\rho])[\tilde\gamma_0]^{-1}, \qquad 
    [\rho]\in \pi_1(T^6_k).
\end{equation}
By an abuse of notation, we denote the projection maps by $\mathrm{pr}\colon M_{k,a}\to X^j_{k,a}$, $\mathrm{pr}\colon \widetilde{Y}^j_{k,a}\to \widetilde{Y}^j_{k,a}/\tilde{\kappa}$. We first deal with the orbifolds obtained from $M_{1,3}$ and $M_{1,6}$:

\begin{proposition}\label{prop:fund-group-1}
The manifolds $\widetilde{X}^1_{1,6}$,$\widetilde{X}^2_{1,3}$ ,$\widetilde{X}^2_{1,6}$ are simply connected, and 
$
\pi_1(\widetilde{X}^1_{1,3})=\Z_3$.
\end{proposition}
\begin{proof}
We first compute $\pi_1(M_{1,3})$. For that, we analyze the map $(F_{1,3})_* \colon \pi_1(T^6_1 ) \to \pi_1(T^6_1)$ in terms of the generators $\gamma_{p,1}$ described in equation \eqref{eqn:generators-t1}:
\begin{align*}
(F_{1,3})_*[\gamma_{p,1}]=&[\gamma_{p,2}], \qquad\qquad (F_{1,3})_*[\gamma_{p,2}]=-[\gamma_{p,1}]-[\gamma_{p,2}], \quad p \neq 1, \\
(F_{1,3})_*[\gamma_{1,1}]=& [\gamma_{1,2}]- [\gamma_{2,2}], \quad (F_{1,3})_*[\gamma_{1,2}]=-[\gamma_{1,1}]-[\gamma_{1,2}]+[\gamma_{2,1}]+ [\gamma_{2,2}]. 
\end{align*}
From the relations \eqref{eqn:relations-mapping torus} we obtain:
\begin{align} \label{eqn:rel-fund-1}
[\tilde{\g}_{p,1}]=&[\tilde{\g}_0][ \tilde{\g}_{p,2}] [\tilde{\g}_0]^{-1},&[\tilde{\g}_{p,2}]=&[\tilde{\g}_0][\tilde{\g}_{p,1}]^{-1}[\tilde{\g}_{p,2}]^{-1} [\tilde{\g}_0]^{-1}, \qquad   
p \neq 1,\\\label{eqn:rel-fund-2}
 [\tilde{\g}_{1,1}]=&[\tilde{\g}_0][ \tilde{\g}_{1,2}][ \tilde{\g}_{2,2}]^{-1} [\tilde{\g}_0]^{-1}, &
[\tilde{\g}_{1,2}]=&[\tilde{\g}_0][\tilde{\g}_{1,1}]^{-1}[\tilde{\g}_{1,2}]^{-1} [\tilde{\g}_{2,1}] [\tilde{\g}_{2,2}] [\tilde{\g}_0]^{-1}.
\end{align}
We now claim that 
$
H_1(M_{1,3})= \Z \oplus \Z_3 \oplus \Z_3 \oplus \Z_3
$.
 To verify this, consider $c_0, c_{p,q} \in H^1(M_{1,3})$ the homology classes determined by $[\tilde{\gamma}_{0}]$ and $[\tilde{\gamma}_{p,q}]$ respectively. If $p\neq 1$, equality \eqref{eqn:rel-fund-1} implies $c_{p,2}=c_{p,1}$ and $2c_{p,2}=-c_{p,1}$; therefore, $c_{p,1}$ has order $3$. If $p=1$, equality \eqref{eqn:rel-fund-2} shows $c_{1,1}=c_{1,2}-c_{2,2}$ and $2c_{1,2}= -c_{1,1}  +2 c_{2,2}$; therefore $c_{1,1}$ and $c_{1,2}$ have order $3$. This implies the claimed equality, and the genertors are $c_0$ and $c_{p,1}$, $p=1,2,3$.

Observe that $\kappa \circ \tilde{\gamma}_0(s)=\tilde{\gamma}_0^{-1}(s)$, and $\kappa \circ \tilde{\gamma}_{p,1}(s)=[\frac{1}{2}, \eta \circ \gamma_{p,1}(s)]$. From the second expression we deduce $\kappa \circ \tilde{\gamma}_{3,1}(s)=\tilde{\gamma}_{3,1}^{-1}(s)$. Using the observation at the beginning of the section we obtain $\mathrm{pr}_*[\tilde{\g}_0]=1$ and $\mathrm{pr}_*[\tilde{\g}_{3,1}]= 1$. Equality \eqref{eqn:rel-fund-1} yields:
$$
\mathrm{pr}_*[\tilde{\g}_{2,1}]=\mathrm{pr}_*[\tilde{\g}_{2,2}] , \qquad 
\mathrm{pr}_*[\tilde{\g}_{2,2}]^{2}=\mathrm{pr}_*[\tilde{\g}_{2,1}]^{-1}, \qquad 
\mathrm{pr}_*[\tilde{\g}_{3,2}]= 1.
$$
 In particular, the order of $\mathrm{pr}_*[\tilde{\g}_{2,1}]$ is either zero or $3$.
 In addition, on $M_{1,3}$ we also have  $\kappa_*[\tilde{\gamma}_{1,1}]= [\gamma_{1,1}]^{-1}[\gamma_{2,1}]$; hence
$
\mathrm{pr}_*[\tilde{\g}_{1,1}]^2 =\mathrm{pr}_*[\tilde{\g}_{2,1}]
$ and the order of $\mathrm{pr}_*[\tilde{\g}_{1,1}]$ coincides with that of $\mathrm{pr}_*[\tilde{\g}_{2,1}]$. Equation \eqref{eqn:rel-fund-2} implies $\mathrm{pr}_*[\tilde{\g}_{1,2}]
=\mathrm{pr}_*[\tilde{\g}_{1,1}]\mathrm{pr}_*[\tilde{\g}_{2,2}]=\mathrm{pr}_*[\tilde{\g}_{1,1}]^3=1 $. Hence, there is a surjection
 $\Z_3 \to
\pi_1({X}^1_{1,3})
$ that maps $1$ to $\mathrm{pr}_*[\tilde{\g}_{1,1}]$. In particular, $\pi_1({X}^1_{1,3})$ is either $\{1\}$ or $\Z_3$ and $\pi_1(X_{1,3})=H_1(X_{1,3},\Z)$.

 To prove that $\pi_1({X}^1_{1,3})=\Z_3$ it suffices to check that  $H_1(X_{1,3}^1,\Z_3)\neq \{0\}$ because the 
universal coefficient theorem for homology ensures that  $H_1(X_{1,3}^1,\Z_3)=H_1(X_{1,3}^1,\Z)\otimes_{\Z} \Z_3$, and the right hand side is either $\{0\}$ or $\Z_3$ according to our previous argument. 
In addition, $H_1(X_{1,3}^1,\Z_3)=H_1(M_{1,3},\Z_3)^{\Z_2}$ by \cite[Chapter 3, Theorem 2.4]{Bre}. A new application of the universal coefficient theorem yields  
$
H_1(M_{1,3},\Z_3)=H_1(M_{1,3},\Z)\otimes_{\Z} \Z_3 =(\Z \oplus \Z_3 \oplus \Z_3 \oplus \Z_3)\otimes_{\Z} \Z_3 =\Z_3 \oplus\Z_3 \oplus \Z_3 \oplus \Z_3,
$
this group is spanned by $c_0$ and $c_{p,1}$ for $p=1,2,3$. The action of $\Z_2$ is determined by:
$$
c_0 \mapsto -c_0, \quad c_{1,1}\mapsto -c_{1,1}+ c_{2,1}, \quad c_{2,1}\mapsto c_{2,1}, \quad c_{3,1} \mapsto -c_{3,1}.
$$
Hence, $c_{2,1}\in H_1(M_{1,3},\Z_3)^{\Z_2}$ and therefore,
$\pi_1(X_{1,3})=\Z_3$. We now show that $X_{1,6}$ is simply connected. First, the generators $[\tilde{\g}_0]$, $[\tilde{\g}_{p,q}]$, of $\pi_1(M^1_{1,6})$ satisfy the relations:
\begin{align} 
\label{eqn:rel-fund-3}
[\tilde{\g}_{1,1}]=& [\tilde{\g}_0][ \tilde{\g}_{1,1}] [ \tilde{\g}_{1,2}] [ \tilde{\g}_{2,1}]^{-1}[ \tilde{\g}_{2,2}]^{-1} [\tilde{\g}_0]^{-1},\qquad
&[\tilde{\g}_{1,2}]=&[\tilde{\g}_0][\tilde{\g}_{1,1}]^{-1}[\tilde{\g}_{2,1}][\tilde{\g}_0]^{-1}, \\
\label{eqn:rel-fund-4}
[\tilde{\g}_{2,1}]=&[\tilde{\g}_0][ \tilde{\g}_{2,1}][ \tilde{\g}_{2,2}] [\tilde{\g}_0]^{-1},\qquad
&[\tilde{\g}_{2,2}]=&[\tilde{\g}_0][\tilde{\g}_{2,1}]^{-1}[\tilde{\g}_0]^{-1}, \\
\label{eqn:rel-fund-5}
[\tilde{\g}_{3,1}]=&[\tilde{\g}_0][\tilde{\g}_{3,1}]^{-1}[ \tilde{\g}_{3,2}]^{-1}[\tilde{\g}_0]^{-1},\qquad
&[\tilde{\g}_{3,2}]=&[\tilde{\g}_0][ \tilde{\g}_{3,1}] [\tilde{\g}_0]^{-1}.
\end{align}
We again have $\kappa \circ \tilde{\gamma}_0(s)=\tilde{\gamma}_0^{-1}(s)$, and $\kappa \circ \tilde{\gamma}_{p,1}(s)=[\frac{1}{2}, \eta \circ \gamma_{p,1}(s)]$. In particular, $\kappa \circ \tilde{\gamma}_{3,1}(s)=\tilde{\gamma}_{3,1}^{-1}(s)$. Therefore $\mathrm{pr}_*[\tilde{\g}_0]=1$ and $\mathrm{pr}_*[\tilde{\g}_{3,1}]= 1$. Equality  \eqref{eqn:rel-fund-5} yields $\mathrm{pr}_*[\tilde{\g}_{3,2}]=1 $ and  \eqref{eqn:rel-fund-4} implies $\mathrm{pr}_*[\tilde{\g}_{2,2}]=\mathrm{pr}_*[\tilde{\g}_{2,1}]=1$. Taking these equalities into account together with equation \eqref{eqn:rel-fund-3}, we obtain $\mathrm{pr}_*[\tilde{\g}_{1,2}]=\mathrm{pr}_*[\tilde{\g}_{1,1}]=1$. 
Hence $\pi_1(X_{1,6})=\{1\}$.

If $a=3,6$, then $\pi_1(\widetilde{Y}_{1,k})$ is generated by $[\tilde{\g}_0]$, 
$[\tilde{\gamma}_{3,1}]$ and $[\tilde{\gamma}_{3,2}]$. They satisfy the relations:
$$
\begin{cases}
\begin{aligned}
[\tilde{\g}_{3,1}] &= [\tilde{\g}_0][\tilde{\g}_{3,2}][\tilde{\g}_0]^{-1}, &
[\tilde{\g}_{3,2}] &= [\tilde{\g}_0][\tilde{\g}_{3,1}]^{-1}[\tilde{\g}_{3,2}]^{-1}[\tilde{\g}_0]^{-1}, & k=3,\\
[\tilde{\g}_{3,1}] &= [\tilde{\g}_0][\tilde{\g}_{3,1}]^{-1}[\tilde{\g}_{3,2}]^{-1}[\tilde{\g}_0]^{-1}, &
[\tilde{\g}_{3,2}] &= [\tilde{\g}_0][\tilde{\g}_{3,1}][\tilde{\g}_0]^{-1}, & k=6.
\end{aligned}
\end{cases}
$$
In both cases, $\tilde{\kappa} \circ \tilde{\gamma}_0=\tilde{\g}_0^{-1}$ and $\tilde{\kappa}  \circ \tilde{\gamma}_{3,1}=\tilde{\gamma}_{3,1}^{-1}$. Hence, $\mathrm{pr}_*[\tilde{\gamma}_{0}]=\mathrm{pr}_*[\tilde{\gamma}_{3,1}]=1$, which implies $\mathrm{pr}_*[\tilde{\gamma}_{3,2}]=1$. Thus, $\pi_1(\widetilde{Y}_{k,3}/\tilde{\kappa})=\{1\}$.
\end{proof}

We now focus on the orbifolds obtained from ${M}_{2,4}$ and ${M}_{3,4}$.

\begin{proposition} \label{prop:fund-group-2}
The manifolds $\widetilde{X}^1_{3,4}$, $\widetilde{Z}_{3,4}$, $\widetilde{X}^2_{2,4}$, $\widetilde{Z}_{2,4}$
are simply connected and $\pi_1(\widetilde{X}^1_{2,4})=\pi_1(\widetilde{Z}_{2,4})=\Z_2$. In particular, $\widetilde{X}^1_{3,4}$ and $\widetilde{Z}_{3,4}$ are, respectively, the universal cover of $\widetilde{X}^1_{2,4}$ and $\widetilde{Z}_{2,4}$.
\end{proposition}
\begin{proof}
For $\pi_1({M}_{2,4})$, we deduce the following relations for the generators $[\tilde{\g}_0]$, $[\tilde{\g}_{p,q}]$  :
 \begin{align}
 \label{eqn:rel-fund-6}
[\tilde{\g}_{1,1}]=&[\tilde{\g}_0][ \tilde{\g}_{1,2}][ \tilde{\g}_{2,2}]^{-1}[\tilde{\g}_0]^{-1},
&[\tilde{\g}_{1,2}]=&[\tilde{\g}_0][\tilde{\g}_{1,1}]^{-1}[ \tilde{\g}_{2,1}][\tilde{\g}_0]^{-1}, \\
\label{eqn:rel-fund-7}
[\tilde{\g}_{2,1}]=&[\tilde{\g}_0][ \tilde{\g}_{2,2}] [\tilde{\g}_0]^{-1},
&[\tilde{\g}_{2,2}]=&[\tilde{\g}_0][\tilde{\g}_{2,1}]^{-1}[\tilde{\g}_0]^{-1}, \\
\label{eqn:rel-fund-8}
[\tilde{\g}_{3,1}]=&[\tilde{\g}_0][ \tilde{\g}_{3,1}]^{-1} [\tilde{\g}_0]^{-1},
&[\tilde{\g}_{3,2}]=&[\tilde{\g}_0][ \tilde{\g}_{3,2}]^{-1} [\tilde{\g}_0]^{-1}.
\end{align}
Similar to the proof of Proposition \ref{prop:fund-group-1}, we obtain the following relations on $X_{2,4}^1$:
\begin{equation} \label{eqn:rel-orbi-2}
 \mathrm{pr}_*[\tilde{\g}_0]= 1, \quad  \quad \mathrm{pr}_*[\tilde{\g}_{2,2}]= 1 , \quad \mathrm{pr}_*[\tilde{\g}_{3,1}]= 1.
 \end{equation}
 Combining them with equation \eqref{eqn:rel-fund-6}, \eqref{eqn:rel-fund-7} and \eqref{eqn:rel-fund-8},  we find that the projections of the generators to $\pi_1(X_{2,4}^1)$  are trivial except possibly for $\mathrm{pr}_*[\tilde{\g}_{1,1}]=\mathrm{pr}_*[\tilde{\g}_{1,2}]$ and $\mathrm{pr}_*[\tilde{\g}_{3,2}]$, whose order is at most $2$. We now show that $\mathrm{pr}_*[\tilde{\g}_{3,2}]=1$. On $M_{2,4}$, the loop $\gamma_{3,2}$ is homotopic to $\gamma_0' \cdot \gamma_{3,2}' \cdot (\gamma_{0}')^{-1}$, where
$$
\gamma_0' , \gamma_{3,2}'\colon [0,1]\to M_{2,4}, \qquad \gamma_0'(t)=[(1+t)/2,[0,0,0]], \quad \gamma_{3,2}'(t)=[1,[0,0,it]].
$$
Since $\kappa\circ \gamma_{3,2}'(t)= [0,[0,0,it]]=[1,[0,0,-it]]=(\gamma_{2,3}')^{-1}(t)$,
there is a base-point preserving homotopy from the projection of $[\gamma_{3,2}']$ to the constant loop. Hence,  $\mathrm{pr}_*[\gamma_{3,2}]=\mathrm{pr}_*[\gamma_0' \cdot \gamma_{3,2}' \cdot (\gamma_{0}')^{-1}]=1$. Then $\pi_1(X_{2,4}^1)$ is either $\{1\}$ or $\Z_2$. Lemma \ref{lem:double-cov} guarantees that  $\pi_1(X_{3,4}^1)=\{1\}$ and $\pi_1(X_{2,4}^1)=\Z_2$.

The generators of $\pi_1(\widetilde{Y}_{2,4})$ are $[\tilde{\gamma}_0]$,
$[\tilde{\gamma}_{3,1}],[\tilde{\gamma}_{3,2}]$, and they satisfy
$
[\tilde{\g}_{3,p}]=[\tilde{\g}_0][ \tilde{\g}_{3,p}]^{-1}[\tilde{\g}_0]^{-1},
$
for
$ p=1,2
$. From $\tilde{\kappa} \circ \tilde{\gamma}_0=\tilde{\g}_0^{-1}$ and $\tilde{\kappa}  \circ \tilde{\gamma}_{3,1}=\tilde{\gamma}_{3,1}^{-1}$, it follows that $\mathrm{pr}_*[\tilde{\gamma}_{0}]=\mathrm{pr}_*[\tilde{\gamma}_{3,1}]=1$. One obtains that $\mathrm{pr}_*[\tilde{\gamma}_{3,2}]=1$ arguing as in the case of $X_{2,4}^1$. Hence, $\pi_1(Y_{2,4}/\tilde{\kappa})=\{1\}$.
 
To compute $\pi_1(Z_{2,4})$ we recall that there is a branched cover $\mathrm{q}\colon X_{2,4}^1 \to Z_{2,4}$. In fact,
$Z_{2,4}=(X_{2,4}^1/ {\jota}_1')/ {\jota}_2'$, where ${\jota}_k'$ are induced by $\jota_k$. The induced map $\mathrm{q}_*\colon \pi_1(X_{2,4}^1) \to \pi_1(Z_{2,4})$ is surjective, because both maps $\pi_1(X_{2,4}^1)\to \pi_1(X_{2,4}^1/ {\jota}_1')$ and $ \pi_1(X_{2,4}^1/ {\jota}_1')\to \pi_1(Z_{2,4})$ are surjective. The reason is that $\Fix({\jota}_1')\neq \emptyset$ and  $\Fix({\jota}_2')\neq \emptyset$, as these are the projections of $\Fix(\jota_1\circ \kappa)$ and $\Fix(\jota_2\circ \kappa)$, respectively. Then, one uses \cite[Chapter 2, Corollary 6.3]{Bre}. Since $\pi_1(X_{2,4}^1)=\Z_2$ and there is a double cover $Z_{3,4}\to Z_{2,4}$, we conclude that $\pi_1(Z_{2,4})=\Z_2$ and $\pi_1(Z_{3,4})=\{1\}$. 
\end{proof}

\begin{remark}
Equalities \eqref{eqn:rel-fund-1} -- \eqref{eqn:rel-fund-8} yield:
\begin{equation*}
H_1(M_{1,3})= \Z \oplus \Z_3 \oplus \Z_3 \oplus \Z_3, \quad
H_1(M_{1,6})= \Z \oplus \Z_3 , \quad
H_1(M_{2,4})= \Z \oplus \Z_2 \oplus \Z_2 \oplus \Z_2 \oplus \Z_2.
\end{equation*}
\end{remark}

\subsubsection{Cohomology algebra}

We find $H^m(M_{k,a},\R)$ for $m=1,2,3$ in Proposition \ref{prop:cohomology-M}, using the description in equation \eqref{eqn:short-mp}. We use the notation introduced there. 

\begin{proposition} \label{prop:cohomology-M}
For $m=1,2,3$, $H^m(M_{k,a})= K^m_{k,a} \oplus \delta^* C^{m-1}_{k,a}$, where
\begin{align*}
C_{k,a}^0=&C^1_{k,a}=K^1_{k,a}=0,\\
K^2_{k,a}=& \la i [dz_{1\bar{1}}], i[ dz_{3\bar{3}}], \mathrm{Re}[dz_{1\bar{2}}]  \ra, \qquad a\neq 3\\
K^2_{1,3}=& \la i [dz_{1\bar{1}}], i [dz_{3\bar{3}}], \mathrm{Re}[dz_{1\bar{2}}], \mathrm{Re}[dz_{1\bar{3}}], \mathrm{Im}[dz_{1\bar{3}}] \ra, \\
C^2_{k,a}=& \la i [dz_{2\bar{2}}], i [dz_{3\bar{3}}],\mathrm{Re}[dz_{1\bar{2}}] \ra,\quad  a \neq 3, \\ 
C^2_{1,3}=& \la i [dz_{2\bar{2}}], i [dz_{3\bar{3}}],\mathrm{Re}[dz_{1\bar{2}}],
\mathrm{Re}[dz_{2\bar{3}}], \mathrm{Im}[dz_{2\bar{3}}] \ra, \\
K^3_{1,a}=C^3_{1,a}=& \la \mathrm{Re}[dz_{123}], \mathrm{Im}[dz_{123}] \ra, \quad a=3,6, \\
K^3_{k,4}=C^3_{k,4}=& \la \mathrm{Re}[dz_{123}], \mathrm{Im}[dz_{123}], \mathrm{Re}[dz_{12\bar{3}}], \mathrm{Im}[dz_{12\bar{3}}] \ra, \quad k=2,3 .
\end{align*}
\end{proposition}
\begin{proof}
Taking into account that $b_1(M_{k,a})=1$  and $C^{0}= \R$, the short exact sequence \eqref{eqn:short-mp} shows $K^1=\{0\}$; therefore, $C^1=\{0\}$. Since $F_{k,a}$ is a holomorphic map, to obtain $K^m$ or $C^{m}$, it suffices to compute the kernel or the cokernel of the maps $F_{k,a}^* \colon H^{p,q}(T^6_k, \CC) \to H^{p,q}(T^6_k, \CC)$ with $m=p+q$ and $p\geq q$, and then take real parts to conclude. 
From now on, we abbreviate $u= e^{\frac{2\pi i}{a}}$ and we understand $F=F_{k,a}$; note that $F[z]=[uz_1,uz_2- uz_1,u^{-2}z_3]$. 

The $F^*$-invariant subspaces of $H^{2,0}(T^6_k)$ are $V^{20}_1= \CC \la[dz_{12}]  \ra$ and $V^{20}_2= \CC \la [dz_{13}],[dz_{23}] \ra$. The map $F^*$ acts as $u^2$ on the first, and on the second by 
\begin{equation*}
A^{20}_2=\begin{pmatrix}
u^{-1} & - u^{-1} \\
0 & u^{-1}
\end{pmatrix}.
\end{equation*}
The $F^*$-invariant subspaces of 
 $H^{1,1}(T^6_j)$ are $V^{11}_1= \CC \la [dz_{1\bar{2}}+ dz_{\bar{1}2}], [dz_{3\bar{3}}] \ra$, $
V^{11}_2= \CC \la [dz_{1\bar{1}}], [dz_{1\bar{2}} - dz_{\bar{1}2}], [dz_{2\bar{2}}] \ra$,
 $V^{11}_3= \CC \la [dz_{1\bar{3}}],  [dz_{2\bar{3}}] \ra$,
 and $
V^{11}_4= \overline{V^{11}_3}$.
The action of $F^*$ is trivial on $V^{11}_1$, and for $j\geq 2$, the action of $F^*$ on $V_j^{11}$ is determined by the matrix $A^{11}_j$:
\begin{equation*}
 A^{11}_2= \begin{pmatrix}
 1 & -2 & 1\\
 0 & 1 & -1 \\
 0 & 0 &  1
\end{pmatrix}, \quad 
A^{11}_3=\begin{pmatrix}
u^{3} & - u^{3} \\
0 & u^{3}
\end{pmatrix}, \quad A^{11}_4= \overline{A^{11}_3}.
\end{equation*}
Taking real parts of the kernel and cokernel of the restrictions of $F^*- \rId$ to $V_j^{p,q}$, we deduce the expressions for $K^2_{k,a}$ and $C^2_{k,a}$. To compute $K^3_{k,a}$ and $C^3_{k,a}$, we observe that $F^*$ acts trivially on $H^{3,0}(T^6_k)$ and $H^{0,3}(T^6_k)$. The $F^*$-invariant subspaces of $H^{2,1}(T^6_k)$ are $W_1=\CC \la [dz_{1\bar{2}3} + dz_{\bar{1}23}] \ra$ ,
$W_2= \CC \la [dz_{12\bar{3}}] \ra$,  
$W_3=\CC \la [dz_{13\bar{3}}],[dz_{23\bar{3}}]\ra$,
$W_4= \CC \la [dz_{1\bar{1}2}],[dz_{12\bar{2}}] \ra$, and
$W_5=\CC \la [dz_{1\bar{1}3}], [dz_{1\bar{2}3} - dz_{\bar{1}23}], [dz_{2\bar{2}3}]\ra =V^{11}_2 \wedge \la [dz_3] \ra$.
The matrix $B_j$ determines $F^*$ on $W_j$:
\begin{equation*}
B_1= \begin{pmatrix} u^{-2} \end{pmatrix}, \quad
B_2= \begin{pmatrix} u^{4} \end{pmatrix}, \quad
B_3=\begin{pmatrix}
u &  -u \\
0 & u
\end{pmatrix}, \qquad B_4=\begin{pmatrix}
u &  u \\
0 & u
\end{pmatrix}, \qquad
B_5= u^{-2} A^{1,1}_2.
\end{equation*}
The stated formulas are obtained by taking real parts of the kernel and cokernel of  $F^*- \rId$ on these subspaces.
\end{proof}

\begin{remark} \label{rem:jordan}
From the proof of Proposition \ref{prop:cohomology-M}, we deduce that the multiplicity of $1$ as an eigenvalue of $F^*\colon H^2(M_{k,a},\R)\to H^2(M_{k,a},\R) $ is $3$. This is because the Jordan form of $A_{2}^{11}$, which determines the restriction of $F^*$ on the real subspace $\la i[dz_{1\bar{1}}], i[dz_{1\bar{2}} - dz_{\bar{1}2}], i[dz_{2\bar{2}}] \ra$, is a $3\times 3$ Jordan block with eigenvalue $1$.
\end{remark}

Taking into account that $H^m(X_{k,a})=H^m(M_{k,a})^\kappa$, $H^m(Z_{k,4})=H^m(M_{k,4})^{\langle \kappa, \jota_1, \jota_2\rangle}$ and Proposition \ref{prop:cohomology-short-sequence} we obtain:
\begin{theorem}\label{theo:cohom-ex-1}
The manifolds $\widetilde{X}^1_{k,a}$, $\widetilde{Z}_{k,4}$ have $b_1=0$, in addition:
\begin{align*}
H^2(\widetilde{X}^1_{1,3})=& \la \mathrm{Re}[dz_{1\bar{3}}],\mathbf{x}_1, \mathbf{x}_2 \ra, \quad
H^2(\widetilde{X}^1_{1,6})=\la \mathbf{x}_1, \mathbf{x}_2 \ra, \\
  H^2(\widetilde{X}^1_{2,4})=& \la \mathbf{x}_1, \mathbf{x}_2, \mathbf{x}_3, \mathbf{x}_4 \ra, \quad   H^2(\widetilde{X}^1_{3,4})= \la \mathbf{x}_1, \dots,  \mathbf{x}_8 \ra, \\
  H^2(\widetilde{Z}_{2,4})=& \la \mathbf{x}_1, \mathbf{x}_2, \mathbf{x}_3, \mathbf{x}_4\ra, \quad   H^2(Z_{3,4})=\la \mathbf{x}_1, \dots,  \mathbf{x}_8 \ra, \\
H^3(\widetilde{X}^1_{1,3})=&\delta^*( \la i[dz_{2\bar{2}}], i[dz_{3\bar{3}}],\mathrm{Re}[dz_{1\bar{2}}],\mathrm{Re}[dz_{2\bar{3}}]\ra) \oplus \la \mathrm{Re}[dz_{123}]  \ra \oplus \oplus_{j=1}^{2} H^1(T^3)\otimes \mathbf{x}_j,\\
H^3(\widetilde{X}^1_{1,6})=&\delta^*(\la i[dz_{2\bar{2}}], i[dz_{3\bar{3}}],\mathrm{Re}[dz_{1\bar{2}}] \ra)  \oplus  \la \mathrm{Re}[dz_{123}] \ra \oplus \oplus_{j=1}^{2} H^1(T^3)\otimes \mathbf{x}_j, \\
H^3(\widetilde{X}^1_{2,4})=& \delta^*(\la  i[dz_{2\bar{2}}], i[dz_{3\bar{3}}],\mathrm{Re}[dz_{1\bar{2}}] \ra)  \oplus  \la \mathrm{Re}[dz_{123}] , \mathrm{Re}[dz_{12\bar{3}}]  \ra \oplus \oplus_{j=1}^{4} H^1(T^3)\otimes \mathbf{x}_j,\\
H^3(\widetilde{X}^1_{3,4})=&  \delta^*(\la  i[dz_{2\bar{2}}], i[dz_{3\bar{3}}],\mathrm{Re}[dz_{1\bar{2}}] \ra)  \oplus  \la \mathrm{Re}[dz_{123}] , \mathrm{Re}[dz_{12\bar{3}}]  \ra \oplus \oplus_{j=1}^{8} H^1(T^3)\otimes \mathbf{x}_j,\\
H^3(\widetilde{Z}_{2,4})=&  \delta^*(\la  i[dz_{2\bar{2}}], i[dz_{3\bar{3}}],\mathrm{Re}[dz_{1\bar{2}}] \ra)  \oplus  \la \mathrm{Re}[dz_{123}] , \mathrm{Re}[dz_{12\bar{3}}]  \ra \oplus 
\oplus_{k=1}^2 \la  [dx_3] \otimes \mathbf{x}_k \ra \oplus \oplus_{j=3}^4 \la  [dy_3] \otimes \mathbf{x}_k \ra, \\
H^3(\widetilde{Z}_{3,4})=&  \delta^*(\la  i[dz_{2\bar{2}}], i[dz_{3\bar{3}}],\mathrm{Re}[dz_{1\bar{2}}] \ra)  \oplus  \la \mathrm{Re}[dz_{123}] , \mathrm{Re}[dz_{12\bar{3}}]  \ra \oplus \oplus_{k=1}^4 \la  [dx_3] \otimes \mathbf{x}_k \ra \oplus \oplus_{j=5}^8 \la  [dy_3] \otimes \mathbf{x}_k \ra.
\end{align*}
Therefore, their pairs $(b_2,b_3)$ 
 are, respectively,
 $(3,11)$, $(2,10)$, $(4,17)$, $(8,29)$, $(4,9)$, and $(8,13)$.
\end{theorem}

Proposition \ref{prop:cohomology-orbifold} and Remark \ref{rem:cohomology-double-resolution} allow us to compute $H^*(\widetilde{X}_{k,a})$. The groups $H^*(M_{k,a})^{\la \iota, \kappa \ra}$ are deduced from Proposition \ref{prop:cohomology-M}. From the description of the singular locus in the proof of Lemma \ref{lem:ki}, and Remark \ref{rem:L_1}, we obtain that on $M_{1,a}$ we have
$
H^1(L_1)^-=\la [dt] \ra$, and  $H^1(L_j)^-=\la [dt], [dx_3] \ra$ when $j\neq 1.
$
Similarly, on $M_{2,4}$ we have
$ H^1(L_1)^-=H^1(L_2)^-=H^1(L_5)=\la [dt] \ra$, $H^1(L_3)^-=H^2(L_4)^-=\la [dt], [dx_3] \ra$, and $H^1(L_6)=H^1(L_7)=\la [dt], [dx_3], [dy_3] \ra$.

\begin{theorem}\label{theo:cohom-ex-2}
Let $\Sigma$ be a genus $3$ surface. We have $H^1(\widetilde{X}^2_{k,a})=0$ and
\begin{align*}
H^2(\widetilde{X}^2_{1,a})=&\la \mathbf{y}_1,\mathbf{y}_2 \ra,\\
H^3(\widetilde{X}^2_{1,a})=&\delta^*(\la  i[ dz_{2\bar{2}}], i [dz_{3\bar{3}}], \mathrm{Re}[dz_{1\bar{2}}] \ra) \oplus \la \mathrm{Re}[dz_{123}] \ra \\
&\oplus  \la [dt] \ra \otimes \mathbf{x}_1 \oplus \oplus_{j=2}^4 \la [dt],[dx_3] \ra\otimes \mathbf{x}_j
 \oplus \oplus_{p=1}^2 H^1(\Sigma \times S^1)\otimes \mathbf{y}_p ,\\
H^2(\widetilde{X}^2_{2,4})=& \la \{ (\rId + \kappa)^*\mathbf{x}_j\}_{j=5}^{7} \ra \oplus \la \{ \mathbf{y}_p\}_{p=1}^4 \ra,\\
H^3(\widetilde{X}^2_{2,4})=&\delta^*(\la  i[ dz_{2\bar{2}}], i [dz_{3\bar{3}}], \mathrm{Re}[dz_{1\bar{2}}] \ra) \oplus \la  \mathrm{Re}[dz_{123}],  \mathrm{Re}[dz_{12\bar{3}}] \ra \\
&\oplus  \oplus_{j=1,2}\la [dt] \ra \otimes \mathbf{x}_j \oplus \oplus_{j=3}^4 \la [dt],[dx_3] \ra\otimes \mathbf{x}_j
\\
&\oplus  (\rId + \kappa)^*\left(\la [dt] \ra \otimes \mathbf{x}_5   \oplus \oplus_{j=6}^7  \la [dt], [dx_3], [dy_3] \ra \otimes \mathbf{x}_j  \right) \\
&\oplus  \oplus_{p=1}^4 H^1(\Sigma \times S^1)\otimes \mathbf{y}_p .
\end{align*}
The pairs $(b_2,b_3)$ of $\widetilde{X}_{k,a}^2$ are, respectively,  $(2,25)$ if $k=1$ and $(7,46)$ if $k=2$.
\end{theorem}

The algebra structure on $H^*(\widetilde{X}_{k,a}^1)$ and $H^*(\widetilde{Z}_{k,a})$ is determined by Proposition \ref{prop:cohom-alg}. For $\widetilde{X}^2_{k,a}$, we need to use it twice. In addition, one needs to compute the Poincaré duals of the components of the singular locus, oriented by the restriction of the $\Gtwo$ structure. The rest of this section is devoted to this.

\begin{lemma} \label{lem:pduals-ex}
Consider the invariant forms on $M_{2,4}$,  $\b_1=\delta^*(\mathrm{Im}(dz_{123}))$, $\b_2=\delta^*(\mathrm{Im}(dz_{12\bar{3}}))$ and $\beta_3= dz_{1\bar{1}2\bar{2}}$. Then,
\begin{enumerate}
\item On $X_{2,4}^1$, $
PD[N_1]=PD[N_2]=[-\beta_1+\beta_2]$, and
$
 PD[N_3]=PD[N_4]=-[\beta_1 + \beta_2].
$
\item On $X_{3,4}^1$, $
PD[N_j]=\frac{1}{2}[-\beta_1+\beta_2], 
$
if $j\leq 4$ and $
PD[N_j]=-\frac{1}{2}[\beta_1+\beta_2], 
$  if $5\leq j\leq 8$.
 
\item On $Z_{2,4}$,
$
PD[N_1]=PD[N_2]=2[-\beta_1+\beta_2]$, and 
$
PD[N_3]=PD[N_4]=-2[\beta_1 + \beta_2]. 
$

\item On $Z_{3,4}$,
$
PD[N_j]=[-\beta_1+\beta_2]$, if $j\leq 4$  and 
$
PD[N_j]=- [\beta_1 + \beta_2]
$ if $5\leq j\leq 8$.
\item On $Y_{2,4}$, if $j=1,2,5,8$, and $k=3,4,6,7,9,10$, then
$
2PD[L_j]=PD[L_k]= - [\beta_3]
$.
\item On $\widetilde{Y}_{2,4}/\widetilde{\kappa}$, 
$
PD[K_1]=PD[K_2]= 2[-\beta_1+\beta_2]
$, and $
PD[K_3]=PD[K_4]= -2[\beta_1+\beta_2].
$
\end{enumerate}
\end{lemma}
\begin{proof}
 Proposition \ref{prop:cohomology-M} implies that $\beta_1$,$\beta_2$ are closed $\langle \iota,\kappa\rangle $-invariant forms on $M_{k,4}$. It is easy to check that  $dz_{1\bar{1}2\bar{2}}$ satisfies the same properties. Recall that the unit-length volume form of $M_{k,4}$ is $-\frac{i}{8}dt\wedge dz_{1\bar{1}2\bar{2}3\bar{3}}$. 

Let $X$ be one of the orbifolds $X_{2,4}^1$, $X_{3,4}^1$, $Z_{2,4}$, $Z_{3,4}$, $\widetilde{Y}_{2,4}^2/\widetilde{\kappa}$. 
Let
$[\alpha] \in H^3(X)$, then, $[\alpha]=[\delta^*\beta + \lambda_1\mathrm{Re}(dz_{123}) + \lambda_2\mathrm{Re}( dz_{12\bar{3}})]$. From
$
 [\alpha \wedge (\mu_1 \beta_1+ \mu_2 \beta_2)]= \frac{i}{2}(\lambda_1 \mu_1 - i \lambda_2 \mu_2) [dt\wedge dz_{1\bar{1}2\bar{2}3\bar{3}}],
 $
we obtain 
 $$
 \int_{X}{\alpha \wedge (\mu_1 \beta_1+ \mu_2 \beta_2)}=  (-4\lambda_1 \mu_1 + 4\lambda_2 \mu_2)\vol(X).
 $$
Let $K$ be a connected component of the singular locus at $t=0$. Since its last coordinate is real (see equations \eqref{eqn:M24-zero-1},\eqref{eqn:M24-zero-2},\eqref{eqn:Z24-2}), we have $dz_{123}|_K=dz_{12\bar{3}}|_{K}$. Similarly, the last coordinate of a connected component $K'$ at $t=1/2$ is imaginary (see equations \eqref{eqn:M24-half-1},\eqref{eqn:M24-half-2},\eqref{eqn:Z24-4}) so that  $dz_{123}|_{K'}=-dz_{12\bar{3}}|_{K'}$.
Taking into account that the unit-length volume form of any connected component is the restriction of $\varphi$, which coincides with the restriction of $\mathrm{Re}(dz_{123})$, we obtain
 $
 \int_{K}{\alpha}=(\lambda_1+\lambda_2)\vol(K), $ and $
  \int_{K'}{\alpha}=(\lambda_1-\lambda_2)\vol(K') .
 $
 Hence 
 $$PD[K]=\dfrac{\vol(K)}{4\vol(X)}(-\b_1+ \b_2),\qquad PD[K']=\dfrac{-\vol(K')}{4\vol(X)}(\b_1+ \b_2).
 $$
 The conclusion follows from the values of the volumes. For the orbifolds, we have $1=\vol(M_{2,4})=2\vol(X_{2,4}^1)=4 \vol(\widetilde{Y}_{2,4}/\tilde{\kappa})=8\vol(Z_{2,4})$, $2=\vol(M_{3,4})=2\vol(X_{3,4}^1)=8\vol(Z_{3,4})$. One can check that all connected components of the singular locus have the same volume: $2$ for those on $X_{2,4}^1$, $X_{3,4}^1$ and $\widetilde{Y}_{2,4}/\tilde{\kappa}$, and $1$ for those on $Z_{2,4}$ and $Z_{3,4}$.

We finally let $[\alpha] \in H^3(Y_{2,4})$. Then,
$[\alpha]=[\alpha_1 + \delta^*\alpha_2]$ with $[\alpha_1]\in K_{2,4}^3$ and $\alpha_2 = \lambda_1 (dz_{1\bar{2}} + dz_{\bar{1}2}) + \lambda_2 idz_{2\bar{2}}+ \lambda_3 idz_{3\bar{3}}$.
Using that $[\alpha \wedge \b_3]= [i\l_3 dt \wedge dz_{1\bar{1}2\bar{2}3\bar{3}}]$, $\vol(Y_{2,4})=1/2$, $[\alpha]|_{L_j}=i\lambda_3 \delta^* [dz_{3\bar{3}}]|_{L_j}$ and $\varphi|_{L_j}=\frac{i}{2}dt\wedge dz_{3\bar{3}}|_{L_j}$ (which follows from the proof of Lemma \ref{lem:ki}), we obtain
$$
\int_{Y_{2,4}}  \alpha \wedge \b_3 =-4\lambda_3, \qquad \int_{L_j}{\alpha} = 2\lambda_3\vol(L_j).
$$
Hence, $PD[L_j]=-\frac{1}{2}\vol(L_j)\b_3$ and the statement follows from
$\vol(L_j)=1$ if $j=1,2,5,8$, and $\vol(L_k)=2$ if $k=3,4,6,7,9,10$. 
\end{proof}

Similar computations allow us to prove the next result. For the case of $X_{1,a}$ and $\widetilde{Y}_{1,4}/\tilde{\kappa}$, one takes into account that $\vol(X_{1,a})=2\vol(\widetilde{Y}_{1,4}/\tilde{\kappa})=3\sqrt{3}/16$, and $\vol(N_j)=\vol(K_j)=3$ for $j=1,2$. On $Y_{1,a}$, one uses that $\vol(Y_{1,4})=3\sqrt{3}/16$, and $3\vol(L_1)=\vol(L_j)=3\sqrt{3}/2$ for $j=2,3,4$.

\begin{lemma} \label{lem:pduals-ex2}
On $X^1_{1,a}$, we have $PD[N_1]= PD[N_2]= -4/\sqrt{3}  \delta^*(\mathrm{Im}[dz_{123}])$, on $Y_{1,a}$, $3 PD[L_1]=PD[L_j]=-2[dz_{1\bar{1}2\bar{2}}]$, $j=2,3,4$, and on $\widetilde{Y}_{1,a}/\tilde{\kappa}$, $PD[K_1]= PD[K_2]= -8/\sqrt{3}  \delta^*(\mathrm{Im}[dz_{123}])$.
\end{lemma}

\subsubsection{Properties (P2) and (P3)} \label{subsec:properties}

We first show that the closed $\Gtwo$ structures $\widetilde{\varphi}_t$ satisfy Property (P2) when $t\to 0$. The first identity follows from directly from equation \eqref{eqn:integral of pont}, applying it twice for $\widetilde{X}^2_{k,a}$. Of course, one needs to take into account that $M_{k,a}$ and $M_{k,4}/\la \jota_1,\jota_2\ra$
are finitely covered by a parallelizable manifold (see Remark \ref{rem:nilmanifolds}). Hence, $p_1(M_{k,a})=p_1(M_{k,4}/\la \jota_1,\jota_2\ra)=0$.  Proposition \ref{prop:pont} and Lemmas \ref{lem:pduals-ex}, \ref{lem:pduals-ex2} allow us to compute the first Pontryagin class: 

\begin{corollary} \label{cor:pont-examples}
Let $\widetilde{X}$ be one of $\widetilde{X}_{k,a}^j$, or $\widetilde{Z}_{k,4}$. Then, 
$$
\int_{\widetilde{X}}{p_1(\widetilde{X})\wedge \tilde{\varphi}_t}<0.
$$
In addition,
    \begin{align*}
 p_1(\widetilde{X}_{1,3}^1)=&p_1(\widetilde{X}_{1,6}^1)=8\sqrt{3} \delta^*(\mathrm{Im}[dz_{123}]), \\
  p_1(\widetilde{X}_{1,3}^2)=&p_1(\widetilde{X}_{1,6}^2)=14 [dz_{1\bar{1}2\bar{2}}] + 16\sqrt{3} \delta^*(\mathrm{Im}[dz_{123}])-8[\omega_1]\otimes \mathbf{y}_1 -8[\omega_2]\otimes \mathbf{y}_2  , \\
    p_1(\widetilde{X}_{2,4}^1)=&p_1(\widetilde{X}_{3,4}^1)= 12\, \delta^*(\mathrm{Im}[dz_{123}]), \\
    p_1(\widetilde{Z}_{2,4}^1)=&p_1(\widetilde{Z}_{3,4}^1)= 24\, \delta^*(\mathrm{Im}[dz_{123}]),\\
     p_1(\widetilde{X}_{2,4}^2)=& 24 [dz_{1\bar{1}2\bar{2}}]+ 24 \, \delta^*(\mathrm{Im}[dz_{123}]) -8\sum_{p=1}^4 [\omega_p]\otimes \mathbf{y}_p. 
    \end{align*}
\end{corollary} 

 We now verify the second identity in equation \eqref{eqn:T2}  for $\widetilde{\varphi}_t$:

\begin{proposition}\label{prop:square-of-a-2-form}
The resolutions $\rho \colon \widetilde{X}^j_{k,a} \to {X}^j_{k,a}$ and $\rho \colon \widetilde{Z}_{k,4} \to {Z}_{k,4}$ satisfy:
$$
\int{\alpha^2\wedge \widetilde{\varphi}_t  }<0 , \quad \mbox{for every }[\alpha] \in H^2.
$$
\end{proposition}
\begin{proof}
The argument is similar for resolutions done in a single step, as $H^2$ is spanned by the Thom classes of the connected components of the exceptional divisors, except for $\widetilde{X}^1_{1,3}$, where an additional class appears. We therefore focus on this example. Similarly, $\widetilde{X}_{2,4}^2$ is the most involved among the rest.

On $\widetilde{X}_{1,3}^1$, 
we let $[\alpha]=\lambda_1 \mathbf{x}_1  + \lambda_2 \mathbf{x}_2 + \lambda_3 \mathrm{Re}[dz_{1\bar{3}}]$, then $[\alpha]^2=-2\lambda_1^2 \mathrm{Th}[N_1] -2 \lambda_2^2 \mathrm{Th}[N_2] + \frac{1}{2} \lambda_3^2 [dz_{1\bar{1}3\bar{3}}] + 2\l_1\l_3\mathrm{Re}[dz_{1\bar{3}}]\otimes \mathbf{x}_1 + 2\l_2 \l_3 \mathrm{Re}[dz_{1\bar{3}}]\otimes \mathbf{x}_2$. Set $y_j \in \R$ such that $[\theta\wedge \mathrm{Re}(dz_{1\bar{3}})]|_{N_j}= y_j [\varphi|_{N_j}]$. Then,
\begin{align*}
[\rho^*(\varphi)] \cdot [\a^2]
=& -2\lambda_1^2 \rho^*(\varphi \wedge \mathrm{Th}[N_1]) -2\lambda_2^2 \rho^*( \varphi \wedge \mathrm{Th}[N_2])
+ \frac{i}{4} \lambda_3^2 \rho^*(dt\wedge dz_{1\bar{1}2\bar{2}3\bar{3}}), \\
[\theta] \otimes \mathbf{x}_j \cdot [ \alpha^2]=& -4 y_j \l_j\l_3 \rho^*(\varphi \wedge \mathrm{Th}[N_j]).
\end{align*}
Taking into account that the unit-length volume form of $X_{1,3}^1$ is $-\frac{i}{8}dt\wedge dz_{1\bar{1}2\bar{2}3\bar{3}}$, that $\varphi$ calibrates $N_j$, and the values of the volumes of  $N_1$, $N_2$, $X_{1,3}^1$ we get:
\begin{align*}
\int_{\widetilde{X}_{1,3}^1}{\alpha^2 \wedge \widetilde{\varphi}_t }=&-6\lambda_1^2-6\lambda_2^2- \frac{3\sqrt{3}}{8} \lambda_3^2 + 12\pi t^2 y_1 \l_3 \lambda_1  + 12 \pi t^2 y_2 \l_2 \\
=&  
- 6 \left(\lambda_1 - {\pi t^2 y_1 \lambda_3}\right)^2 -6 \left(\lambda_2 - {\pi t^2 y_2 \lambda_3}\right)^2 
- \left(\frac{3 \sqrt{3}}{8} - 6\pi^2(y_1^2 +y_2^2) t^4 \right) \lambda_3^2,
\end{align*}
which is negative for all values of $(\lambda_1,\lambda_2,\lambda_3)\neq (0,0,0)$ when $t\to 0$.

According to Lemma \ref{lem:sing-2}, and the discussion at the end of section \ref{subsec:pre-resol}, on the component $K_p=\Sigma_p \times S^1$ of the singular locus of $\widetilde{Y}_{2,4}^2/\tilde{\kappa}$ we have $[\theta \wedge \omega_p] = y_p [\varphi|_{K_p}]$ for some $y_p>0$.
Let $[\alpha]= \sum_{j=5}^7\lambda_j (\mathrm{Id}+ \kappa^*)\mathbf{x}_j  + \sum_{p=1}^4{\mu_p\mathbf{y}_p }$, and assume $(\lambda_5,\lambda_6,\lambda_7,\mu_1,\mu_2,\mu_3,\mu_4)\neq (0,\dots,0)$. Since $\Sigma_p$ has genus $3$, and $\mathrm{Th}[L_j]=\mathrm{Th}[L_{j+3}]$, $j=5,6,7$.
\begin{equation} \label{eqn:X224-square}
[\alpha]^2= -4\sum_{j=5}^7\lambda_j^2 (\bar\rho_1 \circ \rho_2)^*\mathrm{Th}[L_j]  -2 \sum_{p=1}^4{ \mu_p^2 (\rho_2^*\mathrm{Th}[K_p] + 4 \omega_p\otimes \mathbf{y}_p) }.
\end{equation}
In this case, 
$
[\widetilde{\varphi}_{t_1,t_2}]= [(\bar\rho_1 \circ \rho_2)^*(\varphi)] - \pi t_1^2 \sum_{j=5}^7  (\rId + \kappa^*)(\theta_j^1\otimes\mathbf{x}_j) - \pi t_2^2 \sum_{p=1}^4 \theta_p^2\otimes \mathbf{y}_p 
$. 
We have 
$ [\theta_p^2\otimes \mathbf{y}_p]\cdot[\alpha^2 ]=16\mu_p^2 y_p \rho_2^*([\varphi] \wedge \mathrm{Th}[K_p])$, and Remark \ref{rem:cohomology-double-resolution} implies $ (\rId + \kappa)^*[\theta_j^1]\otimes \mathbf{x}_j \cdot[\alpha^2 ]= 0 $ . Hence,
\begin{align*}
    \int_{\widetilde{X}_{2,4}^2}{\alpha^2 \wedge \widetilde{\varphi}_{t_1,t_2}}=- \sum_{j=5}^7 4 \lambda_j^2\vol(L_j) - \sum_{p=1}^4  \mu_p^2(2+16 y_p \pi t_2^2)\vol(K_p) <0.
\end{align*}
\end{proof}

We finally argue that property (P3) holds. From the first three paragraphs of the proof of \cite[Proposition 3.1]{Baraglia}, one deduces the following result:
 \begin{lemma} \cite{Baraglia}
 Let $M$ be a compact $7$-dimensional manifold with finite fundamental group, and let $\widetilde{M}$ be its universal cover. 
 If $\pi \colon M \to B$ is a locally trivial fibration onto a $3$-dimensional base, then  $b_3(\widetilde{M})\in \{0,1\}$.
 \end{lemma}
 From the inequality $b_3(\widetilde{M})\geq b_3(M)$, Propositions \ref{prop:fund-group-1}, \ref{prop:fund-group-2}, and Theorems \ref{theo:cohom-ex-1}, \ref{theo:cohom-ex-2}, we deduce that the constructed manifolds satisfy (P3).

\subsubsection{Formality} \label{sec:examples-formality}
We start by proving that 
 $\widetilde{X}_{k,a}^j$ and $Z_{k,4}$, are the resolution of orbifolds of the form $M/\ZZ_2$, where $M$ is a non-formal manifold. 

\begin{theorem}\label{prop:non-formal-non-holonomy}
The closed $\Gtwo$ manifolds $M_{k,a}$, $M_{k,4}/\langle \jota_1,\jota_2\rangle$ and $\widetilde{Y}_{k,a}$ are not formal, and have $b_1=1$. Therefore, they do not admit a metric with holonomy contained in $\Gtwo$.
\end{theorem}
\begin{proof}
Proposition \ref{prop:cohomology-M} shows that $b_1(M_{k,a})=1$. The same holds for $M_{k,4}/\langle \jota_1,\jota_2\rangle$ because it is the mapping torus of $F_{k,4}\colon T^6/\langle \zeta_1, \zeta_2 \rangle \to T^6/\langle \zeta_1, \zeta_2 \rangle$.
In addition, $b_1(Y_{k,a})=b_1(M_{k,a}/\iota \circ \kappa)=1$. To conclude non-formality, we use Theorem \ref{theo:BFM-non-formality} (\cite[Corollary 16]{BFM}). As we have already shown in Proposition \ref{prop:cohomology-M} and Remark \ref{rem:jordan},  $(F_{k,a})^*$ has the eigenvalue $\lambda=1$ with multiplicity $3$ on $H^2(T^6_{j,a})$. This shows that $M_{k,a}$ is not formal. The manifold $M_{k,4}/\la \jota_1,\jota_2 \ra$ is not formal because the subspace where $F_{k,a}^*$ has eigenvalue $1$ with multiplicity $3$ in Remark \ref{rem:jordan} is
 invariant under the action of $\la \zeta_1,\zeta_2 \ra$.  
 
By Remark \ref{rem:K3}, $\widetilde{Y}_{k,a}$ is the mapping torus of $\widetilde{F}_{k,a} \colon S_k \times \C/\Gamma_k \to S_k \times \C/\Gamma_k$, $\widetilde{F}_{k,a}=(\widetilde{f}_{k,a}, \mathrm{r}_a^{-2})$.   In addition, $H^2(S_k \times \C/\Gamma_k)= H^2( \C/\Gamma_k \times \C/\Gamma_k)^{\Z_2}\oplus \la \mathbf{y}_1,\dots \mathbf{y}_{16} \ra  \oplus H^2(\C/\Gamma_k)$,  and $\widetilde{F}_{k,a}$ preserves all the summands of this decomposition, restricting to $\widetilde{f}_{k,a}$ on $H^2( \C/\Gamma_k \times \C/\Gamma_k)$. Using the computations in the proof of Proposition \ref{prop:cohomology-M}, we deduce that the restriction of $\widetilde{f}_{k,a}$ to $  \la i dz_{1\bar{1}}, i(dz_{1\bar{2}} - dz_{\bar{1}2}), idz_{2\bar{2}} \ra$ is 
\begin{equation}
 A^{11}_2= \begin{pmatrix}
 1 & -2 & 1\\
 0 & 1 & -1 \\
 0 & 0 &  1
\end{pmatrix}.
\end{equation}
This shows that algebraic order of $1$ as an eigenvalue of $\widetilde{F}_{k,a}$ is greater or equal than $3$. Again, \cite[Corollary 16]{BFM} ensures $\widetilde{Y}_{k,a}$ is non-formal. As shown in \cite[Corollary 5.4]{FKLS}, manifolds with holonomy contained in $\Gtwo$ and positive first Betti number are formal. Therefore, these manifolds do not admit torsion-free $\Gtwo$ structures.
\end{proof}

The remainder of this section is devoted to prove that all the manifolds constructed are formal, except for $\widetilde{Z}_{2,4}$ and $\widetilde{Z}_{3,4}$. We start by showing that $\tilde{X}_{1,3}^1, \widetilde{X}^1_{1,6},\widetilde{X}^2_{1,3},\widetilde{X}^2_{1,6}$ are formal using \cite[Theorem 1.14]{CN}.

\begin{proposition} \label{prop:formality-1}
The manifolds $\widetilde{X}^j_{1,a}$ with $j=1,2$ and $a=3,6$ are formal.
\end{proposition}
\begin{proof}
Let $j\in \{1,2\}$, $a=3,6$, then $\widetilde{X}^j_{1,a}$ satisfies $b_1=0$ and $b_2 \leq 3$. In addition, Proposition \ref{prop:square-of-a-2-form} implies that
 $(\cdot) \wedge [\widetilde{\varphi}_t]\colon H^2(\widetilde{X}^j_{1,a}) \to H^5 (\widetilde{X}^j_{1,a})$ is an isomorphism. Hence, $\widetilde{X}^j_{1,a}$ is formal by Proposition \ref{prop:CN-formality} (\cite[Theorem 1.14]{CN}).
\end{proof}

We prove that $\widetilde{X}^1_{3,4}$ and $\widetilde{X}^2_{2,4}$ are formal by computing the Bianchi-Massey tensor. Since nonzero-degree maps preserve formality  (see \cite[Theorem 1]{MSZ}), this will imply that $\widetilde{X}^1_{2,4}$ is formal. However, the formality of $\widetilde{X}^1_{2,4}$  can also be deduced from a direct computation, analogous to the one carried out for $\widetilde{X}^1_{3,4}$. In fact, it is simpler in this case because the second Betti number of $\widetilde{X}^1_{2,4}$ is smaller.

We follow the notations introduced in section  \ref{subsec:formality}. To compute $E^4(\widetilde{X}_{k,a}^j)$ we take into account Theorems \ref{theo:cohom-ex-1} and \ref{theo:cohom-ex-2} for the formulas of $H^2(\widetilde{X}_{k,a}^j)$, and Proposition \ref{prop:cohom-alg}, Remark \ref{rem:cohomology-double-resolution} and Lemma \ref{lem:pduals-ex} for the product formulas. A basis for $E^4(\widetilde{X}_{3,4}^1)$ is:
\begin{equation*} \label{eqn:x34-1-kernel}
n_{ij}=\mathbf{x}_i\odot\mathbf{x}_j,\quad 1\leq i <j \leq 8, \qquad m_{1j}=\mathbf{x}_1^{\odot 2}-\mathbf{x}_j^{\odot 2}, \quad 2 \leq j \leq 4, \qquad
m_{5j}=\mathbf{x}_5^{\odot 2}-\mathbf{x}_j^{\odot 2}, \quad 6\leq j \leq 8.
\end{equation*}
We define $N=\langle n_{ij}\rangle$, $M_1=\langle \{m_{1j}\}_{j=2}^4\rangle$, and $M_{2}=\langle \{ m_{5j} \}_{j=6}^8 \rangle$.
To obtain a basis of $E^4(\widetilde{X}_{2,4}^2)$  we denote $\mathbf{y}_p=(\mathrm{Id}+\kappa)^*(\mathbf{x}_j)$ for $p=5,6,7$. Equation \eqref{eqn:X224-square} and Lemma \ref{lem:pduals-ex} show that the square of a non-zero element of the form $\sum_{p=5}^8 \lambda_p \mathbf{y}_p$ is linearly independent to the subspace generated by $\{ \mathbf{y}_p^2, \, p=1,2,3,4\}$, which has dimension $4$. A basis of $E^4(\widetilde{X}_{2,4}^2)$ is: 
\begin{align}\label{eqn:x24-2-kernel-1}
n_{pq}=\mathbf{y}_p \odot \mathbf{y}_q, \quad 1 \leq p<q \leq 7, \quad  \qquad  m_{5p}=2\mathbf{y}_5^{\odot 2}-\mathbf{y}_p^{\odot 2}, \quad p=6,7.
\end{align}
We define $N=\langle n_{pq} \rangle$, $M=\langle m_{56}, m_{57} \rangle$.

\begin{lemma}\label{lem:comb-sym}
For $(j,k)\in \{(1,3),(2,2)\}$, consider the kernel $\mathcal{B}^8(X_{k,4}^j)$ of the full symmetrization tensor $\mathfrak{S} \colon \mathrm{Symm}^2(E^4(\widetilde{X}_{k,4}^j)) \to \mathrm{Symm}^4 (H^2(\widetilde{X}_{k,4}^j))$, 
\begin{enumerate}
\item If $x \in \mathcal{B}^8(X_{3,4}^1)$, then $x$ is congruent to
\begin{align*} 
&\lambda_{23} m_{12}\odot m_{13}+ \lambda_{24} m_{12}\odot m_{14} - (\lambda_{23} + \lambda_{24} )  m_{13}\odot m_{14}\\
&+ \mu_{67} m_{56}\odot m_{57}+ \mu_{68} m_{56}\odot m_{58} - (\mu_{67} + \mu_{68} )  m_{57}\odot m_{58}
\end{align*}
modulo $N \odot (N\oplus M_1 \oplus M_2)\oplus M_1\odot M_2$.
\item If $x\in \mathcal{B}^8(X_{2,4}^2)$, then 
$x \equiv 0   \mod N \odot (N\oplus M) $.
\end{enumerate}
\end{lemma}
\begin{proof}
On $X_{3,4}^1$, the image of $N \odot (N\oplus M_1 \oplus M_2)\oplus M_1\odot M_2$ under $\mathfrak{S}$ is contained in the set of degree-$4$ homogeneous polynomials on the variables
$\mathbf{x}_1,\dots, \mathbf{x}_8$, with degree at most $3$ in each variable.
In adittion, if $i=1$ and $j,k\in \{2,3,4\}$ or $i=5$ and $j,k\in \{6,7,8\}$ we have:
$$
\mathfrak{S}(m_{ij}\odot m_{ik})= \mathbf{x}_i^4 - \mathbf{x}_i^2\mathbf{x}_k^2 - \mathbf{x}_i^2\mathbf{x}_j^2 + \mathbf{x}_j^2\mathbf{x}_k^2.
$$
In particular, the image of the subspace  $\langle \{ m_{ij}\odot m_{ij}\}_{i=1,j=2,3,4}\cup \{ m_{ij}\odot m_{ij}\}_{i=5,j=6,7,8}\rangle$ is $6$-dimensional and linearly independent to the image of $N \odot (N\oplus M_1 \oplus M_2)\oplus M_1\odot M_2 \oplus \langle \{ m_{ij}\odot m_{ik}, \,i=1,\,  j\neq k \in \{2,3,4\} \mbox{ or } i=5, j\neq k \in \{ 6,7,8\}  \}\rangle $ by $\mathfrak{S}$, because the polynomials in that set do not contain any non-zero monomial proportional to  $ \mathbf{x}_{j}^4$ with $j=2,3,4,6,7,8$. Hence, if $x\in \mathcal{B}^8(X_{3,4}^1)$ then $x$ is congruent to
$$\lambda_{23} m_{12}\odot m_{13}+ \lambda_{24} m_{12}\odot m_{14} + \lambda_{34}  m_{13}\odot m_{14}\\
+ \mu_{67} m_{56}\odot m_{57}+ \mu_{68} m_{56}\odot m_{58} + \mu_{78}  m_{57}\odot m_{58}
$$
modulo $N \odot (N\oplus M_1 \oplus M_2)\oplus M_1\odot M_2$. The coefficient of $\mathbf{x}_1^4$ in $\mathfrak{S}(x)$ is $\lambda_{23}+\lambda_{24}+\lambda_{34}$ and that of $\mathbf{x}_5^4$ is $ \mu_{67} + \mu_{68}  + \mu_{78}$. Since $\mathbf{x}_1^4,\mathbf{x}_5^4 \not\in \mathfrak{S}(N \odot (N\oplus M_1 \oplus M_2)\oplus M_1\odot M_2)$, those coefficients should be $0$, as claimed.

Similarly, $\mathfrak{S}$ restricted to $N\odot N \oplus N \odot M$ is contained in the set of degree-$4$ homogeneous polynomials in which each $\mathbf{y}_i$ appears with degree at most $3$. In addition, if $6\leq p\leq q \leq 7$ we have
$$
\mathfrak{S}(m_{5p}\odot m_{5q})= 4\mathbf{y}_5^4 - 2\mathbf{y}_5^2\mathbf{y}_q^2 -2 \mathbf{y}_5^2\mathbf{y}_p^2 + \mathbf{y}_p^2\mathbf{y}_q^2.
$$
Hence, $\mathfrak{S}(\la \{ m_{5p}\odot m_{5q})\}_{6\leq p\leq q \leq 7} \ra )$ is $3$-dimensional and linearly independent to $\mathfrak{S}(N\odot N \oplus N \odot M)$. The conclusion follows from this. 
\end{proof}

Lemma \ref{lem:comb-sym} allows us to prove that the Bianchi-Massey tensor vanishes on $\widetilde{X}_{2,4}^2$. 
\begin{proposition}  \label{prop:formality-2}
$\widetilde{X}_{2,4}^2$ is formal.
\end{proposition}
\begin{proof}
Notations as in diagram \eqref{diag:resol-2}. If $1 \leq p \leq 4$, we let $\tau_p$ a Thom forms of the connected component of the exceptional divisor in the resolution $\rho_2 \colon \widetilde{X}_{k,a}\to \widetilde{Y}_{k,a}/\tilde{\kappa}$ determined by $K_p$. If $5 \leq p \leq 7$, then we pick a
Thom form $\tau_p'\in \Omega^2(\widetilde{Y}_{2,4})$ of the connected component of the exceptional divisor of $\rho_1 \colon \widetilde{Y}_{2,4} \to Y_{2,4}$ determined by $L_p$. Since the support of $(\mathrm{Id}+\tilde{\kappa})^*(\tau_p')$ is disjoint from $\cup_{q=1}^4K_q$, the form $\tau_p=\rho_2^*(\mathrm{Id}+\tilde{\kappa})^*(\tau_p')$, is smooth.  Of course, $\tau_p\wedge \tau_q=0$ if $1\leq p <q \leq 8$.

We define $\alpha \colon H^2(\widetilde{X}_{2,4}^2) \to \Omega^2(\widetilde{X}_{2,4}^2)$,
$
\alpha(\mathbf{y}_p)=\tau_p,
$
$
1 \leq p \leq 7 
$.
Since $\alpha^2(n_{pq})=\tau_p \wedge \tau_q=0$, there is a homomorphism  $\gamma \colon E^4(\widetilde{X}_{2,4}^2) \to \Omega^3(\widetilde{X}_{2,4}^2)$ with $d\gamma(e)=\alpha^2(e)$ such that $\gamma(n_{pq})=0$ for every $p<q$. If  $x \in \mathcal{B}^8(\widetilde{X}_{2,4}^2)$, then $x \equiv 0 \mod N\odot(N\oplus M)$ and $\mathcal{F}_{\widetilde{X}_{2,4}^2}(x)=0$ because
$\gamma(n_{jk})\alpha^2(n_{pq})=\gamma(n_{jk})\alpha^2(m_{5j})=\gamma(m_{pq})\alpha^2(n_{ij})=0$. Proposition \ref{prop:CN-formality} (\cite[Theorem 1.14]{CN}) shows that $\tilde{X}_{2,4}^2$ is formal.
\end{proof}

To conclude formality for $\widetilde{X}_{3,4}^1$, we need to study the configuration of the singular locus. Fix two connected components $N_j$ and $N_k$ that lie in the same level set $t_0=0,1/2$, we claim that we can find a cobordism $C$ contained in $\{t_0\}\times T^6_{3}$ such that $(C_{jk}\cup \kappa(C_{jk}))\cap N_\ell=\emptyset$ when $\ell \notin \{j,k\}$. For instance, if $t=1/2$, given two different connected components determined by $(n,\e_3)$ and $(m,\delta_3)$, the cobordism $C$ is the image of
$
[0,1]\times N_j \to M_{3,4}^1$, 
$$
[t,1/2,n-2iy_2,x_2+iy_2,\e_3 + iy_3]\longmapsto [1/2,[ h_1(t)-2iy_2, x_2+iy_2,h_2(t)+ iy_3]],
$$
where $h_1(t)=(1-t)n+ tm$, $h_2(t)=(1-t)\e_3 + t\delta_3$, and we oriented $C$ so that $\partial C=N_j\sqcup -N_k$.  $C$ does not intersect the remaining connected components: if $n\neq m$, 
then $h_1(t)$ is an integer when $t=0,1$, and $h_2(0)=\e_3$, $h_4(1)=\delta_3$; otherwise, $(n,(1-t)\e_3 + t\delta_3)\notin \{(1-n,0),(1-n,1/2)\}$.
A similar argument shows that $\kappa(C_{jk})=[1/2,[-h_1(t)-2iy_2, h_1(t)+x_2-3iy_2, -h_2(t)+ iy_3]]$ is  disjoint from the remaining connected components of the singular locus.

There are tubular neighborhoods $U_{jk}$ of $C_{jk}$, and $V_\ell$ of $N_\ell$ such that if $\ell \neq j,k$, then $U_{jk}\cup \kappa^*(U_{jk})\cap V_{\ell}=\emptyset$. Note also that $V_{\ell}$ is $\kappa$-invariant as $N_\ell\subset \Fix(\kappa)$ and $\kappa$ is an isometry. We fix $\tau_j$ a Thom form for $Q_j=\rho^{-1}(N_j)$ supported inside a tubular neighborhood of $Q_j$ contained in $\rho^{-1}(V_j)$. According to Remark \ref{rmk:sqtau-trivial}, we can further assume that $\tau_j^2$ vanishes around a smaller tubular neighborhood of $Q_j$. In addition, equation \eqref{eqn:square-thom} shows that $-\frac{1}{2} \rho_*(\tau_j^2)$ is a Thom form of $N_j$ on $X_{3,4}^1$.  This allows us to find the maps $\alpha$ and $\gamma$.
\begin{proposition} \label{prop:gamma-and-alpha-341}
Define $\alpha \colon H^2(\widetilde{X}_{3,4}^1) \to \Omega^2(\widetilde{X}_{3,4}^1)$,  $\alpha(\mathbf{x}_p)=\tau_p$. There is 
 $\gamma \colon E^4(\widetilde{X}_{3,4}^1) \to \Omega^3(\widetilde{X}_{3,4}^1)$ such that $d\gamma(e)=\alpha^2(e)$ and
\begin{enumerate}
\item If $p \neq q$, $\alpha^2(n_{pq})=0$ and $\gamma(n_{pq})=0$.
\item If $a=1$ and $p=2,3,4$ or if $a=5$ and $p=6,7,8$, then for every $q\neq a$, $q \neq p$ we have
$\gamma(m_{ap})\wedge \alpha(\mathbf{x}_q)=0$.
\end{enumerate}
\end{proposition}
\begin{proof}
Since $\tau_p \wedge \tau_q=0$ for $p\neq q$, we set $\gamma (n_{pq})=0$. 
We now define $\gamma(m_{ap})$ for $a=1$ and $p=2,3,4$ or $a=5$ and $p=6,7,8$. By Lemma \ref{lem:intersections}, given $N_{a'}\neq N_p$ a component that is cobordant to $N_p$, then we can find $\beta'_{a'p}$ supported on $U_{a'p} \cup V_{a'} \cup V_p$ such that $d\beta'_{a'p}=-\frac{1}{2} \rho_*(\tau_{a'}^2)+ \frac{1}{2} \rho_*(\tau_p^2)$. Then, $-(\beta_{a'p}' + \kappa^*\beta_{a'p}')$ is a $\kappa$-invariant primitive of $\rho_*(\tau_{a'}^2)- \rho_*(\tau_p^2)$ and it is supported on
$U_{a'p} \cup \kappa(U_{a'p}) \cup V_{a'} \cup V_p$. Indeed, by Lemma \ref{lem: smooth extension} (\cite[Lemma 5]{LMM-2}), we can assume that $\beta_{a'p}= -\rho^*(\beta_{a'p}' + \kappa^*\beta_{a'p}')$ is smooth. In particular, $\rho^*(\beta_{a'p})\wedge \tau_q=0$ if $q \neq a',p$. The choice $\gamma(m_{ap})=\beta_{ap}$ satisfies the announced properties.
\end{proof}

\begin{remark} \label{rmk:primitives-observation}
Given two cobordant connected components $N_a'$ and $N_p$, we obtained primitives ${\beta}_{a'p}$ of $\tau_{a'}^2-\tau_p^2$ with  ${\beta}_{a'p}\wedge \tau_{q}=0$ when $q \neq a',p$ and such that $\rho_*({\beta}_{a'p})$ is well-defined.
\end{remark}

\begin{proposition}  \label{prop:formality-3}
The manifolds $\widetilde{X}^1_{3,4}$ and $\widetilde{X}^1_{2,4}$ are formal. 
\end{proposition}
\begin{proof}
Using Proposition \ref{prop:CN-formality} (\cite[Theorem 1.14]{CN}) and \cite[Theorem 1]{MSZ}, it suffices to show that 
 $\mathcal{F}_{\widetilde{X}^1_{3,4}} =0$. If $x \in \mathcal{B}(\widetilde{X}^1_{3,4})$, then $x$ is congruent to
\begin{align*}
&\lambda_{23} m_{12}\odot m_{13}+ \lambda_{24} m_{12}\odot m_{14} - (\lambda_{23} + \lambda_{24} )  m_{13}\odot m_{14}\\
&+ \mu_{67} m_{56}\odot m_{57}+ \mu_{68} m_{56}\odot m_{58} - (\mu_{67} + \mu_{68} )  m_{57}\odot m_{58}
\end{align*}
modulo $N \odot (N\oplus M_1 \oplus M_2)\oplus M_1\odot M_2$ by Lemma \ref{lem:comb-sym}. Let $\alpha$ and $\gamma$ be the maps obtained in Proposition  \ref{prop:gamma-and-alpha-341}. Using its properties, we deduce that $\mathcal{F}_{\widetilde{X}^1_{3,4}}=0$ on $N \odot (N\oplus M_1 \oplus M_2)\oplus M_1\odot M_2$. Denoting $\gamma(m_{ap})=\beta_{ap}$,  and using $\beta_{ap}\wedge \tau_q=0$ for $q\neq a,p$, we obtain:
\begin{align*}
\mathcal{F}_{\widetilde{X}^1_{3,4}}(x)=& [\lambda_{23} \b_{12}\wedge(\tau_1^2-\tau_3^2)+ \lambda_{24} \b_{12}\wedge(\tau_1^2-\tau_4^2) - (\lambda_{23} + \lambda_{24} )  \b_{13}\wedge (\tau_1^2-\tau_4^2)]\\
&+ [\mu_{67} \b_{56}\wedge(\tau_5^2-\tau_7^2)+ \mu_{68} \b_{56}\wedge(\tau_5^2-\tau_8^2) - (\mu_{67} + \mu_{68} )  \b_{57}\wedge(\tau_5^2-\tau_8^2)]\\
 =& (\lambda_{23}+\lambda_{24})[ (\b_{12}-\beta_{13})\wedge\tau_1^2]
 + (\mu_{67} + \mu_{68} )[ (\b_{56}-\b_{57}) \wedge\tau_5^2]\\
  =& (\lambda_{23}+\lambda_{24})[ (\b_{12}-\beta_{13})\wedge(\tau_1^2-\tau_{4}^2)]
 + (\mu_{67} + \mu_{68} )[ (\b_{56}-\b_{57}) \wedge(\tau_5^2-\tau_8^2)],
\end{align*}
It suffices to show that $[(\b_{12}-\beta_{13})\wedge(\tau_1^2-\tau_{4}^2)]=[ (\b_{56}-\b_{57}) \wedge(\tau_5^2-\tau_8^2)]=0$. The proofs of both identities are similar, so we conclude by verifying that
$\int_{\widetilde{X}_{3,4}^1}{(\b_{12}-\beta_{13})\wedge(\tau_1^2-\tau_{4}^2)}=0$. Similar to the proof of Proposition \ref{prop:massey-linking}, we have
\begin{align*}
\int_{\widetilde{X}_{3,4}^1}{(\b_{12}-\beta_{13})\wedge(\tau_1^2-\tau_{4}^2)}=&
8 \int_{M_{3,4}}{\rho_*\left(-\frac{1}{2}\b_{12}+\frac{1}{2}\beta_{13}\right)\wedge \rho_*\left( -\frac{1}{2}\tau_1^2+  \frac{1}{2}\tau_{4}^2\right)}\\
=&8\mathrm{lk}_{M_{3,4}}(N_3\sqcup -N_2, N_1 \sqcup -N_4),
\end{align*}
because $-\frac{1}{2}\rho_*(\tau_1^2-\tau_{4}^2)$ is a Thom form for $N_1\sqcup -N_4$, $-\frac{1}{2}d(\rho_*(\b_{12}-\beta_{13}))=-\frac{1}{2}\rho_*(\tau_3^2-\tau_2^2)$, and $-\frac{1}{2}\rho_*(\tau_3^2-\tau_2^2)$ is  a Thom form for $N_3\sqcup -N_2$. 

The form $-\beta_{23}$ constructed in Remark \ref{rmk:primitives-observation}, pushes forward to $M$, and it satisfies $d(-\beta_{23})=\tau_3^2-\tau_2^2$, and $\beta_{23}\wedge \tau_1=\beta_{23}\wedge \tau_4=0$. Hence,
$$
\mathrm{lk}_{M_{3,4}}(N_3\sqcup -N_2, N_1 \sqcup -N_4)=\frac{1}{4} \int_{M_{3,4}}\rho_*\beta_{23}\wedge \rho_*(\tau_1^2-\tau_4^2) =0. 
$$
\end{proof}

Utilizing Proposition \ref{prop:massey-linking}, we
 finally show that $\widetilde{Z}_{2,4}$ is not formal, and hence, by \cite[Theorem 1]{MSZ}, neither is $\widetilde{Z}_{3,4}$.
 First of all,
from equations \eqref{eqn:M24-half-1}, \eqref{eqn:Z24-4} we observe 
that a component determined by $\Fix(\eta_2)$ is not homologous to a component determined by $\Fix(\eta \circ \zeta_2)$ on $T^6_2$ (compare for instance, their Poincaré duals). However, we now show that these are homologous on $M_{2,4}$.

\begin{proposition}\label{prop:Z24-geom}
There is an oriented cobordism $C\subset M_{2,4}$, with boundary on the projection of $\{1/2\}\times T^6_2$ such that $[\partial C] +  [\partial(\kappa(C))]= 2[N_3^1]-2[N_4^1]$ on $H_3(\{1/2\}\times T^6_2,\Z)$.

For every $j,k=1,2$ we have $C\cap N_j^k=\{p_j^k\}$ and the intersection is negative at every $p_1^k$ and positive at every $p_2^k$.
\end{proposition}
\begin{proof}
Equation \eqref{eqn:M24-half-1} shows that $N_3^1$ is the image of the embedding
$$
(\R/\Z)^3 \to M_{2,4}, \quad [x_2,y_2,y_3]\mapsto [1/2, [-2iy_2, x_2 + iy_2, iy_3]].
$$
The pullback of $\varphi|_{N_3^1}$ is $-2dx_2\wedge dy_2 \wedge dy_3$, i.e., an oriented frame for $TN_3^1$ is $(\partial_{x_2}, 2\partial_{y_1}-\partial_{y_2}, \partial_{y_3})$. We now view $N_4^1$ in the level set $t=-1/2$;
using equations \eqref{eqn:M24-half-2} and $[1/2,[-2x_2,x_2+iy_2,-1/4+iy_3]]=[-1/2, [i2x_2, y_2+ ix_2,1/4- iy_3]]$, and performing a change of variables, we deduce that $N_4^1$ is the image of 
$$
(\R/\Z)^3 \to M_{2,4}, \quad [x_2,y_2,y_3]\mapsto [-1/2, [2iy_2, x_2 + iy_2, 1/4+ iy_3]].
$$
The pullback of $\varphi|_{N_3^1}$ is $2dx_2\wedge dy_2 \wedge dy_3$, and  an oriented frame for $TN_4^1$  is $(\partial_{x_2}, 2\partial_{y_1}+\partial_{y_2}, \partial_{y_3})$. Define the embedding $\Psi \colon [-1/2,1/2]\times N_4^1 \to M_{2,4}$,
$$
\Psi(t,[1/2,[2iy_2, x_2 + iy_2, 1/4+ iy_3]])=
[t,[2iy_2, x_2 + iy_2, h(t)+ iy_3]],
$$
where $h=0$ on $[-1/2,1/8]$ and $h=1/4$ on $[1/4,1/2]$. We orient $C=\mathrm{Im}(\Psi)$ by the frame:
$$
(\psi_*(\partial_t), \psi_*(\partial_{x_2}),\psi_*(\partial_{y_2}),\psi_*(\partial_{y_3}))=(\partial_t+ h'(t)\partial_{x_3},\partial_{x_2},2\partial_{y_1} + \partial_{y_2},\partial_{y_3}).
$$
 According to the orientations described in section \ref{sec:intersection}, $\partial C = -N_4^1 \sqcup L_1$, where $L_1=\Psi(\{1/2\}\times N_4^1)$ is oriented by $(\partial_{x_2}, 2\partial_{y_1}+ \partial_{y_2},\partial_{y_3})$. Since
 $$\kappa[t,[2iy_2,x_2+iy_2,h(t)+iy_3]]=[1-t,[2iy_2,x_2-3iy_2,-h(t) + iy_3]],  \qquad 1-t \in [1/2,3/2],
 $$
  the orientation induced by $\kappa$ in $\kappa(C)$ is $(-\partial_t- h'(t)\partial_{x_3},\partial_{x_2}, 2\partial_{y_1} - 3 \partial_{y_2},\partial_{y_3})$.  Note that $\partial \kappa(C)=\kappa(\partial C)$. Since  $\kappa|_{N_4^1}=\mathrm{Id}$, we obtain that $\partial \kappa(C)=-N_4^1\sqcup \kappa(L_1)$, where $L_2=\kappa(L_1)$ is oriented by the restriction of $(\partial_{x_2}, 2\partial_{y_1}-3 \partial_{y_2},\partial_{y_3})$ when its viewed on the level set $t=1/2$.
  This shows that  $[\partial C]+ [\partial \kappa(C)]= -2[N_4^1]+ [L_1]+[L_2]$ on $H_3(T^6_2,\Z)$. We now claim that $[L_1]+[L_2]=2[N_3^1]$ on $H_3(T^6,\Z)$. To verify this, we rewrite $T^6_2=T^2_1\times T^2_2 \times T^2_3$, where $T^2_1$, $T^2_2$ and $T_2^3$ are determined respectively by the variables $(y_1,y_2)$, $(x_2, y_3)$ and $(x_1,x_3)$. Note that $(\partial_{y_1},\partial_{y_2},\partial_{x_2}, \partial_{y_3}, \partial_{x_1}, \partial_{x_3})$ is a positive frame when $T_2^6$ is endowed with the standard orientation.
Define $A_1,A_2,A_3\subset T^2_1$ as the (oriented) image of the embeddings $e_j\colon \R/\Z \to T^2_1$, 
$$
e_1[x]=[2x,x],\quad e_2[x]=[2x,-3x],\quad e_3[x]=[2x,-x].
$$
Orienting $T_2^2$ by $(\partial_{x_2}, \partial_{y_3})$, we obtain $A_j \times T^2_2 \times \{[0,0]\}$, is $-L_j$ if $j=1,2$ and $-N_3^1$ otherwise. The claimed equality follows from $[A_1]+[A_2]-2[A_3]=0$ on $H_1(T^2_1,\ZZ)$.

Regarding the intersections, we have $C \cap N_j^k = \{ p_j^k\}$, where $p_j^k=[0,0,0,\e_j^k]$ and $\e_1^1=0$, $\e_1^2=1/2$ $\e_2^1=-1/4$, and $\e_2^2=1/4$. We now compute their sign. At the level set $t=0$, the metric restricts to the standard flat metric on $T^6_2$, and an
oriented frame for $\nu(C)|_{t=0}$ is $(\partial_{x_1}, -\partial_{y_1}+2\partial_{y_2}, \partial_{x_3})$. In addition, taking into account expressions \eqref{eqn:M24-zero-1}, \eqref{eqn:Z24-2} and the formula of $\varphi$, we obtain that an oriented frame of $TN_1^k$ is $(\partial_{x_1}+ \partial_{y_1},\partial_{x_2}- \partial_{y_2}, \partial_{x_3})$, and an oriented frame for $TN_2^k$ is $(\partial_{x_1}-\partial_{y_1},\partial_{x_2}+\partial_{y_2}, \partial_{x_3})$. Hence, an oriented frame for $\nu(N_j^k)$ is
$(-\partial_t, \partial_{x_1}+ (-1)^j\partial_{y_1},\partial_{x_2}+ (-1)^{j+1}\partial_{y_2}, \partial_{y_3})$. As explained in section \ref{sec:intersection}, the orientation on $C\cap N_j^k=\{p_j^k\}$ is determined by comparing those of $TM|_{\{p_j^k\}}$ and $\nu(C)|_{p_j^k} \oplus \nu (N_j^k)|_{p_j^k}$, where the latter is endowed with the direct sum orientation.  The statement follows from the equality:
\begin{align*}
&\partial_{x_1}\wedge ( -\partial_{y_1}+2\partial_{y_2}) \wedge \partial_{x_3} \wedge (-\partial_t)\wedge( \partial_{x_1}+ (-1)^{j} \partial_{y_1}) \wedge (\partial_{x_2}+ (-1)^{j+1}\partial_{y_2}) \wedge \partial_{y_3}\\
& = \partial_t \wedge \partial_{x_1}\wedge 2\partial_{y_2} \wedge \partial_{x_3} \wedge
( (-1)^j  \partial_{y_1}) \wedge \partial_{x_2} \wedge \partial_{y_3} =(-1)^j 2 \, \partial_t \wedge \partial_{x_1}\wedge \partial_{y_1} \wedge \partial_{x_2} \wedge \partial_{y_2} \wedge \partial_{x_3} \wedge \partial_{y_3}.
\end{align*}
\end{proof}

Using Lemma \ref{lem:intersections} we construct a primitive of a representative of $\mathrm{Th}[N_3^1]-\mathrm{Th}[N_4^1]$.

\begin{proposition} \label{prop:Z24-primitives}
For $j=3,4$, and $k=1,2$, let  $\upsilon_j^k$  be a Thom form of  $N_j^k$. There is $\beta \in \Omega^3(M_{2,4})$ such that $d\beta=\upsilon_3^1-\upsilon_4^1$, and for every $j,k=1,2$ and for every choice of Thom form $\upsilon_j^k$ of $N_{j}^k$ we have,
$$ 
\int_{M_{2,4}}{\beta \wedge \upsilon_j^k}=(-1)^j.
$$
\end{proposition}
\begin{proof}
Since $[\partial C] + [\partial \kappa(C)]=2[N_3]-2[N_4]$ on $\{1/2\}\times T^6_2$, Lemma \ref{lem:PD-and-MTorus} allows us to construct
Thom forms $\tau_{\partial C}$, $\tau_{\partial \kappa (C)}$, $\tau_{3}^1$, and $\tau_{4}^1$ of $\partial C$, $\partial \kappa(C)$, $N_3^1$ and $N_4^1$, and 
$\alpha_1\in \Omega^3(M_{2,4})$ supported on $(1/2-\e, 1/2+\e)\times T^6_2$ with $d\alpha_1= \tau_{\partial C} + \tau_{\partial \kappa (C)} - 2\tau_{3} + 2\tau_{4}$. Lemma \ref{lem:intersections} provides us with  $\beta_C\in \Omega^3(M_{2,4})$  such that $d\beta_C=\tau_{\partial C}$. 

Note that $\kappa^*(\partial_C)$ is a Thom form of $\partial \kappa(C)$ because $\kappa$ preserves the orientation of $M$, and we chose the orientation of $\kappa(C)$ so that $\kappa \colon C \to \kappa(C)$ is orientation-preserving. Hence, there is $\alpha_2\in \Omega^3(M_{2,4})$ supported near $\partial \kappa (C)$ that satisfies $d\alpha_2=\tau_{\partial \kappa(C)} -\kappa^*(\tau_{\partial C})$. We also choose $\alpha_3$ and $\alpha_4$ supported near $N_3^1$ and $N_4^1$ such that $d\alpha_j=\upsilon_j^1-\tau_j$, $j=3,4$. Therefore,
$$
d\beta=
 \upsilon_3^1 - \upsilon_4^1, \qquad \beta= \frac{1}{2} \left( \beta_C +  \kappa^*(\beta_C) -  \alpha_1 +  \alpha_2\right)  +  \alpha_3- \alpha_4.
$$

Let $j,k=1,2$, and let $\upsilon_j^k$ be a Thom form of $N_{j}^k$, then
$\beta \wedge \upsilon_j^k = \frac{1}{2}(\beta_C + \kappa^* \beta_C)\wedge \upsilon_j^k$. Lemma \ref{lem:intersections} ensures $[\beta_C \wedge \upsilon_j^k]=[\beta_C \wedge \kappa^*(\upsilon_j^k)]=\mathrm{Th}[C\cap N_j^k]$, since $\kappa^*(\upsilon_j^k)$ is also a Thom form of $N_j^k$, as $\kappa$ preserves the orientations of both $M_{2,4}$ and $N_j^k$. Using Proposition \ref{prop:Z24-geom}, and that $\kappa$ preserves the orientation of $M_{2,4}$, we obtain:
$$
\int_{M_{2,4}}{\beta \wedge \upsilon_j^k}= \frac{1}{2} \int_{M_{2,4}}{\beta_C \wedge (\upsilon_j^k + \kappa^*(\upsilon_j^k))}=(-1)^{j}.
$$
\end{proof}

\begin{proposition} \label{prop:formality-4}
Let $\rho \colon \widetilde{Z}_{2,4}\to Z_{2,4}$ be the resolution map and $Q_j= \rho^{-1}(N_j)$.
Then, $\mathrm{lk}_{M_{2,4}/\langle \jota_1,\jota_2 \rangle}(N_1\sqcup -N_2,N_3\sqcup- N_4)\neq 0$ and the triple Massey product 
$$
\langle \mathrm{Th}[Q_3]+\mathrm{Th}[Q_4], \mathrm{Th}[Q_1]+\mathrm{Th}[Q_2] , \mathrm{Th}[Q_1]-\mathrm{Th}[Q_2]\rangle,
$$
on $\widetilde{Z}_{2,4}$ does not vanish. Both $\widetilde{Z}_{2,4}$ and $\widetilde{Z}_{3,4}$ are not formal. 
\end{proposition}
\begin{proof}
To compute the linking number (see Definition \ref{def:lk}), we 
consider Thom forms $\upsilon_j\in \Omega^4(M_{2,4}/\langle \jota_1,\jota_2 \rangle)$ of $N_j$ with $j=1,2,3,4$, and we now find a primitive $\alpha \in \Omega^3(M_{2,4}/\langle \jota_1,\jota_2 \rangle)$ of $\upsilon_3-\upsilon_4$.
Let $q \colon M_{2,4}\to M_{2,4}'/\langle \jota_1,\jota_2  \rangle$ be the quotient map, then $q^*\upsilon_j$ is a $\langle \jota_1,\jota_2  \rangle$-invariant
Thom form of $N_j^1\sqcup N_j^2$, and we write 
$q^*(\upsilon_j)=\upsilon_j^1+\upsilon_j^2$, where $\upsilon_j^k$ are Thom forms for $N_j^k$. Using Proposition \ref{prop:Z24-primitives}, we  pick $\beta \in \Omega^3(M_{2,4})$ such that $d\beta=\upsilon_3^1-\upsilon_4^1$ and $\int_{M_{2,4}}{\beta \wedge \bar{\upsilon}_j^k}=(-1)^{j}$ for every $j,k=1,2$ and $\bar{\upsilon}_j^k$ Thom form of $N_j^k$.

Recall that for $j=1,2$, the involution $\jota_1$ fixes $N_j^k$, whereas for $j=3,4$, it swaps $N_j^1$ with $N_j^2$. From the $\jota_1$-invariance of $q^*\upsilon_j$ it follows $\jota_1^*(\upsilon_j^k)=\upsilon_j^k$ for $j,k=1,2$ and $\jota_1^*(\upsilon_j^1)=\upsilon_j^2$ for $j=3,4$. Using also that $\jota_1$ preserves the orientation, we deduce that $\beta^1=\beta + \jota_1^*\beta\in \Omega^3(M_{2,4})$ satisfies  
$$
d\beta^1= q^*\upsilon_3-q^*\upsilon_4,\quad \int_{M_{2,4}}{ \beta^1 \wedge \upsilon_j^k}=2
\int_{M_{2,4}}{ \beta\wedge \upsilon_j^k}=(-1)^{j}2, \qquad j,k=1,2. 
$$
Consider the $\langle \jota_1,\jota_2\rangle$-invariant form $\alpha'=\frac{1}{2} (\beta^1 + \jota_2^*\beta^1)$. Similarly, using that $\jota_2$ is orientation-preserving and fixes $q^*\upsilon_j$, we obtain
$$
d\alpha'= q^*\upsilon_3-q^*\upsilon_4,\quad \int_{M_{2,4}}{ \alpha' \wedge q^*(\upsilon_{j})}=
\int_{M_{2,4}}{ \beta^1\wedge q^*(\upsilon_{j})}=(-1)^{j}4, \qquad j=1,2. 
$$
Let $\alpha= q_*(\alpha')$. Then
$d\alpha= \upsilon_3-\upsilon_4$, and
$$
\mathrm{lk}(N_1\sqcup -N_2, N_3\sqcup -N_4)=\int_{M_{2,4}/\langle \jota_{1},\jota_2\rangle}{ \alpha \wedge (\upsilon_1-\upsilon_2)}
= \frac{1}{4} \int_{M_{2,4}}{ \alpha' \wedge q^*(\upsilon_1-\upsilon_2)}= \frac{1}{4}(-4-4)=-2.
$$
The conclusion follows from Proposition \ref{prop:massey-linking}, Lemma \ref{lem:tmp-vanish}  (\cite[Proposition 2.90]{FOT}), and  \cite[Theorem 1]{MSZ}.
\end{proof}

\begin{remark} A similar argument shows that $\widetilde{Z}_{3,4}$ has a non-vanishing triple Massey product.
\end{remark}

\subsection{Comparison with known compact holonomy $\Gtwo$ manifolds}

We now argue that, among the constructed manifolds only $\widetilde{X}_{1,6}^2$ shares the triple  $(\pi_1(\bullet),b_2(\bullet),b_3(\bullet))$ with the holonomy $\Gtwo$ manifolds found in \cites{CHNP, J1,J2,Joyce2,Joyce-Const,JK,Kovalev,KovalevLee,LMM-2,Reid}.

In Joyce's cited works, the pair $(b_2,b_3)=(3,11)$ appears in \cite[Example 8]{J2},  which match those of $\widetilde{X}_{1,3}^1$. However, his example is simply connected and $\pi_1(\widetilde{X}_{1,3}^1)=\Z_3$. Similarly, a simply connected example with $(b_2,b_3)=(4,17)$ appears in \cite[Example 11]{J2}, but the manifold constructed here with those Betti numbers, namely $\widetilde{X}_{2,4}^1$, has fundamental group $\Z_2$. As mentioned, there are two simply connected manifolds with $(b_2,b_3)=(2,10)$ in \cite[Examples 13 and 14]{J2}, but constructed in a different way from $\widetilde{X}_{1,6}^1$. The first is the quotient of $\R/\Z \times \C^3/\Lambda$, where $\Lambda$ is constructed using a $7$th root of unity, by the dihedral group $D_7$. The second 
 is the resolution of the orbit space of $\R/\Z \times (\C/\Gamma_1)^3$ under the action of a group of order $18$ that contains $D_3$. 
As for the remaining pairs of $(b_2,b_3)$, namely 
$(2,25),(4,9),(7,46),(8,13),(8,29)$, in \cite[Figure 1, Tables 12.1--12.7, sections 12.2,12.3, 12.8]{Joyce-Const} and \cite[Figure 12.3]{Joyce2} we observe that there is not any holonomy $\Gtwo$ manifold with $b_2=2$ and $20\leq b_3\leq 30$, or $b_2=4$ and $b_3 \leq 16 $, or $b_2=7$,  and $45\leq b_3 \leq 47$, $b_2=8$ and $b_3 \leq 18$, or $b_2=8$ and $24 \leq b_3 \leq 30$.

The examples found by Kovalev in \cite{Kovalev} have $0 \leq b_2 \leq 1$.  Kovalev and Lee  found manifolds with $b_2=2,8$ in \cite[Table 1]{KovalevLee}, but these have $b_3 \geq 53$. 
Their examples with $b_2=3$ have $b_3 \geq 58$, and those with $b_2=7$ have $b_3 \geq 62$ (see discussion after Proposition 6.1).
All the examples constructed by  Corti, Haskins, Nordström and Pacini \cite{CHNP} are simply connected (see Theorem 4.9), and most are $2$-connected. The positive values of $b_2$ that occur are $b_2=3,20$, so there is no overlap with our examples. The examples by Nordström in \cite{No} are $2$-connected, except for one, which has $b_2=1$.
The new example by Joyce and  Karigiannis \cite[Example 7.3]{JK} has $b_2=0$. The compact holonomy $\Gtwo$ manifolds in \cite{Reid} are simply connected and they do not have $b_2=2$ or $b_2=7$. Examples with $b_2=4$ and $b_2=8$ appear in Theorems 4.5 and 4.8. Those with $b_2=4$ have $b_3\geq 27$ and those with $b_2=8$ have $b_3 \geq 31$. Finally, the non-formal example in \cite{LMM-2} has $b_2=11$.

\end{document}